\documentclass[a4paper,11pt, twoside]{article}
\usepackage[english]{babel}
\usepackage[latin1]{inputenc}
\usepackage[T1]{fontenc}
\usepackage{amsthm,amsmath,amssymb,amsfonts,graphics,graphicx}
\usepackage{amsthm,amsmath,amssymb,amsfonts,graphics,graphicx,bbm,fancybox}
\usepackage{longtable}
\usepackage{rotating}
\usepackage{natbib}
\def\pen{\mathrm{pen}}
\def\argmin{\mathrm{argmin}}
\newcommand{\normp}[1]{\ensuremath{\vert\!\vert #1 \vert\!\vert_{\tp}}}
\newcommand{\indic}{1}
\newtheorem{Th}{Theorem}
\newtheorem{Remark}{Remark}
\newtheorem{Def}{Definition}
\newtheorem{Prop}{Proposition}
\newtheorem{Lemma}{Lemma}

\newtheorem{Cor}{Corollary}
\renewenvironment{proof}{\noindent{\bf Proof.}}{\hfill
  $\blacksquare$\par\noindent}
 \newcommand{\com}[1]{}
\newcommand{\norm}[1]{\ensuremath{\vert\!\vert #1 \vert\!\vert}}
\newcommand{\E}{\ensuremath{\mathbb{E}}}

\renewcommand{\P}{\ensuremath{\mathbb{P}}}
\newcommand{\X}{\ensuremath{\mathbb{R}}}
\newcommand{\R}{\ensuremath{\mathbb{R}}}

\newcommand{\Z}{\ensuremath{\mathbb{Z}}}

\renewcommand{\ln}{{\log}}
\renewcommand{\L}{\ensuremath{\mathbb{L}}}
\newcommand{\La}{\ensuremath{\Lambda}}
\newcommand{\Ga}{\ensuremath{\Gamma}}
\newcommand{\ga}{\ensuremath{\gamma}}
\newcommand{\al}{\ensuremath{\alpha}}
\newcommand{\la}{\ensuremath{\lambda}}
\newcommand{\si}{\ensuremath{\sigma}}
\newcommand{\e}{\ensuremath{\varepsilon}}
\newcommand{\p}{\ensuremath{\varphi}} 
\newcommand{\tp}{\ensuremath{{\tilde{\varphi}}}}
\newcommand{\be}{\ensuremath{\beta}}
\newcommand{\tb}{\ensuremath{\tilde{\beta}}}
\newcommand{\hb}{\ensuremath{\hat{\beta}}}
\newcommand{\jk}{\ensuremath{{j,k}}}

\renewcommand{\com}[1]{{\it #1}}
\numberwithin{equation}{section}

\title{Adaptive thresholding estimation of a Poisson intensity with infinite support}
\date{}
\begin{document} 

\thispagestyle{empty}

\noindent{\bf \LARGE Adaptive thresholding estimation of a Poisson intensity}
\thispagestyle{empty}

\vspace{0.2cm}

\noindent{\bf \LARGE with infinite support}

\vspace{0.6cm}

\noindent{\bf   \Large  Patricia  Reynaud-Bouret\footnotemark[1]  and  Vincent
Rivoirard\footnotemark[2]\footnotetext[1]{CNRS and D\'epartement
de Math\'ematiques et Applications, ENS-Paris, 45 Rue d'Ulm, 75230 Paris Cedex
05,     France. Email:    reynaud@dma.ens.fr}\footnotetext[2]{Equipe    Probabilit\'e,
Mod\'elisation et  Statistique, Laboratoire de Math\'ematique,  CNRS UMR 8628,
Universit\'e Paris Sud, 91405 Orsay Cedex, France. D\'epartement
de Math\'ematiques et Applications, ENS-Paris, 45 Rue d'Ulm, 75230 Paris Cedex
05, France. Email: Vincent.Rivoirard@math.u-psud.fr}}

\vspace{0.8cm}

\noindent 
\textbf{Abstract } The purpose of this paper is to estimate the intensity of a Poisson
process $N$ by using thresholding rules. In this paper, the intensity,
defined as the derivative of the mean measure of $N$ with respect to
$ndx$ where $n$ is a fixed parameter, is assumed to be non-compactly
supported. The estimator $\tilde{f}_{n,\gamma}$ based on random
thresholds is proved to achieve the same performance as the oracle
estimator up to a logarithmic term. Oracle inequalities allow to
derive the maxiset of $\tilde{f}_{n,\gamma}$. Then, minimax properties
of $\tilde{f}_{n,\gamma}$ are established. We first prove that the
rate of this estimator  on Besov spaces ${\cal B}^\al_{p,q}$ when
$p\leq 2$ is $(\ln(n)/n)^{\al/(1+2\al)}$. This result has two
consequences. First, it establishes that the minimax rate of Besov
spaces ${\cal B}^\al_{p,q}$ with $p\leq 2$ when non compactly
supported functions are considered is the same as for compactly
supported functions up to a logarithmic term. This result is
new. Furthermore, $\tilde{f}_{n,\gamma}$ is adaptive minimax up to a
logarithmic term. When $p>2$, the situation changes dramatically and
the rate of $\tilde{f}_{n,\gamma}$ on Besov spaces ${\cal
B}^\al_{p,q}$ is worse than $(\ln(n)/n)^{\al/(1+2\al)}$. Finally, the random threshold depends on a parameter $\gamma$ that has to be suitably chosen in practice. Some theoretical results provide upper and lower bounds of $\gamma$ to obtain satisfying oracle inequalities. Simulations reinforce these results.\\


\section{Introduction}
\subsection{Motivations}
Statistical  inference for  the problem  of estimating  the intensity  of some
Poisson process is considered in this paper. For this purpose, we assume that we
are given  observations of a Poisson  process on  $\R$ and our goal  is to
provide  a data-driven  procedure  with  good performance  for  estimating  the
intensity of this process.  

This problem has already been extensively investigated. For instance, Rudemo \cite{rud} studied data-driven histogram and kernel
estimates  based on the  cross-validation method.  Kernel estimates  were also
studied by Kutoyants \cite{kut} but in a non-adaptive framework. Donoho \cite{don2} fitted the
universal thresholding procedure proposed by Donoho and Johnstone \cite{dojo} for estimating
Poisson intensity by using the Anscombe's transform.  Kolaczyk \cite{kol} refined this
idea by investigating the tails of the distribution of the noisy wavelet
coefficients of the intensity. Still in the wavelet setting, Kim and Koo \cite{kk} studied maximum likelihood type
estimates  on  sieves  for  an  exponential  family of  wavelets.  And  for  a
particular inverse problem, Cavalier and Koo \cite{ck} first derived
optimal estimates in the minimax setting.  More precisely, for their tomographic
problem, Cavalier and Koo \cite{ck} pointed out minimax thresholding rules on Besov balls. By using
model   selection,   other   optimal   estimators  have   been   proposed   by
Reynaud-Bouret \cite{ptrfpois} who obtained oracle type inequalities and minimax rates on a
particular class of Besov spaces. 
In  the more general  setting of  point measure,  let us  mention the  work by
Baraud and Birg\'e \cite{bb}
which deals with histogram selection with the use of Hellinger distance.
 These model selection results have been generalized by Birg\'e
\cite{bir} who applied a general methodology based on $T$-estimators whose performance is measured
by the Hellinger distance.  However, as explained by Birg\'e \cite{bir}, this
 methodology is too computationally intensive to be implemented.
Related
works in other settings are worth citing.  For instance, in Poisson
regression, Kolaczyk and Nowak \cite{kn}
considered penalized maximum likelihood estimates, whereas Antoniadis {\it et
al.} \cite{abs} and Antoniadis and Sapatinas
\cite{as} focused on wavelet shrinkage. 

For our purpose, it is capital  to note that in the previous works, estimation
is performed by assuming that the intensity has in practice a compact support known by the
statistician,  $[0,1]$  in  general. Actually,  procedures  of
previous works are used after preprocessing. The support is indeed assumed to be in $[0,M]$, where $M$ is a known constant given either by some extra-knowledge
concerning the data or by the largest observation. Then, all the
observations are rescaled by dividing by $M$ so that observations belong to $[0,1]$. But all the
previous estimators depend on a tuning parameter, which therefore depends
in practice on $M$. If $M$ is overestimated, the estimation is poor. Even
taking the largest observation can be too rough if the distribution is
heavy-tailed so that the largest observation may be very far away from the
main part of the intensity. These problems become more crucial if one
 deals with data coming from other more complex point processes
(see \cite{gs} or \cite{rbr}) where one knows that the
support is overestimated by the theory and where the classical trick of using the largest observation cannot
be considered. Consequently the assumption of known and bounded support is
not considered in the present paper.

Let  us now  describe more  precisely our  framework. We  begin by  giving the
definition of a Poisson process to fix notations. 
\begin{Def}
Let $N$  be a random  countable subset of  $\R$. $N$ is  said to be  a Poisson
process on $\R$ if
\begin{itemize}
\item[-] for all $A\subset\R$,  the number of points of $N$ lying  in $A$ is a
random  variable, denoted  $N_A$, which  obeys  a Poisson  law with  parameter
denoted by $\mu(A)$ where $\mu$ is a measure on~$\R$,
\item[-]   for  all   finite   family  of   disjoints  sets   $A_1,\dots,A_n$,
$N_{A_1},\dots,N_{A_n}$ are independent. 
\end{itemize}
\end{Def}
The measure $\mu$, called the mean measure  of $N$, is assumed to be finite to
obtain almost surely a finite set of points for $N$. We denote
by $dN$ the discrete random measure $\sum_{T\in N}\delta_T$ so we have for any function $g$, 
$$\int g(x) dN_x=\sum_{T\in N}g(T).$$ We assume that the
mean measure is absolutely continuous with respect to the Lebesgue measure and
for $n$, a fixed integer, we denote by $f$ the intensity function
of $N$ defined by
$$\forall\;x\in\R,\quad f(x)=\frac{\mu(dx)}{ndx}.$$
We are  interested in estimating $f$  knowing the almost surely  finite set of
points $N$.  The parameter $n$ is introduced to derive results in an asymptotic
setting where $f$ is held fixed and $n$ goes to $+\infty$. Furthermore, note that observing
the $n$-sample  of Poisson  processes $(N_1,\dots,N_n)$ with  common intensity
$f$  with  respect  to the  Lebesgue  measure  is  equivalent to  observe  the
cumulative Poisson process $N=\cup_{i=A}^nN_i$ with intensity $n\times f$ with
respect to the Lebesgue measure. And in addition, this setting is close to the
problem  of density  estimation where  we  observe a  $n$-sample with  density
$f/\int f(x)dx$. 

Our goal is to build constructive data-driven estimators of $f$
and  for this  purpose,  we consider  thresholding  rules whose  risk is 
measured under the $\L_2$-loss.  Our  framework is the following. Of course, $f$ is  non-negative and since we assume that
$\mu(\R)<\infty$, this  implies that $f\in\L_1$. Since
we consider the $\L_2$-loss, $f$ is assumed to be in
$\L_2$.   In particular,  $f$ is not  assumed  to be  bounded  (except  in the  minimax  setting)  and, as  said
previously, its support may be infinite. 

In a different setting, the problem  of estimating a density  with infinite support  has been
partly solved from the minimax point of view.  See \cite{bh} where minimax results for a class of functions depending on a jauge are established or \cite{ik} and
\cite{gol} for Sobolev classes. In these papers, the loss function depends on the
parameters of the functional class. Similarly, Donoho {\it et al.} \cite{djkp} proved the
optimality of  wavelet linear estimators on Besov  spaces ${\cal B}^\al_{p,q}$
when  the $\L_p$-risk  is considered.  First  general results  where the  loss
is independent of  the functional class have been pointed  out by Juditsky and
Lambert-Lacroix \cite{jll} who
investigated minimax rates on the particular class of the Besov spaces ${\cal B}^\al_{\infty,\infty}$ for the $\L_p$-risk. When
$p>2+1/\al$, the minimax risk is bounded by $(\ln(n)/n) ^{2\al/(1+2\al)}$ so is of the same order up to a logarithmic term as
in the equivalent estimation problem on $[0,1]$. However, the behavior of the
minimax  risk changes dramatically  when $p\leq  2+1/\al$, and  in this  case, it
depends on $p$. In addition, Juditsky and
Lambert-Lacroix \cite{jll} pointed out a data-driven thresholding
procedure  achieving minimax rates  up to  a logarithmic  term. In the maxiset
setting, this procedure  has been studied by Autin  \cite{aut} and compared to
other classical thresholding procedures. Finally, we can also mention
that Bunea {\it et al.} \cite{btw} established oracle inequalities without any support assumption by
using Lasso-type estimators. 
\subsection{The estimation procedure}
Now, let us describe the estimation procedure considered in our paper. For this purpose, we assume
in the following that
the function $f$ can be written as follows:
\begin{equation}\label{decom}
f=\sum_{\la\in\Lambda}\be_\la\tilde\p_\la,\quad \mbox{with }\be_\la=\int f(x)\p_\la(x)dx\end{equation}
where $(\tilde\p_\la)_{\la\in\Lambda}$ and $(\p_\la)_{\la\in\Lambda}$
are two infinite families of linearly independent functions of
$\L_2$.    Most    of    the  further  results       are   valid    by    taking
$(\tilde\p_\la)_{\la\in\Lambda}=(\p_\la)_{\la\in\Lambda}$ to  be an orthonormal
basis of  $\L_2$ (the Haar basis  for instance). However,  minimax results
are established by considering special cases of biorthogonal wavelet bases
and      in      this      case      $(\tilde\p_\la)_{\la\in\Lambda}$      and
$(\p_\la)_{\la\in\Lambda}$ are different (see Section \ref{biorthogonal}). We note
$$\normp{f}=\left(\sum_{\la\in\Lambda}\be_\la^2\right)^{1/2}$$
which     is     equal    to     the    $\L_2$-norm     of     $f$     if
$(\tilde\p_\la)_{\la\in\Lambda}$ is orthonormal.
We  consider  thresholding  estimators based  on  observations
$(\hat \be_\la)_{\la\in\Gamma_n}$, where $\Gamma_n$ is a subset of $\Lambda$ chosen
later and
$$\forall\; \la\in \La,\quad \hb_\la= \frac1n\int_\R\p_\la(x) dN_x.$$
Observe that $\forall\; \la\in \La$,  $\hb_\la$ is an unbiased estimator of
$\be_\la$. As Juditsky and
Lambert-Lacroix \cite{jll}, we threshold $\hb_\la$ according to a random positive function 
of  $\la$ depending on $n$ and on a fixed  parameter $\gamma$  fixed later,  denoted by
$\eta_{\la,\gamma}$ and the thresholding estimator of $f$ is
\begin{equation}\label{defest}
\tilde{f}_{n,\gamma}=\sum_{\la \in \Ga_n} \tb_\la \tp_\la,
\end{equation}
where
$$\forall\; \la\in \La, \quad \tb_\la=\hb_\la \indic_{|\hb_\la|\geq\eta_{\la,\gamma}}.$$
In the sequel, we denote $\tilde f_\gamma=(\tilde{f}_{n,\gamma})_n$.
\\\\
The procedure (\ref{defest}) can also be seen as a model selection procedure. 
Indeed, for all $g=\sum_{\la\in\La} \al_\la\tp_\la$, we define
the least square contrast by
$$\ga_n(g)=-2\sum_{\la\in\La} \al_\la\hb_\la+\sum_{\la\in\La}\al_\la^2.$$
For all subset of indices $m$, we denote by $S_m$ the subspace generated by
$\{\tp_\la,\la\in m\}$. The projection estimator onto $S_m$ is defined by 
$$\hat{f}_m=\arg\min_{g\in S_m} \ga_n(g) = \sum_{\la \in m} \hb_\la \tp_\la.$$
Note that $$\ga_n(\hat{f}_m)=-\sum_{\la \in m} \hb_\la^2.$$
If we set 
$$\pen(m)=\sum_{\la\in m} \eta_{\la,\gamma}^2,$$
then  the  thresholding  estimator  can  be seen  as  a  penalized  projection
estimator since we have
\begin{equation*}
\tilde{f}_{n,\gamma}=\hat{f}_{\hat{m}}=\sum_{\la\in\Ga_n}\hb_\la\indic_{|\hb_\la|\geq\eta_{\la,\gamma}}\tp_\la\nonumber 
\end{equation*}
with
\begin{equation}\label{etoile2}
\hat{m}=\arg\min_{m\subset \Ga_n}
\left[\ga_n(\hat{f}_m)+\pen(m)\right].
\end{equation}
Such an interpretation is used  in Section \ref{oracle} and for the proof
of the main result of this paper. 
\subsection{Overview of the paper}
In  this paper,  our goals  are threefold.  First of  all, we  wish  to derive
theoretical  results for  the $\L_2$-risk  of $\tilde{f}_{\gamma}$  by using
three different points of view (oracle,  maxiset and minimax), then we wish to
discuss precisely the  choice of the threshold and finally  we wish to perform
some simulations. 

Let us now describe our results for our first aim. Theorem \ref{inegoracle}  is the main result  of the paper.  With a convenient
choice of the threshold and under very mild assumptions on $\Gamma_n$, Theorem
\ref{inegoracle} proves that  the thresholding estimate $\tilde{f}_{\gamma}$
satisfies an  oracle type inequality. We  emphasize that this  result is valid
under very mild assumptions on $f$.  Indeed, classical procedures use a bound for the sup-norm of $f$ (see \cite{ck}, \cite{djkp} or \cite{ptrfpois}).  This is not the case here where the threshold is the sum of two terms, a purely random
one that is the main term  and a deterministic one (see (\ref{defseuil})). The
definition of  the threshold  is extensively discussed  in Section  \ref{main}. By
using biorthogonal wavelet bases,  we derive from Theorem \ref{inegoracle} the
oracle inequality satisfied by $\tilde{f}_{\gamma}$. More precisely, Theorem
\ref{inegoraclelavraie}    in   Section    \ref{oracle}     shows   that
$\tilde{f}_{\gamma}$ achieves  the same performance as  the oracle estimator
up  to a logarithmic  term which  is the  price to  pay for  adaptation. From
Theorem \ref{inegoraclelavraie}, we derive the maxiset results of this paper. Let us recall that the maxiset
approach consists in investigating the  maximal space ({\it maxiset}), where
a  given procedure  achieves  a given  rate  of convergence.  For the  maxiset
theory, there is no a priori functional assumption. For a given procedure, the
practitioner states the  desired accuracy by fixing a rate  and points out all
the functions that can be estimated at this rate by the procedure.  Obviously,
the  larger  the  maxiset, the  better  the  procedure.  We prove  in  Section
\ref{maxiset},  that  under  mild  conditions,  the maxiset  of  the  estimate
 $\tilde{f}_{\gamma}$ for classical rates of the form $(\ln(n)/n)^{\al/(1+2\al)}$ is, roughly speaking, the intersection of
two spaces: a weak Besov space denoted $W_\al$ and the classical Besov space
 ${\cal B}^\al_{2,\infty}$ (see Theorem
\ref{maxisets}    and
Section  \ref{maxiset} for  more details).  Interestingly, this  maxiset result
provides examples of non bounded  functions that can be estimated at the
rate $(\ln(n)/n)^{\al/(1+2\al)}$ when $0<\al<1/4$ (see Proposition \ref{contreexlinfini}). Furthermore, we derive from the
maxiset result most of the minimax results briefly described now. 

As said previously, Juditsky and Lambert-Lacroix \cite{jll} established minimax rates for the problem of
estimating  a density with  an infinite  support for  the particular  class of
Besov spaces ${\cal B}^\al_{\infty,\infty}$ and for the $\L_p$-loss. To the best of our knowledge, minimax rates are unknown for Besov spaces ${\cal B}^\al_{p,q}$ except for very special cases described above. Our goal is to deal with this issue in the Poisson setting and for the $\L_2$-loss. We emphasize that for the minimax setting, we assume that the function to be estimated is bounded. The results that we obtain are the following. When $p\leq 2$, under mild assumptions, the minimax rate of convergence associated with ${\cal B}^\al_{p,q}$ is the classical rate $n^{-\al/(1+2\al)}$ up to a logarithmic term. So, it is of the same order
as in the  equivalent estimation problem on compact  sets of $\R$. Furthermore,
our estimate  achieves this rate up  to a logarithmic term.  When $p>2$, using
our maxiset result, we prove that this last result concerning our procedure is
no  more  true.  But  we  prove   under  mild  conditions  that  the  rate  of
$\tilde{f}_{\gamma}$ is not  larger than $(\ln(n)/n)^{\al/(2+2\al-1/p)}$ up to
a constant. Note that when $p=\infty$, $(\ln(n)/n)^{\al/(2+2\al)}$ is the rate
pointed out by Juditsky and Lambert-Lacroix \cite{jll} for minimax estimation under the $\L_2$-loss on the 
space  ${\cal  B}^\al_{\infty,\infty}$. Of  course,  when compactly  supported
functions are  considered, $\tilde{f}_{\gamma}$  is adaptive minimax  on Besov
spaces ${\cal B}^\al_{p,q}$ up to a logarithmic term. 

The  second goal  of  the paper  is to  discuss  the  choice of  the
threshold.  The starting point of this discussion is as follows. The main term of the threshold is
$(2\gamma\ln(n)  \tilde{V}_{\la,n})^{1/2}$  where $\tilde{V}_{\la,n}$  is  an
estimate of  the variance of  $\hat\be_\la$ and $\gamma$  is a constant  to be
calibrated (see (\ref{defseuil}) for further details). As usual, $\gamma$ has to be large enough to
obtain the theoretical results (see Theorem \ref{inegoracle}).  Such an assumption is very classical (see for
instance   \cite{jll},  \cite{djkp},   \cite{ck}  or   \cite{aut}).   But,  as
illustrated by Juditsky and
Lambert-Lacroix \cite{jll}, it is often too conservative for practical issues. In this paper, the assumption on the constant $\gamma$ is as
less conservative  as possible and actually most of the results are valid if
$\gamma>1$. So, the first issue is the following: what happens if $\gamma\leq 1$? 
Theorem \ref{lower} of  Section \ref{penaltyterm} proves that the rate obtained
for   estimating  the   simple  function   $\indic_{[0,1]}$  is   larger  than
$n^{-(\gamma+\e)/2}$ for any  $\e>0$. This proves that $\gamma<1$  is a bad choice
since,  with $\gamma>1$,  we  achieve the  parametric  rate up  to a  logarithmic
term.  Finally we  consider a special  class of  intensity
functions denoted ${\cal F}_n$. Theorems
\ref{classFn}  and \ref{uppth}   provide  upper  and lower  bounds of  the
maximal ratio  on ${\cal F}_n$  of the risk  of $\tilde f_{\gamma}$  by the
oracle risk and prove that $\gamma$ should not be too large.  

Finally we validate the previous range of $\gamma$
and refine it through a simulation study so that one can claim that $\gamma=1$
is a fairly good choice for all the encountered situations (finite/infinite
support, bounded/unbounded intensity, smooth/non-smooth functions).
\subsection{Outlines}
The paper is organized as follows. In Section \ref{main}, the main result of this paper
is  established.   Then,  Section \ref{biorthogonal}  introduces  biorthogonal
wavelet bases that are used to give oracle, maxiset and minimax results pointed out in
Section \ref{omm}.  Section \ref{penaltyterm}  discusses  the  choice of  the
threshold,  whereas  Section  \ref{simu}  provides some  simulations.  Finally,
Section \ref{proofs} gives the proof of the theoretical results.
\section{The main result}\label{main}
In the  sequel, for $R>0$, if ${\cal  F}$ is a  given Banach space,  we denote
${\cal  F}(R)$ the  ball of  radius $R$  associated with  ${\cal F}$.  For any
$1\leq p\leq\infty$, we denote
$$\norm{g}_p=\left(\int    |g(x)|^pdx\right)^{\frac{1}{p}}$$with the usual modification for $p=\infty$. 
To state the  main result, let us introduce the  following notations that are
 used throughout the paper. We set
$$
\forall\;\la\in\Lambda,\quad\hat{V}_{\la,n}=\int_\X\frac{\p_\la^2(x) }{n^2} dN_x,$$
the natural estimate of $V_{\la,n}$ that is the variance of $\hat\be_\la$:
$$\forall\;\la\in\Lambda,\quad
V_{\la,n}=\E(\hat{V}_{\la,n})=\frac{\si_\la^2}{n},$$
where $$\forall\;\la\in\Lambda,\quad\si_\la^2=\int_\X\p_\la^2(x)f(x)dx.$$
\begin{Th}
\label{inegoracle}
We assume that (\ref{decom}) is true and $\Gamma_n$ is such that for
$\la\in\Gamma_n$, $$\norm{\p_\la}_\infty  \leq c_{\p,n} \sqrt{n}$$  and that for
all $x\in \X$, 
\begin{equation}\label{mphi}
\mbox{card}\{ \la \in \Ga_n:\quad\p_\la(x)\not=0\}\leq m_{\p,n} \ln n, 
\end{equation}
where $c_{\p,n}$ and $ m_{\p,n}$ depend on $n$ and on the family $(\p_\la)_{\la\in\Lambda}$. 
 Let $\ga >1$. We set
\begin{equation}\label {defseuil}
\eta_{\la,\gamma}=\sqrt{2\gamma\ln n \tilde{V}_{\la,n} }+\frac{\gamma\ln
    n}{3n}\norm{\p_\la}_\infty,
\end{equation}
where    $$\tilde{V}_{\la,n}=\hat{V}_{\la,n}+\sqrt{2\gamma     \ln    n    \hat{V}_{\la,n}
\frac{\norm{\p_\la}_\infty^2}{n^2}}+3                \gamma                \ln
n\frac{\norm{\p_\la}_\infty^2}{n^2}$$
and consider $\tilde f_{n,\gamma}$ defined in (\ref{defest}).
Then for all $\e<\gamma-1$ and for all $p\geq 2$ and $q$ such that $\frac{1}{p}+\frac{1}{q}=1$ and $\frac{\gamma}{q}>1+\e$,
\begin{multline*}
\frac{\e}{2+\e}\E(\normp{\tilde{f}_{n,\gamma}-f}^2)\leq \E\left[\inf_{m \subset
    \Ga_n}\left\{\left(1+\frac{2}{\e}\right)\sum_{\la \not\in m}\be_\la^2 +\e\sum_{\la\in
    m}(\hb_\la-\be_\la)^2+\sum_{\la\in
    m}\eta_{\la,\gamma}^2\right\}\right]+\\
+c_0(1+\e) p^2\|f\|_1c_{\p,n}^2 m_{\p,n}\ln(n) \left[n^{-\frac{\gamma}{q(1+\e)}}+ n^{-\frac{\gamma}{q}}  (\max(\|f\|_1;1))^{\frac{1}{q}}\right] ,
\end{multline*}
where $c_0$ is an absolute constant. 
\end{Th}
Note that this result is proved under very mild conditions on the
decomposition of $f$. In particular we never use in the proof
that we are working on the real line but only that the decomposition
(\ref{decom}) exists.  Observe also that if we use wavelet bases (see (\ref{Lajk}) in Section
\ref{biorthogonal} below where we recall the standard wavelet setting) and if 
$$\Gamma_n\subset\left\{\la=(j,k)\in\Lambda:\quad 2^j\leq n^c\right\},$$
where $c$ is a constant, then $m_{\p,n}$ does not depend on $n$ and in addition,
$$\sup_n \left[c_{\p,n}n^{-(c-1)/2}\right]<\infty.$$
The threshold seems to be defined in a rather complicated manner. But, first observe that $\forall\,\theta>0$, $\forall\,\la\in\Gamma_n,$
\begin{equation}\label{equivseuil}
\sqrt{2\gamma\ln(n)\hat V_{\la,n}}+\frac{\gamma\ln (n)}{3n}\|\p_\la\|_{\infty}\leq \eta_{\la,\gamma}\leq c_{1,\theta}\sqrt{2\gamma\ln(n)\hat V_{\la,n}}+c_{2,\theta}\frac{\gamma\ln (n)}{3n}\|\p_\la\|_{\infty},
\end{equation}
with                                 $c_{1,\theta}=\sqrt{1+\frac{1}{2\theta}}$,
$c_{2,\theta}=\left(3\sqrt{2\theta+6}+1\right)$. 

Since $\hat V_{\la,n}$ is the  natural estimate of $V_{\la,n}$, the first term
of  the  left  hand side  of  (\ref{equivseuil})  is  similar to  the  threshold
introduced  by   Juditsky  and  Lambert-Lacroix  \cite{jll}   in  the  density
estimation setting.  But unlike Juditsky and Lambert-Lacroix
\cite{jll}, we add a deterministic term that allows to consider $\gamma$ close
to 1 and to control large deviations terms. In addition, since $\eta_{\la,\gamma}$ cannot be equal to 0, this allows to deal with
very  irregular  functions. However,  observe  that,  most  of the  time,  the
deterministic term is negligible compared to the first term as soon as
$\lambda\in\Gamma_n$  satisfies 
$\|\p_\la\|_{\infty}=o_n(n^{1/2})$. Finally, in the same spirit, $V_{\la,n}$ is slightly
overestimated and  we consider $\tilde V_{\la,n}$ instead  of $\hat V_{\la,n}$
to define the threshold. 

The    result    of   Theorem    \ref{inegoracle}    is    an   oracle    type
inequality. By exchanging the expectation and the infimum, the result provides
the expected  oracle inequality claimed in  Theorem \ref{inegoraclelavraie} of
Section \ref{oracle}. Theorem  \ref{inegoraclelavraie} is derived from Theorem
\ref{inegoracle} by evaluating $\E(\sum_{\la\in\Ga_n}\eta_{\la,\gamma}^2)$ and by
using biorthogonal wavelet bases. 
\section{Biorthogonal wavelet bases and Besov spaces}\label{biorthogonal}
In this paper, the intensity $f$ to  be estimated is assumed to belong to
$\L_1\cap\L_2$.   In  this case,  $f$ can  be
decomposed on the Haar wavelet basis and this property is used throughout
this paper. However, in Section \ref{minimax},  the Haar basis that suffers from lack
of regularity is not   considered. Instead, we  consider a particular
class of biorthogonal
wavelet bases that are described now. For this purpose, let us set 
$$\phi=\indic_{[0,1]}.$$ For any $r\geq 0$, there exist three functions
$\psi$, $\tilde \phi$ and $\tilde\psi$ with the following properties:
\begin{enumerate}
\item $\tilde\phi$ and $\tilde\psi$ are compactly supported,
\item  $\tilde\phi$  and $\tilde\psi$  belong  to  $C^{r+1}$, where  $C^{r+1}$
denotes the H\"older space of order $r+1$,
\item $ \psi$ is compactly supported and is a piecewise constant function,
\item $\psi$ is orthogonal to polynomials of degree no larger than $r$,
\item                     $                     \{(\phi_{k},\psi_{j,k})_{j\geq
0,k\in\Z},(\tilde\phi_{k},\tilde\psi_{j,k})_{j\geq     0,k\in\Z}\}$    is    a
biorthogonal family: $\forall\; j,j'\geq 0,$  $\forall\; k,k'\in\Z,$
$$\int_\R\psi_{j,k}(x)\tilde\phi_{k'}(x)dx=\int_\R\phi_{k}(x)\tilde\psi_{j',k'}(x)dx=0,$$
$$\int_\R\phi_{k}(x)\tilde\phi_{k'}(x)dx=1_{k=k'},\quad
\int_\R\psi_{j,k}(x)\tilde\psi_{j',k'}(x)dx=1_{j=j',k=k'},$$
where for any $x\in\R$ and for any $(j,k)\in~\Z^2$,
$$\phi_{k}(x)=\phi(x-k), \quad \psi_{j,k}(x)=2^{j/2}\psi(2^jx-k)$$
and
$$\tilde\phi_{k}(x)=\tilde\phi(x-k), \quad \tilde\psi_{j,k}(x)=2^{j/2}\tilde\psi(2^jx-k).$$
\end{enumerate}
This implies that for any $f\in\L_1\cap\L_2$, for any $x\in\R$,
$$f(x)=\sum_{k\in\Z}\alpha_k\tilde\phi_k(x)+\sum_{j\geq 0}\sum_{k\in\Z}\beta_{j,k}\tilde\psi_{j,k}(x),$$
where for any  $j\geq 0$ and any $k\in\Z$,
$$\alpha_k=\int_\R          f(x)\phi_k(x)dx,\quad          \beta_{j,k}=\int_\R
f(x)\psi_{j,k}(x)dx.$$Such  biorthogonal  wavelet  bases  have been  built  by
Cohen {\it et al.} \cite{cdf} as a special
case  of spline  systems (see  also the  elegant  equivalent construction  of
Donoho \cite{don} from boxcar functions). 
Of course, recall that all these properties except the second and the forth ones are true
for the Haar basis, where $\tilde\phi=\phi$ and
$\tilde\psi=\psi=\indic_{[0,1/2]}-\indic_{]1/2,1]}$, which allows to obtain in
addition an
orthonormal basis. This last point is not true for general biorthogonal wavelet
bases but we have the frame property: there exist two constants $c_1$ and $c_2$ only depending on the basis such that
$$c_1\left(\sum_{k\in\Z}\alpha_k^2+\sum_{j\geq
0}\sum_{k\in\Z}\beta_{j,k}^2\right)\leq                         \|f\|_{2}^2\leq
c_2\left(\sum_{k\in\Z}\alpha_k^2+\sum_{j\geq
0}\sum_{k\in\Z}\beta_{j,k}^2\right).$$
\\
In the sequel, when wavelet bases are used, we  set
\begin{equation}\label{Lajk}
\Lambda=\{\la=(j,k):\quad j\geq -1,k\in\Z\}.
\end{equation}
We denote  for   any  $\la\in\Lambda$,   $\p_\la=\phi_k$  (respectively
$\tilde\p_\la=\tilde\phi_k$) if $\la=(-1,k)$ and
$\p_\la=\psi_{j,k}$ (respectively $\tilde\p_\la=\tilde\psi_{j,k}$) if $\la=(j,k)$ with $j\geq 0$. Similarly, $\be_\la=\al_k$
if  $\la=(-1,k)$ and $\be_\la=\be_{j,k}$  if $\la=(j,k)$  with $j\geq  0$. So,
(\ref{decom}) is valid. An important feature of the bases introduced previously is the following: there exists a constant $\mu_{\psi}>0$ such that
\begin{equation}\label{minophi}
\inf_{x\in[0,1]}|\phi(x)|\geq 1,\quad\inf_{x\in supp(\psi)}|\psi(x)|\geq\mu_{\psi},
\end{equation}
where $supp(\psi)=\{x\in\R:\quad \psi(x)\not=0\}.$ This property is used
throughout the paper. \\\\
Now, let us recall some properties of Besov
spaces that are extensively used in the next section. We refer the
reader to \cite{dl} and \cite{hardle}  for the definition of Besov spaces,
denoted ${\cal B}^{\al}_{p,q}$ in the sequel, and a review
of their properties explaining their important role in approximation
theory and statistics. We just  recall the sequential characterization of Besov
spaces by using the biorthogonal wavelet basis (for further
details, see \cite{dj}). 
Let $1\leq  p,q\leq \infty$ and $0<\al<r+1$, the  ${\cal B}^\al_{p,q}$-norm of
$f$ is equivalent to the norm
$$
\norm{f}_{\al,p,q}=\left\{
\begin{array}{ll}
\norm{(\al_{k})_k}_{\ell_p}+\left[\sum_{j\geq 0}2^{jq(\al+1/2-1/p)}\norm{(\be_{j,k})_{k}}_{\ell_p}^{q}\right]^{1/q}&\mbox{ if } q<\infty,\\
\norm{(\al_{k})_k}_{\ell_p}+\sup_{j\geq 0}2^{j(\al+1/2-1/p)}\norm{(\be_{j,k})_{k}}_{\ell_p}&\mbox{ if } q=\infty.
\end{array}
\right.
$$
We  use this norm to define the radius of Besov balls.  For any $R>0$, if
$0<\al'\leq\al<r+1$, $1\leq  p\leq p'\leq\infty$ and  $1\leq q\leq q'\leq\infty$,
we obviously have
$$
\mathcal{B}^\al_{p,q}(R)\subset\mathcal{B}^\al_{p,q'}(R),\quad
\mathcal{B}^\al_{p,q}(R)\subset\mathcal{B}^{\al'}_{p,q}(R).$$
Moreover 
\begin{equation}\label{inclusion}
\mathcal{B}^{\al}_{p,q}(R)\subset\mathcal{B}^{\al'}_{p',q}(R)
\mbox{ if } \al-\frac{1}{p}\geq \al'-\frac{1}{p'}. 
\end{equation}
The class of Besov  spaces $\mathcal{B}^\al_{p,\infty}$ provides a useful tool
to classify  wavelet decomposed  signals in function  of their  regularity and
sparsity properties (see \cite{joh}). Roughly speaking, regularity increases when $\al$ increases whereas
sparsity  increases when $p$  decreases. Especially,  the spaces  with indices
$p<2$  are of particular  interest since  they describe  very wide  classes of
inhomogeneous  but  {\it  sparse  functions}  (i.e.   with  a  few  number  of
significant coefficients). The case $p\geq 2$ is typical of {\it dense functions}. 
\section{Oracle, maxiset and minimax results}\label{omm}
Along this section, we use biorthogonal wavelet bases as
defined in Section \ref{biorthogonal}.
\subsection{Oracle inequalities}\label{oracle}
 Ideal adaptation is studied in \cite{dojo} using the class
of shrinkage rules in  the context of  wavelet function
estimation. This is the performance that can be achieved with the
aid of an oracle. In our setting, the oracle  does not tell us the true function, but 
tells us, for our thresholding method, the coefficients that have to be kept. This ``estimator'' obtained
with the aid of an oracle is not a true estimator, of course, since it depends
on $f$.  But it represents  an ideal for  a particular estimation  method. The
approach  of  ideal  adaptation  is   to  derive  true  estimators  which  can
essentially ``mimic'' the performance of the oracle estimator. So, using the
interpretation of thresholding rules as model selection rules, the oracle provides the model $\bar{m}\subset \Ga_n$ such that 
the quadratic risk of $\hat{f}_{\bar{m}}$ is minimum. Since, we have for any $m\subset \Gamma_n$,
$$\E(\normp{\hat f_m-f}^2)=\sum_{\la\in m}V_{\la,n}+\sum_{\la\not\in m}\be_\la^2,$$
the  oracle estimator $\hat{f}_{\bar{m}}$  is obtained  by taking $\bar{m}=\{\la\in
\Ga_n:\quad \be_\la^2>V_{\la,n}\}$ and
$$\hat{f}_{\bar{m}}=\sum_{\la\in\Ga_n}\hat\be_\la\indic_{\be^2_\la> V_{\la,n}}\tp_\la.$$
Its risk (the oracle risk) is then 
$$\E(\normp{\hat{f}_{\bar{m}}-f}^2)=\E\sum_{\la\in\Ga_n}(\hat\be_\la\indic_{\be^2_\la> V_{\la,n}}-\be_\la)^2+\sum_{\la\notin\Ga_n}\be_\la^2=\sum_{\la\in\Gamma_n}\min(\be_\la^2,V_{\la,n})+\sum_{\la\notin\Gamma_n}\be_\la^2.$$
Our aim  is now to  compare the risk  of $\tilde{f}_{n,\gamma}$ to  the oracle
risk.  We deduce from Theorem \ref{inegoracle} the following result.
\begin{Th}
\label{inegoraclelavraie} 
Let us fix  two constants $c\geq 1$  and $c'\in\R$, and let us  define for any
$n$, $j_0=j_0(n)$ the integer such that $2^{j_0}\leq n^c(\ln(n))^{c'}<2^{j_0+1}$. Let $\ga >c$ and let $\eta_{\la,\gamma}$ be as in Theorem \ref{inegoracle}. Then $\tilde f_{n,\gamma}$ defined with
$$\Gamma_n=\left\{\la=(j,k)\in\Lambda:\quad j\leq j_0\right\}$$ achieves the following oracle inequality:
\begin{equation}\label{inegoraclelavraie1}
\E(\normp{\tilde{f}_{n,\gamma}-f}^2)\leq C_1(\gamma,\p)\left[\sum_{\la\in\Gamma_n}\min(\be_\la^2,V_{\la,n}\ln(n))+\sum_{\la\notin\Gamma_n}\be_\la^2\right]+\frac{C_2(\gamma,\|f\|_1, c,c', \p)}{n}
\end{equation}
where $ C_1(\gamma,\p)$ is a positive constant depending only on the
basis and of 
the value of $\ga$ and where $C_2(\gamma,\|f\|_1, c,c', \p)$ is also a positive constant depending
on $\ga$ and the basis but also on $\|f\|_1 $, $c$ and $c'$.\\
\end{Th}
The oracle inequality (\ref{inegoraclelavraie1}) satisfied by $\tilde{f}_{n,\gamma}$ proves that this
estimator achieves essentially the oracle risk up to a logarithmic term. This
logarithmic term is the price we pay for adaptivity, i.e.  for not knowing the
wavelet coefficients that have to be kept. 
In  section \ref{penaltyterm}, optimization  of the  constants of  the stated
result is performed for a particular class of functions.\\
\subsection{Maxiset results}\label{maxiset}
As said in the introduction, if $f^*$ is a given procedure, the maxiset study of $f^*$ consists in deciding the accuracy of the estimate by fixing a 
prescribed rate $\rho^*$ and in pointing out all the functions $f$ such 
that $f$ can be estimated by the procedure $f^*$ at the target 
rate $\rho^*$. The maxiset of the procedure $f^*$ for this rate 
$\rho^*$ is the set of all these functions. So, we 
set the following definition. 
\begin{Def}
Let $\rho^*=(\rho^*_n)_n$ be a decreasing sequence of
positive real numbers and let $f^*=(f_n^*)_n$ be an
estimation procedure. The maxiset of $f^*$  associated with the
rate $\rho^*$ and the $\L_{2}$-loss is 
$$MS(f^*,\rho^*) =\left\{f\in\L_{1}\cap\L_{2}:\quad
\sup_n\left[(\rho^*_n)^{-2}\E\normp{f_n^*-f}^2\right]<+\infty\right\},$$
the ball of radius $R>0$ of the maxiset is defined by
$$MS(f^*,\rho^*)(R) =\left\{f\in\L_{1}\cap\L_{2}:\quad
\sup_n\left[(\rho^*_n)^{-2}\E\normp{f_n^*-f}^2\right]\leq R^2\right\}.$$
\end{Def}
To establish the maxiset result of this section, we use Theorem \ref{inegoraclelavraie}, so we
need to assume  that the estimation procedure is performed in a  ball of $\L_1\cap\L_2$.  Even, if the size of
the balls does not play an important role, this assumption is essential.
In this setting, we use the following notation. If ${\cal F}$ is a given space
 \begin{eqnarray*}
MS(f^*,\rho^*) &:=&{\cal F}
 \end{eqnarray*}
means in the sequel that for any $R>0$, there exists $R'>0$ such that
 $$MS(f^*,\rho^*)(R)\cap \L_1(R)\cap\L_2(R) \subset {\cal F}(R')\cap \L_1(R)\cap\L_2(R)$$
and for any $R'>0$, there exists $R>0$ such that $${\cal F}(R')\cap \L_1(R')\cap\L_2(R')\subset MS(f^*,\rho^*)(R)\cap \L_1(R')\cap\L_2(R').$$
In this section, for any $\al>0$, we investigate the set of functions that can be estimated by
$\tilde f_\gamma=(\tilde f_{n,\gamma})_n$ at
the rate $\rho_\al=(\rho_{n,\al})_n$, where for any $n$, 
$$\rho_{n,\al}=\left(\frac{\ln(n)}{n}\right)^{\frac{\al}{1+2\al}}.$$
More precisely, we  investigate for any radius $R>0$:
$$MS(\tilde f_\gamma,\rho_\al)(R)=\left\{f\in\L_{1}\cap\L_{2}:\quad
\sup_n\left[\rho_{n,\al}^{-2}\E\normp{\tilde
f_{n,\gamma}-f}^2\right]\leq R^2\right\}.$$ 
To  characterize  maxisets of  $\tilde  f_\gamma$, we introduce the following spaces. 
\begin{Def}
We define for all $R>0$ and for all $s>0,$
$$W_s=\left\{f=\sum_{\la\in\Lambda}\be_\la\tp_\la:\quad \sup_{t>0}t^{\frac{-4s}{1+2s}}\sum_{\la\in\Lambda}\be_\la^2
\indic_{|\be_\la|\leq \si_\la t}<\infty\right\},$$
  the ball of radius $R$ associated with $W_s$ is:
$$W_s(R)=\left\{f=\sum_{\la\in\Lambda}\be_\la\tp_\la:\quad \sup_{t>0}t^{\frac{-4s}{1+2s}}\sum_{\la\in\Lambda}\be_\la^2
\indic_{|\be_\la|\leq \si_\la t}\leq R^{\frac{2}{1+2s}}\right\},$$
and for any sequence of spaces $\Gamma=(\Gamma_n)_n$ included in $\Lambda$,
$$B^s_{2,\Gamma}=\left\{f=\sum_{\la\in\Lambda}\be_\la\tp_\la:\quad  \sup_n\left[\left(\frac{\ln(n)}{n}\right)^{-2s}\sum_{\la\not\in\Gamma_n}\be_\la^2\right]<\infty\right\}$$
and
$$B^s_{2,\Gamma}(R)=\left\{f=\sum_{\la\in\Lambda}\be_\la\tp_\la:\quad  \sup_n\left[\left(\frac{\ln(n)}{n}\right)^{-2s}\sum_{\la\not\in\Gamma_n}\be_\la^2\right]\leq R^2\right\}.$$
\end{Def}
In \cite{dl}, a justification of the form of the radius of $W_s$ and 
further details are
 provided. These
spaces can be viewed as weak versions of classical Besov spaces, hence
they are
 denoted in the sequel weak Besov spaces. In particular, the spaces $W_s$ naturally model sparse signals (see \cite{rivbernoulli}).  
Note that if for all~$n$,
$$\Gamma_n=\left\{\la=(j,k)\in\Lambda:\quad     j\leq    j_0\right\}$$    with
$$2^{j_0}\leq \left(\frac{n}{\ln n}\right)^c<2^{j_0+1},\quad c>0$$ then, $B^s_{2,\Gamma}$ is the classical Besov
space ${\cal  B}^{s/c}_{2,\infty}$ if  some properties of regularity  and vanishing
moments are
satisfied by the wavelet basis (see Section \ref{biorthogonal}). 
We define $ B^s_{2,\Gamma}$ and $W_s$ by using biorthogonal wavelet
bases. However, as established in \cite{dl}, they also have
different definitions proving that, under mild conditions, this dependence on the basis is not crucial at all.
Using Theorem \ref{inegoraclelavraie}, we have the following result. 
\begin{Th}\label{maxisets}
Let us fix two constants $c\geq 1$ and $c'\in\R$, and let us define for any
$n$, $j_0=j_0(n)$ the integer such that $2^{j_0}\leq
n^c(\ln(n))^{c'}<2^{j_0+1}$.  Let $\ga  >c$ and  let $\eta_{\la,\gamma}$ be  as  in Theorem
\ref{inegoracle}. Then, the procedure  defined in (\ref{defest}) with
the sequence $\Gamma=(\Gamma_n)_n$ such that
$$\Gamma_n=\left\{\la=(j,k)\in\Lambda:\quad  j\leq j_0\right\}$$  achieves the
following maxiset performance: for all $\al>0$,
$$MS(\tilde f_\gamma,\rho_\al):=B^{\frac{\al}{1+2\al}}_{2,\Gamma}\cap W_\al.$$
In particular, if $c'=-c$  and $0<\frac{\al}{c(1+2\al)}<r+1$, where $r$ is the
parameter of the biorthogonal basis introduced in Section \ref{biorthogonal},
$$MS(\tilde f_\gamma,\rho_\al):={\cal B}^{\frac{\al}{c(1+2\al)}}_{2,\infty}\cap W_\al.$$
\end{Th}
\begin{Remark}\label{uppga}
In order to obtain maxisets  as large as possible, Inequality (\ref{gammamin})
of the proof of Theorem \ref{maxisets} suggests
to choose $\ga>1$ as small as possible. 
\end{Remark}
The maxiset of $\tilde f_\gamma$ is  characterized by two spaces: a weak Besov
space  that  is directly  connected  to  the  thresholding nature  of  $\tilde
f_\gamma$  and  the   space  $B^{\al/(1+2\al)}_{2,\Gamma}$  that  handles  the
coefficients  that  are  not  estimated,  which  corresponds  to  the  indices
$j>j_0$. This maxiset result is similar to the result obtained by Autin
\cite{aut} in the density estimation setting but our assumptions are
less restrictive (see Theorem 5.1 of \cite{aut}). \\
Now, let us point out a family of examples of functions that
illustrates the previous result. For this purpose, we consider the Haar
basis that allows to have simple formula for the wavelet coefficients. Let us consider for any $0<\beta<1/2$, $f_\beta$ such that
$$\forall\;x\in\R,\quad f_\beta(x)=x^{-\beta}1_{x\in ]0,1]}.$$
The following result points out that if $\al$ is small
enough, for  a convenient  choice  of $\beta$,
$f_\beta$  belongs  to  $MS(\tilde  f_\gamma,\rho_\al)$ (so  $f_\beta$  can  be
estimated at the rate $\rho_\al$), and in addition $f_\beta\not\in \L_{\infty}.$
\begin{Prop}\label{contreexlinfini}
We consider the Haar basis and we set  $c'=-c$.  For  $0<\al<1/4$, under  the
assumptions of Theorem \ref{maxisets}, if 
$$0<\beta\leq\frac{1-4\al}{2+4\al},$$ then for $c$ large enough,
\begin{equation}\label{MSbeta}
f_\beta\in                 MS(\tilde                 f_\gamma,\rho_\al):={\cal
B}^{\frac{\al}{c(1+2\al)}}_{2,\infty}\cap W_\al,\nonumber
\end{equation}
where ${\cal
B}^{\frac{\al}{c(1+2\al)}}_{2,\infty}$ and $W_\al$ are viewed as sequence
spaces. In addition, $f_\beta\not\in \L_{\infty}.$
\end{Prop}
This result is proved by using the Haar basis, so the functional spaces are
viewed as sequence spaces.  We conjecture that for more general biorthogonal wavelet bases, we can also build not bounded functions that belong to $MS(\tilde f_\gamma,\rho_\al).$ 
\subsection{Minimax results}\label{minimax}
Let ${\cal F}$ be  a functional space and ${\cal F}(R)$ be  the ball of radius
$R$ associated with ${\cal F}$. ${\cal F}(R)$ is assumed to belong to a ball of $\L_1\cap\L_2$. Let us recall that a procedure $f^*=(f_n^*)_n$ achieves the rate $\rho^*=(\rho^*_n)_n$ on ${\cal F}(R)$ (for the $\L_2$-loss) if 
$$\sup_n\left[(\rho^*_n)^{-2}\sup_{f\in {\cal F}(R)}\E(\normp{f_n^*-f}^2)\right]<\infty.$$
Let us consider the procedure $\tilde f_\gamma$ and the rate $\rho_\al=(\rho_{n,\al})_n$ where for any $n$,
$$\rho_{n,\al}=\left(\frac{\ln (n)}{n}\right)^{\frac{\al}{1+2\al}}$$
as in  the previous  section. Obviously, $\tilde  f_\gamma$ achieves  the rate
$\rho_\al$ on ${\cal F}(R)$ if and only if there exists $R'>0$ such that $${\cal F}(R)\subset MS(\tilde
f_\gamma,\rho_\al)(R')\cap \L_1(R')\cap \L_2(R').$$
Using  results  of the  previous  section, if  $c'=-c$  and  if properties  of
regularity and vanishing moments are
satisfied by the wavelet basis, this is satisfied if
and only if there exists $R''>0$ such that $${\cal F}(R)\subset {\cal
B}^{\frac{\al}{c(1+2\al)}}_{2,\infty}(R'')\cap W_\al(R'')\cap \L_1(R'')\cap \L_2(R'').$$  We  apply this simple rule for
Besov balls.   
So, in the sequel, we  assume that the function $f$ to be estimated belongs
to  a ball of  $\L_1\cap \L_2$.  In addition,  we   assume that  $f$ also
belongs to a ball of $\L_{\infty}$. This last assumption which is not necessary
to derive maxiset results (see Theorem \ref{maxisets} or Proposition \ref{contreexlinfini}) is
unavoidable in some sense in the minimax setting. For a precise justification of this point,
see for  instance Corollary 1 of  \cite{bir}. Consequently, in  the sequel, we
set for any $R>0$,
$${\cal  L}_{1,2,\infty}(R)=\left\{f:\quad  \norm{f}_1\leq  R,  \norm{f}_2\leq
R,\norm{f}_{\infty}\leq R\right\}.$$
In  the sequel,  minimax  results   depend  on the  parameter  $r$ of  the
biorthogonal basis introduced in
Section \ref{biorthogonal} to measure the regularity of the reconstruction wavelets $(\tilde\phi,\tilde\psi)$. 
\subsubsection{Minimax  estimation on  Besov spaces  ${\cal B}^\al_{p,q}$
when $p\leq 2$}\label{nc-sparse}
To  the  best  of  our  knowledge,  the minimax  rate  is  unknown  for  ${\cal
B}^\al_{p,q}$ when $p<\infty$. Let us investigate this problem by pointing out
the minimax properties of $\tilde f_\gamma$ on ${\cal B}^\al_{p,q}$
when $p\leq 2$. We have the following result.
\begin{Th}\label{minimaxnoncompact}
Let $R,R'>0$, $1\leq p,q\leq \infty$ and
$\al\in\R$ such that $\max(0,1/p-1/2)<\al<r+1$. 
Let $c\geq 1$ large enough such that
\begin{equation}\label{sural}
\al\left(1-\frac1{c(1+2\alpha)}\right)\geq \frac1p -\frac12.
\end{equation}
Let us define for any $n$, $j_0=j_0(n)$ the integer such
that $$2^{j_0}\leq n^c(\ln(n))^{-c}<2^{j_0+1}.$$ Then, if $p\leq 2$, $\tilde f_\gamma=(\tilde f_{n,\gamma})_n$ defined with
$$\Gamma_n=\left\{\la=(j,k)\in\Lambda:\quad j\leq j_0\right\}$$ and $\ga >c$ achieves the
rate $\rho_{\al}$
on ${\cal B}^{\al}_{p,q}(R)\cap{\cal L}_{1,2,\infty}(R') $. Indeed, for any $n$,
\begin{equation}\label{risqueBesov}
\sup_{f\in       {\cal       B}^{\al}_{p,q}(R)\cap{\cal L}_{1,2,\infty}(R')}\E(\normp{\tilde{f}_{n,\gamma}-f}^2)\leq
C(\ga,c,R,R',\al,p,\p)\left(\frac{\ln n}{n}\right)^{2\al/(1+2\al)}
\end{equation}
where $C(\ga,c,R,R',\al,p,\p)$ depends on $R'$, $\ga$, $c$, on the parameters of the Besov ball and on the
basis.\\ Furthermore, let  $p^*\geq 1$ and $\al^*>0$ such that 
\begin{equation}\label{adapt*}
\al^*\left(1-\frac1{c(1+2\alpha^*)}\right)\geq \frac{1}{p^*} -\frac12. 
\end{equation}
 Then, $\tilde f_\gamma$ is adaptive minimax up to a logarithmic term on $$\left\{{\cal
B}^{\al}_{p,q}\cap{\cal  L}_{1,2,\infty}:\quad   \al^*\leq\al<  r+1,  \
p^*\leq p\leq 2,\ 1\leq
q\leq\infty\right\}.$$
\end{Th}
This result points out the minimax rate associated 
 with ${\cal B}^{\al}_{p,q}(R)\cap{\cal L}_{1,2,\infty}(R')$  up
to a logarithmic term and in addition  proves that it is of the same order
as  in  the  equivalent  estimation  problem  on $[0,1]$  (see
\cite{djkp}). It means that, roughly speaking, it is not harder
to estimate sparse  non-compactly supported  functions than  sparse compactly
supported functions from the minimax point of view. In addition, the procedure
$\tilde f_\gamma$ does the job up to a logarithmic term.  When $p>2$ (i.e., when dense functions are
considered),  this conclusion  does  not  remain true. 
\subsubsection{Minimax estimation on  Besov spaces  ${\cal B}^\al_{p,q}$ when $p> 2$}\label{nc-dense}
Before considering the case of estimation of non-compactly
supported functions, let us establish the
following result. We denote ${\cal K}$ the set of
compact  sets of $\R$  containing a  non-empty interval.  We define  for $K\in
{\cal K}$, ${\cal B}^\al_{p,q,K}(R)$ the set of functions supported by
$K$ and belonging to ${\cal B}^\al_{p,q}(R)$.
\begin{Cor}
We  assume  that
assumptions of Theorem \ref{minimaxnoncompact} are true. 
For any $p\geq 1$, $\tilde f_\gamma$ achieves the
rate $\rho_{\al}$ on ${\cal B}^{\al}_{p,q,K}(R)\cap{\cal L}_{1,2,\infty}(R')$.\\
Furthermore, $\tilde f_\gamma$ is adaptive minimax up to a logarithmic term on  $$\left\{{\cal
B}^{\al}_{p,q,K}\cap{\cal     L}_{1,2,\infty}:\quad    \al^*\leq\al<r+1,\
p^*\leq p\leq \infty ,\ 1\leq
q\leq\infty, K\in{\cal K}\right\},$$
where  $\al^*$ and  $p^*$ satisfy  (\ref{adapt*}). 
\end{Cor}
To prove this corollary, it is enough to apply Theorem \ref{minimaxnoncompact}
and to note that 
${\cal   B}^{\al}_{p,q,K}(R)\subset{\cal  B}^\al_{p,\infty,K}(R)\subset  {\cal
B}^\al_{2,\infty,K}(\tilde R)$ for $\tilde R$ large enough when $p>2$.\\ 

When non-compactly supported functions are considered, this result is not true
and we can prove the following theorem. 
\begin{Th}\label{contrexem}
Let $p>2$ and $\al>0$.
There exists a positive function $f$ such that
$$       f\in       \mathbb{L}_1\cap   \mathbb{L}_2\cap     \mathbb{L}_{\infty}\cap       {\cal
B}^{\al}_{p,\infty} \mbox{ and } f\notin W_\al,$$
where the function spaces are viewed as sequential spaces (the Haar basis is used).
\end{Th}
\begin{Remark}
 This result is established by using the Haar basis. We conjecture that it remains true for more general biorthogonal wavelet bases.
\end{Remark}
This result  proves that $\tilde  f_\gamma$ does not  achieve the rate  $\rho_{\al}$ on
$ {\cal
B}^{\al}_{p,\infty}$  when $p>2$,  showing that  minimax statements  of Section
\ref{nc-sparse} are not valid in this setting. As said previously, it seems to us that minimax rates and adaptive
minimax rates are unknown  for $ {\cal B}^{\al}_{p,\infty}$, when $2<p<\infty$
even if Donoho {\it et al.}
\cite{djkp} provided some lower bounds in the density framework. For the case $p=\infty$, see \cite{jll}. 
\\\\
Now,  let  us   investigate  the  rate  achieved  by   $\tilde  f_\gamma$  on  ${\cal
B}^\al_{p,q}(R)$ when $p>2$. 
\begin{Th}\label{vitesse}
Let $R,R'>0$, $1\leq
q\leq \infty$, $2<
p\leq \infty$ and $\al\in\R$ such that $1/(2p)<\al<r+1$. 
Let us define for any $n$, $j_0=j_0(n)$ the integer such
that $$2^{j_0}\leq n^c(\ln n)^{-c}<2^{j_0+1},$$ with $c\geq 1$.Then,  $\tilde f_\gamma=(\tilde f_{n,\gamma})_n$ defined with 
$$\Gamma_n=\left\{\la=(j,k)\in\Lambda:\quad  j\leq j_0\right\}$$ and  $\ga >c$
achieves the following performance. For any $n$,
\begin{equation}\label{risqueBesovbis}
\sup_{f\in {\cal B}^{\al}_{p,q}(R)\cap {\cal L}_{1,2,\infty}(R')}\E(\normp{\tilde{f}_{n,\gamma}-f}^2)\leq
C(\ga,c,R,R',\al,p,\p)\left(\frac{\ln n}{n}\right)^{\frac{\al}{1+\al-\frac{1}{2p}}}. \nonumber
\end{equation}
where $C(\ga,c,R,R',\al,p,\p)$ depends on $R'$, $\ga$, $c$, $c'$, on the parameters of the Besov ball and on the
basis.
\end{Th}
Note that when $p=\infty$, the risk is bounded by
$\left(\frac{\ln n}{n}\right)^{\frac{\al}{1+\al}}$ up  to a constant, which is
the rate of the minimax risk on ${\cal B}^{\al}_{\infty,\infty}(R)$ up to a logarithmic term in the density estimation
setting      (see      Theorem       1      of      \cite{jll}).      However,
$\frac{\al}{1+\al-\frac{1}{2p}}\stackrel{p\to
2}{\longrightarrow}\frac{\al}{\al+\frac{3}{4}}$      and      $\left(\frac{\ln
n}{n}\right)^{\frac{\al}{\al+\frac{3}{4}}}>>\rho_{n,\al}^2$.                  So,
$\tilde{f}_\gamma$  is probably  not adaptive  minimax on  the whole  class of
Besov spaces.
However, we establish that
our procedure is adaptive minimax (with the exact power of the
logarithmic factor) over weak Besov spaces without any support assumption. 
\subsubsection{Minimax estimation on $W_\al$ and
adaptation with respect to $\al$}
We  investigate  in  this  section  a  lower bound  for  the  minimax  risk  on
$W_\al(R)\cap{\cal       B}^{\frac{\al}{1+2\al}}_{2,\infty}(R')\cap      {\cal
L}_{1,2,\infty}(R'')$ for $R,R',R''>0$ viewed as sequence spaces for
the Haar basis and we set
$${\cal R}(W_\al(R)\cap{\cal       B}^{\frac{\al}{1+2\al}}_{2,\infty}(R')\cap      {\cal
L}_{1,2,\infty}(R''))=\inf_{\hat f}\sup_{f\in W_\al(R)\cap{\cal B}^{\frac{\al}{1+2\al}}_{2,\infty}(R')\cap {\cal
L}_{1,2,\infty}(R'')}\E(\normp{\hat f-f}^2).$$
\begin{Th}\label{minimaxsurmaxiset}
For $\al>0$, we have
$$\liminf_{n\to  \infty}  \rho_{n,\al}^{-2} 
{\cal R}(W_\al(R)\cap{\cal       B}^{\frac{\al}{1+2\al}}_{2,\infty}(R')\cap      {\cal
L}_{1,2,\infty}(R''))
\geq
c(\al) R^{\frac{2}{1+2\alpha}},$$
where  $c(\al)$ depends  only on  $\al$, as  soon as  $R''\geq 1$  and $R'\geq
R^{\frac{1}{1+2\alpha}}\geq 1$. 
\end{Th}
Using Theorem \ref{maxisets}, we immediately deduce the following result. 
\begin{Cor}\label{adaptivemini}
The procedure  $\tilde{f}_\gamma$ defined in Theorem \ref{minimaxnoncompact}  with $c=-c'=1$ and
with $\gamma>1$  is minimax
on    $W_\al(R)\cap{\cal    B}^{\frac{\al}{1+2\al}}_{2,\infty}(R')\cap   {\cal
L}_{1,2,\infty}(R'')$ and is adaptive minimax on
$$\left\{
W_\al(R)\cap{\cal       B}^{\frac{\al}{1+2\al}}_{2,\infty}(R')\cap      {\cal
L}_{1,2,\infty}(R''):\quad\al>0, 1\leq R'', \ 1\leq R\leq R'\right\}.$$
\end{Cor}
\begin{Remark}
 These results are established for  the Haar basis. It
 is probably true for more general biorthogonal wavelet bases, but we
 were not able to prove it.
\end{Remark}
\section{How to choose the parameter $\ga$}\label{penaltyterm}
In this section, our goal is to  find lower and upper bounds for the parameter
$\ga$. The aim and proofs are inspired by Birg\'e and Massart \cite{bm} who considered penalized estimators and
calibrated  constants for  penalties  in a  Gaussian  regression framework.  In
particular, they showed that if the penalty constant is smaller than 1, then the
penalized estimator behaves in a quite unsatisfactory way. This study was used
in practice to derive adequate data-driven penalties by Lebarbier \cite{leb}. 

We  assume that  the function  $f$ to  be estimated
belongs to a  restricted functional space. More precisely,  we assume that for
$n$ large enough, $f$ belongs to ${\cal F}_n$ where for any $n$,
$${\cal    F}_{n}=\left\{f\in    \L_1\cap\L_2\cap\L_{\infty}:\quad   F_\la\geq
\frac{(\ln n)(\ln\ln n)}{n}1_{F_\la>0}, \ \forall\;\la\in\Lambda\right\},$$
with $F_\la=\int_{supp(\p_\la)} f(x)dx.$ Observe that 
${\cal F}_{n}$ only contains functions with finite support. If the Haar basis is considered, any function supported by $[0,1]$
that is constant on each interval  of a dyadic partition of $[0,1]$ belongs to
${\cal F}_n$ for $n$ large enough.  In
addition, the interest of the class ${\cal F}_n$ lies in the natural bridge it
constitutes between the model of this paper and the regression model for which
the number of non-zero coefficients is always bounded by $n$. 
These reasons
justify the  importance of well estimating  functions of ${\cal  F}_n$ with an
appropriate choice for $\gamma$. We naturally consider along this
section the Haar basis and we define for any $n$, $j_0=j_0(n)$ the integer
such that $2^{j_0}\leq  n<2^{j_0+1}$. Then $\tilde f_{n,\gamma}$ is defined with
$$\Gamma_n=\left\{\la=(j,k)\in\Lambda:\quad j\leq j_0\right\}.$$
In the sequel, we prove that, roughly speaking, $\tilde f_{n,\gamma}$ cannot achieve good performance from the oracle point of view if the parameter $\ga$ is smaller than $1$ or larger than $16$. 
\subsection{Lower bound for $\gamma$}
In this  section, we provide a lower  bound for the parameter  $\ga$. We
have the following result.
\begin{Th}\label{lower}
We estimate $f=\indic_{[0,1]}\in {\cal F}_n$ with $\tilde f_{n,\gamma}$ such that in view of (\ref{equivseuil}), we set
$$\forall\, \la\in\Gamma_n,\quad \eta_{\la,\gamma} = \sqrt{2\gamma \ln(n) \hat{V}_{\la,n}} +
\norm{\p_\la}_\infty\frac{\ln(n) u_n}{n}
,$$
with $(u_n)_n$ a deterministic bounded sequence.
Then for all $\e>0$, we obtain for any $n$,
$$\E(\normp{\tilde{f}_{n,\gamma}-f}^2)\geq \frac{1}{n^{\gamma+\e}}(1+o_n(1)).$$
\end{Th}
This result shows that we need $\gamma \geq 1$ to obtain a good
convergence rate. Indeed, for any $n$, Theorem \ref{inegoraclelavraie} (established with $\gamma>1$)  gives the bound 
$$\E(\normp{\tilde{f}_{n,\gamma}-f}^2)\leq C\frac{\ln n}{n},$$
where $C$ is a constant.
\subsection{Upper bound for $\gamma$}
In this  section, we provide an upper  bound for the parameter  $\ga$. In Remark \ref{uppga}, we have already noticed that the performances of $\tilde f_\gamma$ are worse when $\gamma$ increases. More justifications of this point are provided in this section. 
\begin{Th}\label{classFn}
Let $\ga =1+\sqrt{2}$ and let $\eta_{\la,\gamma}$ be as in Theorem \ref{inegoracle}. Then $\tilde f_{n,\gamma}$  achieves the following oracle inequality: for $n$ large enough, 
$$\sup_{f\in {\cal F}_n}\frac{\E(\normp{\tilde{f}_{n,\gamma}-f}^2)}{\sum_{\la \in \Ga_n}
\min(\be_\la^2, V_{\la,n})+\frac{1}{n}}\leq   12   \ln    n .$$
\end{Th}
Now, let us assume that for a  choice of $\ga$, say $\ga_{\min}$, the corresponding threshold $\eta_{\la,\ga_{\min}}$ leads to satisfying
results (for instance, Theorem \ref{classFn} tells us that
$\ga=1+\sqrt{2}$  is  a  good choice).  Then  let  us  fix $\ga$  larger  than
$\ga_{\min}$ and  let us  consider the estimator  $\tilde f_{n,\gamma}$  associated with
the threshold $\eta_{\la,\gamma}$ as built in Theorem \ref{inegoracle}.  Our goal is to
obtain a lower bound of the maximal risk of $\tilde f_{n,\gamma}$ on ${\cal F}_n$ larger than the upper bound obtained
for $\eta_{\la, \ga_{\min}}$. This means that choosing $\gamma$ is a bad choice. This
goal is reached in the following theorem.
 \begin{Th}\label{uppth}
Let $\ga_{\min}>1$ be fixed and let $\ga >\ga_{\min}$.  We still consider the
thresholding rule associated with $\ga$ (see Theorem \ref{inegoracle}).
Then, 
$$\sup_{f\in {\cal F}_n}\frac{\E(\normp{\tilde{f}_{n,\gamma}-f}^2)}{\sum_{\la \in \Ga_n}
\min( \beta_\la^2, V_{\la,n})+\frac{1}{n}}\geq (\sqrt{\ga}-\sqrt{\ga_{\min}})^2 2\ln n(1+o_n(1)).$$
\end{Th}
If we  choose $\ga_{\min}  = 1+\sqrt{2}$ and  apply Theorem  \ref{classFn}, the
maximal oracle ratio of the estimator $\tilde{f}_{n,\gamma}$ is not larger than $12 \ln n$.  So,  if
$\ga>16$,  which  yields  $(\sqrt{\ga}-\sqrt{\ga_{\min}})^2>6$,  the  resulting
maximal oracle ratio of $\tilde{f}_{n,\gamma}$
is larger than $12 \ln n$. In addition, note that the function used in Theorem
\ref{lower} is also in ${\cal F}_n$. So, finally the convenient value of $\gamma$
belongs to $[1,16]$.
\section{Simulations}\label{simu}
In this  section, some  simulations are provided  and the performances  of the
thresholding  rule are  measured from  the numerical  point of  view.  We also
discuss the ideal  choice for the parameter $\gamma$ keeping in mind that the
value $\gamma=1$ constitutes a border for the theoretical results (see Section
\ref{penaltyterm}). For these purposes, the procedure is performed for estimating various intensity
signals and the wavelet set-up associated with biorthogonal
wavelet bases is considered. More precisely, we focus either on the Haar basis where
$$\phi=\tilde\phi=\indic_{[0,1]},\quad
\psi=\tilde\psi=\indic_{[0,1/2]}-\indic_{]1/2,1]}$$or  on a  special  case of
spline systems given in Figure \ref{fig-basespline}. 
\begin{figure}[t]
\begin{center}
\includegraphics[width=0.7\linewidth,angle=0]{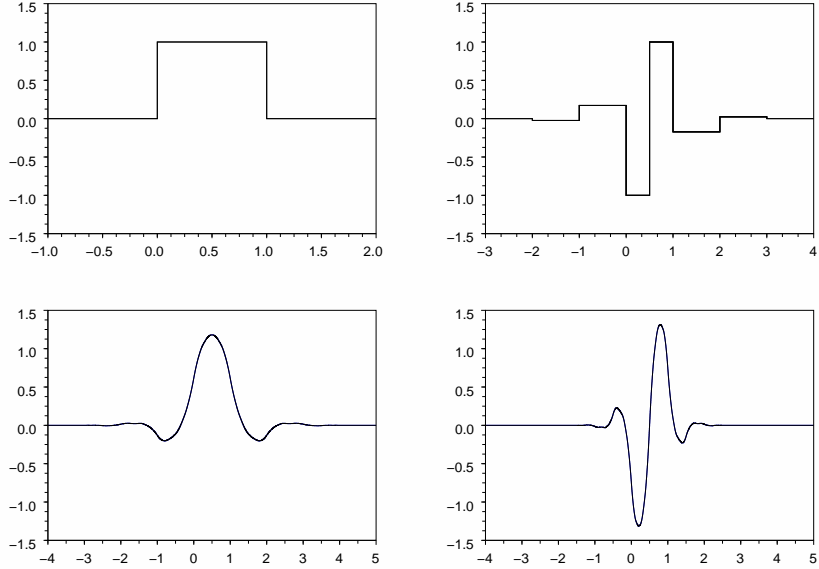}
\end{center}
\caption{The spline basis.  Top: $\phi$ and $\psi$, Bottom: $\tilde\phi$ and $\tilde\psi$}\label{fig-basespline}
\end{figure}
This latter basis, called hereafter the spline
basis,  has  the following  properties.  First,  the  support  of  $\phi$,  $\psi$,
$\tilde\phi$ and $\tilde\psi$ is included in $[-4,5]$. The reconstruction wavelets $\tilde\phi$ and $\tilde\psi$
belong to  $C^{1.272}$. Finally,  the wavelet  $\psi$ is  a piecewise
constant function orthogonal to polynomials of degree 4 (see \cite{don}). So, such
a basis has properties 1--5 required in Section 
\ref{biorthogonal} with $m=0.272$.  Then, the signal $f$ to be estimated is decomposed
as follows:
$$f=\sum_{\la\in\Lambda}\be_\la\tilde\p_\la=\sum_{k\in\Z}\be_{-1,k}\tilde\phi_{k}+\sum_{j\geq
0}\sum_{k\in\Z}\be_{j,k}\tilde\psi_{j,k}.$$
For  estimating $f$, we  use the  observations $(\hat\be_\la)_{\la\in\Lambda}$
associated with a Poisson process $N$ whose
intensity with respect to the Lebesgue measure is $n\times f$.
Since $\phi$ and $\psi$ are piecewise constant functions, accurate values of the observations are available, which allows to avoid many computational and approximation issues that often arise in the wavelet setting.
 To shed light on typical aspects of Poisson intensity estimation,
Figure \ref{fig-observations}
displays the reconstruction  obtained by using only the  coarsest noisy wavelet
coefficients of  a particular signal  (the density of a Gaussian
variable with mean 0.5 and standard deviation 0.25) with $n=4096$. We mean that $(\be_{j,k})_{j\geq -1,k\in\Z}$ is estimated
by $(\hat\be_{j,k})_{-1\leq j\leq 10,k\in\Z}$ without using thresholding. 
\begin{figure}[t]
\begin{center}
\includegraphics[width=0.7\linewidth,angle=0]{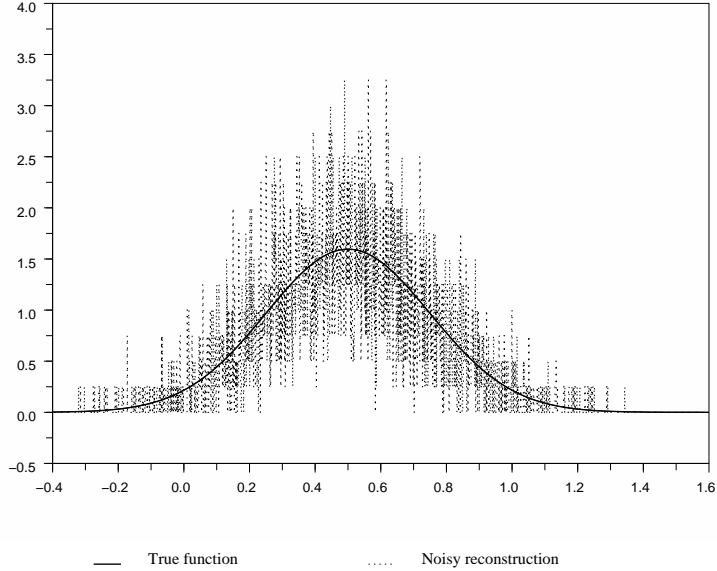}
\end{center}
\caption{Plots of the signal $f(x)=\frac{1}{0.25\sqrt{2\pi}}
\exp\left(\frac{(x-0.5)^2}{2\times 0.25^2}\right)$ and purely noisy
reconstruction with $n=4096$ based on the wavelet coefficients until the level 10 and by using the Haar basis.}\label{fig-observations}
\end{figure}
As  expected, variability  highly depends  on the  local
values of the signal.  So, our framework is very different from classical regression where we observe random variables with common
variance.  The thresholding rule considered in this
section      is      $\tilde      f_\gamma=(\tilde{f}_{n,\gamma})_n$      with
$\tilde{f}_{n,\gamma}$ defined in (\ref{defest}) with 
$$\Gamma_n=\left\{\la=(j,k):\quad -1\leq j\leq j_0,\ k\in\Z\right\}$$
and 
$$\eta_{\la,\gamma}=\sqrt{2\gamma\ln (n) \hat{V}_{\la,n} }+\frac{\gamma\ln
    n}{3n}\norm{\p_\la}_\infty.$$
Observe that  $\eta_{\la,\gamma}$ slightly differs from  the threshold defined
in    (\ref{defseuil})   since    $\tilde{V}_{\la,n}$    is now  replaced    with
$\hat{V}_{\la,n}$.   Such    a   modification   is   natural    in   view   of
(\ref{equivseuil}) and Theorem \ref{lower}. In particular, it allows to derive the parameter $\gamma$ as an explicit function of the
threshold. We guess that the performances of our
thresholding rule associated with the threshold $\eta_{\la,\gamma}$ defined in
(\ref{defseuil}) are very close.  Now, to complete the definition of the
estimate, we  have to  choose the parameters  $j_0$ and $\gamma$.  This choice
is  capital  and  is  extensively  discussed in  the  sequel.  Using  $n=1024$,  Figure \ref{fig-reconstruction}  displays 9  examples  of intensity
reconstructions obtained with $j_0=\log_2(n)=10$ and $\gamma=1$.  These
functions  are  respectively   denoted  'Haar1',  'Haar2',  'Blocks',  'Comb',
'Gauss1', 'Gauss2', 'Beta0.5', 'Beta4' and 'Bumps' and
have been  chosen to represent the  wide variety of signals  arising in signal
processing (see the Appendix for a precise definition of each signal).  Each  of them  satisfies  $\norm{f}_1=1$  and  can be  classified
according to the following criteria: the smoothness, the size of the support
(finite/infinite), the value of the sup norm (finite/infinite) and the shape (to be
piecewise constant  or a mixture of  peaks). In particular,  the signal 'Comb'
(respectively 'Beta0.5') is inspired  by  the construction  of  the  counter-example  proposed in  Theorem
\ref{contrexem} (respectively Proposition \ref{contreexlinfini}). 
\begin{figure}[tpb]
\begin{center}
\includegraphics[width=1.2\linewidth,angle=-90]{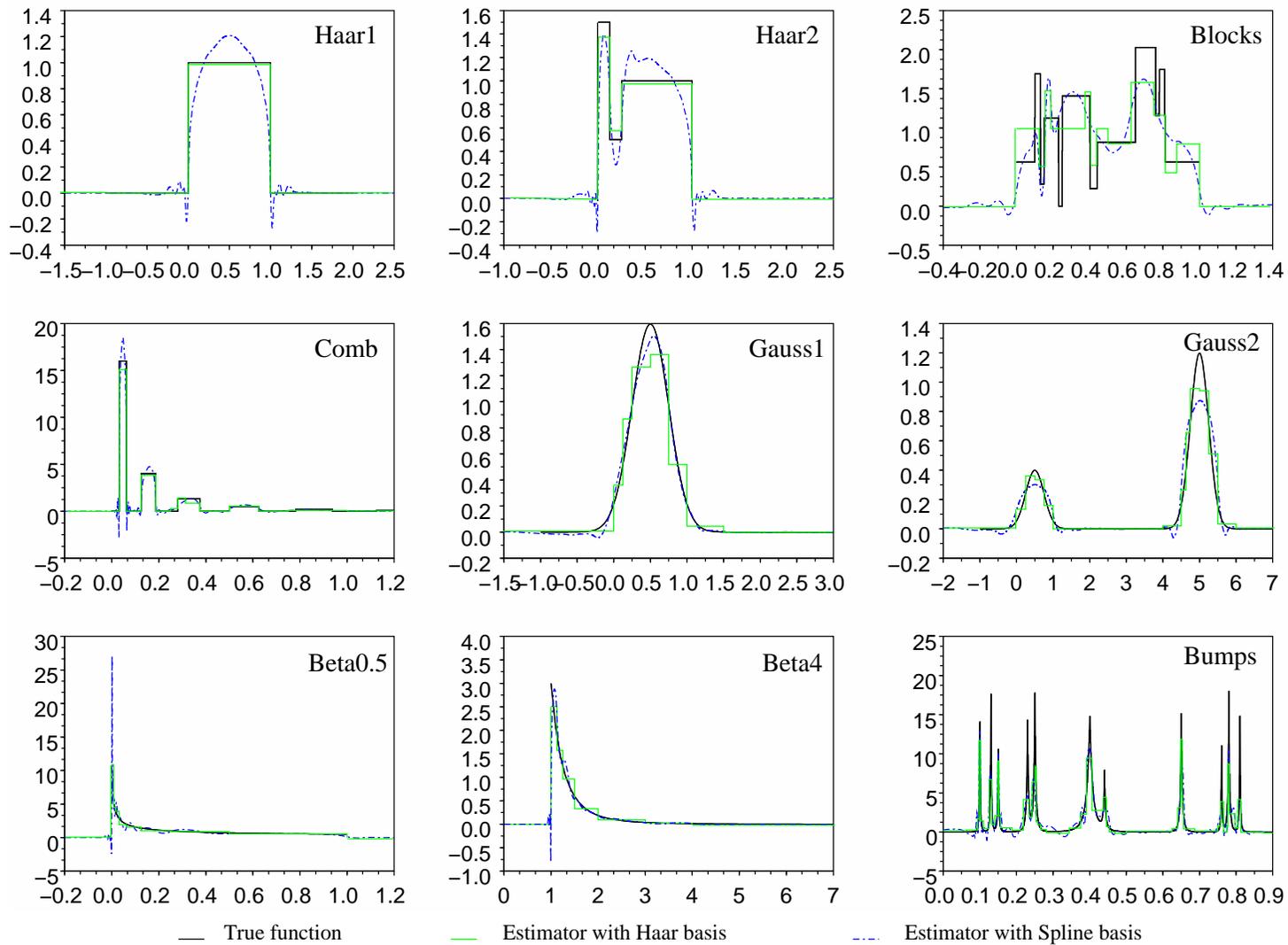}
\end{center}
\caption{Reconstructions  by  using  the  Haar   and  the  spline  bases  of  9
signals with $n=1024$, $j_0=10$ and $\gamma=1$. Top: 'Haar1', 'Haar2', 'Blocks'; Middle: 'Comb', 'Gauss1', 'Gauss2';
Bottom: 'Beta0.5', 'Beta4', 'Bumps'}\label{fig-reconstruction}
\end{figure}

More interestingly, numerical results are provided to answer the question about the choice of
$\gamma$.  Given $n$ and a function $f$, we denote 
$R_n(\gamma)$ the ratio between the $\ell_2$-performance of our
procedure (depending on $\gamma$) and the oracle risk where the wavelet coefficients at levels $j>j_0$
are omitted. We have:
$$
R_n(\gamma)=\frac{\sum_{\la\in\Gamma_n}(\tilde\be_\la-\be_\la)^2}{\sum_{\la\in\Gamma_n}\min(\be_\la^2,V_{\la,n})}=\frac{\sum_{\la\in\Gamma_n}(\hb_\la \indic_{|\hb_\la|\geq\eta_{\la,\gamma}}-\be_\la)^2}{\sum_{\la\in\Gamma_n}\min(\be_\la^2,V_{\la,n})}. 
$$
 Of course, $R_n$ is a stepwise function and the change points of $R_n$
correspond to  the values  of $\gamma$ such  that there exists  $\lambda$ with
$\eta_{\la,\gamma}=|\hat\be_\la|$.  
The average over 1000 simulations of $R_n(\gamma)$ is computed
providing an estimation of $\E(R_n(\gamma))$. 
This average ratio, denoted $\overline{R_n}(\gamma)$ and viewed as a function of $\gamma$, is plotted for three signals
'Haar1', 'Gauss1' and 'Bumps' for $n\in\{64,128,256,512,1024,2048,4096\}$. For non compactly supported signals, to compute the ratio, the wavelet
coefficients associated with the tails of the signals are omitted but we ensure
that this approximation is negligible with respect to the values of $R_n$.
The parameter $j_0$ takes the value 
$j_0=\log_2(n)$.   Fixing  $j_0=\log_2(n)$  is  natural  in  view  of  Theorem
\ref{inegoraclelavraie} (applied with $c=1$ and $c'=0$) and Theorem
\ref{lower}.
Figure  \ref{fig-Haar1}     displays
$\overline{R_n}$ for 'Haar1' decomposed on the Haar basis. The left side of Figure
\ref{fig-Haar1} gives a  general idea of the shape  of $\overline{R_n}$, while the
right side focuses on small values of $\gamma$.  
\begin{figure}[tbp]
\begin{minipage}[c]{.56\linewidth}
\hspace{-0.7cm}\includegraphics[width=\linewidth,angle=0]{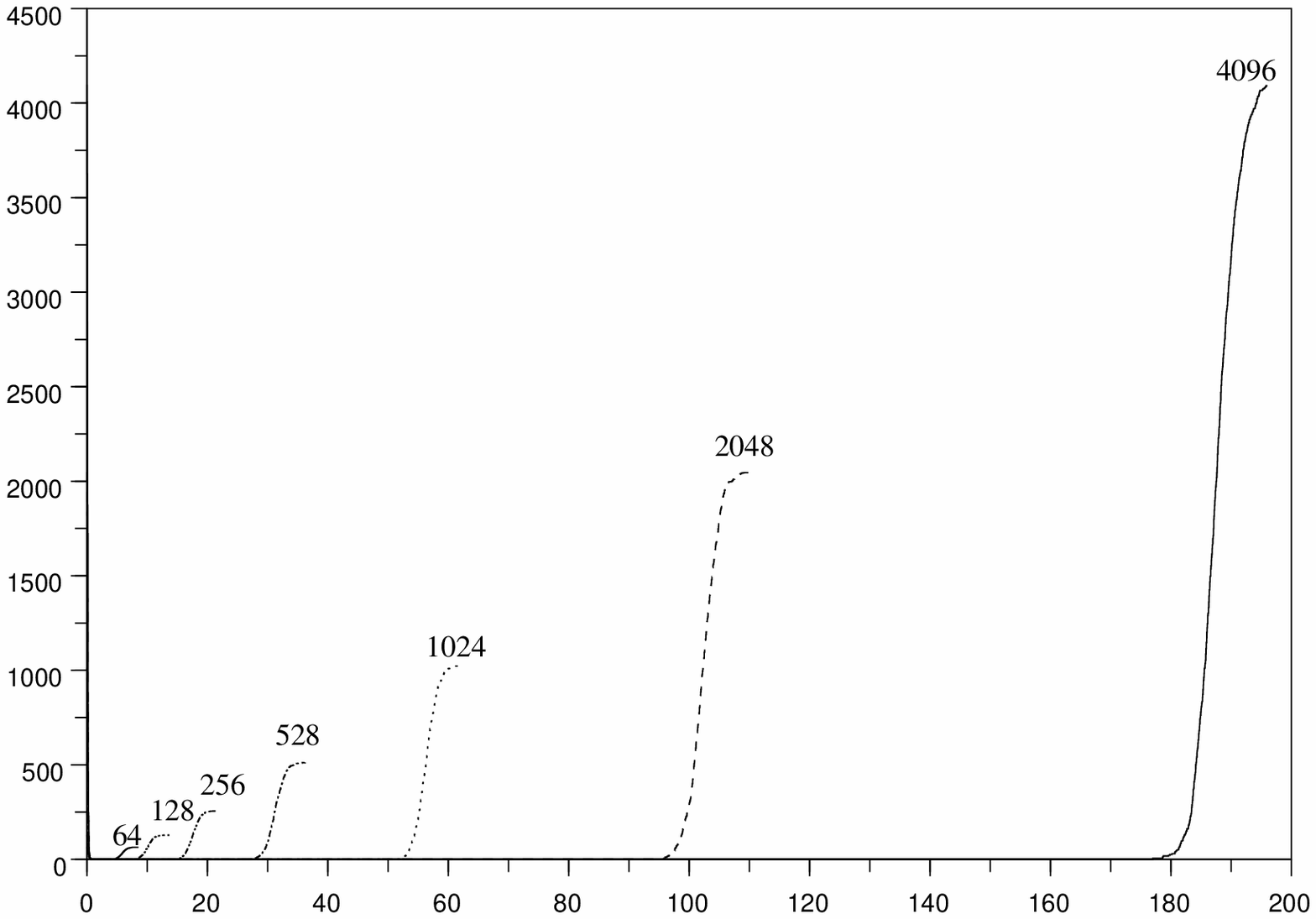}
\end{minipage} \hfill
\hspace{-1.5cm} \begin{minipage}[c]{.56\linewidth}
\includegraphics[width=\linewidth,angle=0]{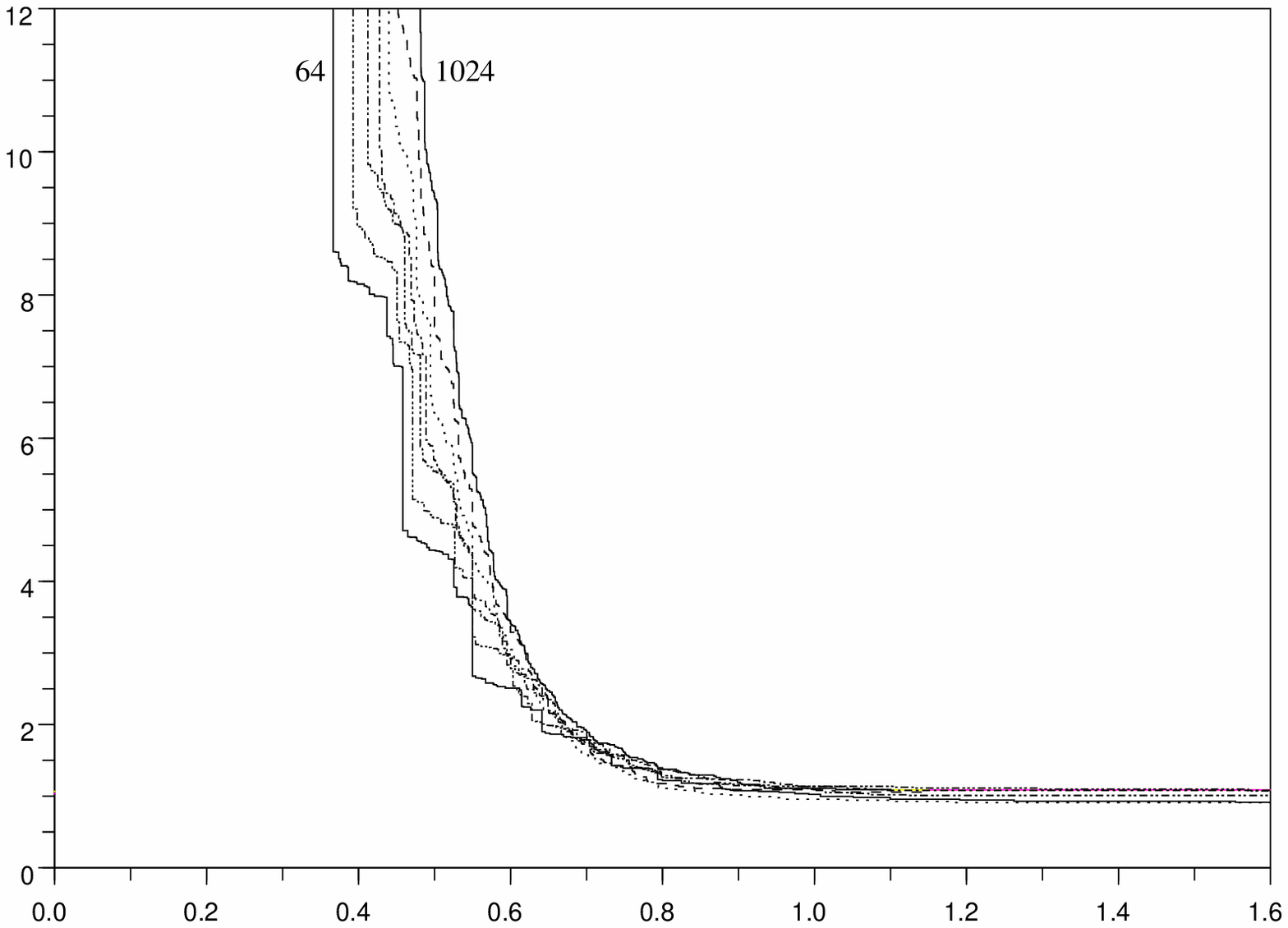}
\end{minipage}
\caption{The  function  $\gamma\to  \overline{R_n}(\gamma)$  at  two  scales  for
'Haar1'      decomposed      on      the      Haar     basis      and      for
$n\in\{64,128,256,512,1024,2048,4096\}$ with $j_0=\log_2(n)$.}\label{fig-Haar1}
\end{figure}
Similarly, Figures \ref{fig-Gauss1} and 
 \ref{fig-Bumps} display  $\overline{R_n}$ for 'Gauss1' decomposed on
the spline basis and for 'Bumps' decomposed on
the Haar and the spline bases. 
\begin{figure}[]
\begin{center}
\includegraphics[width=0.7\linewidth,angle=-0]{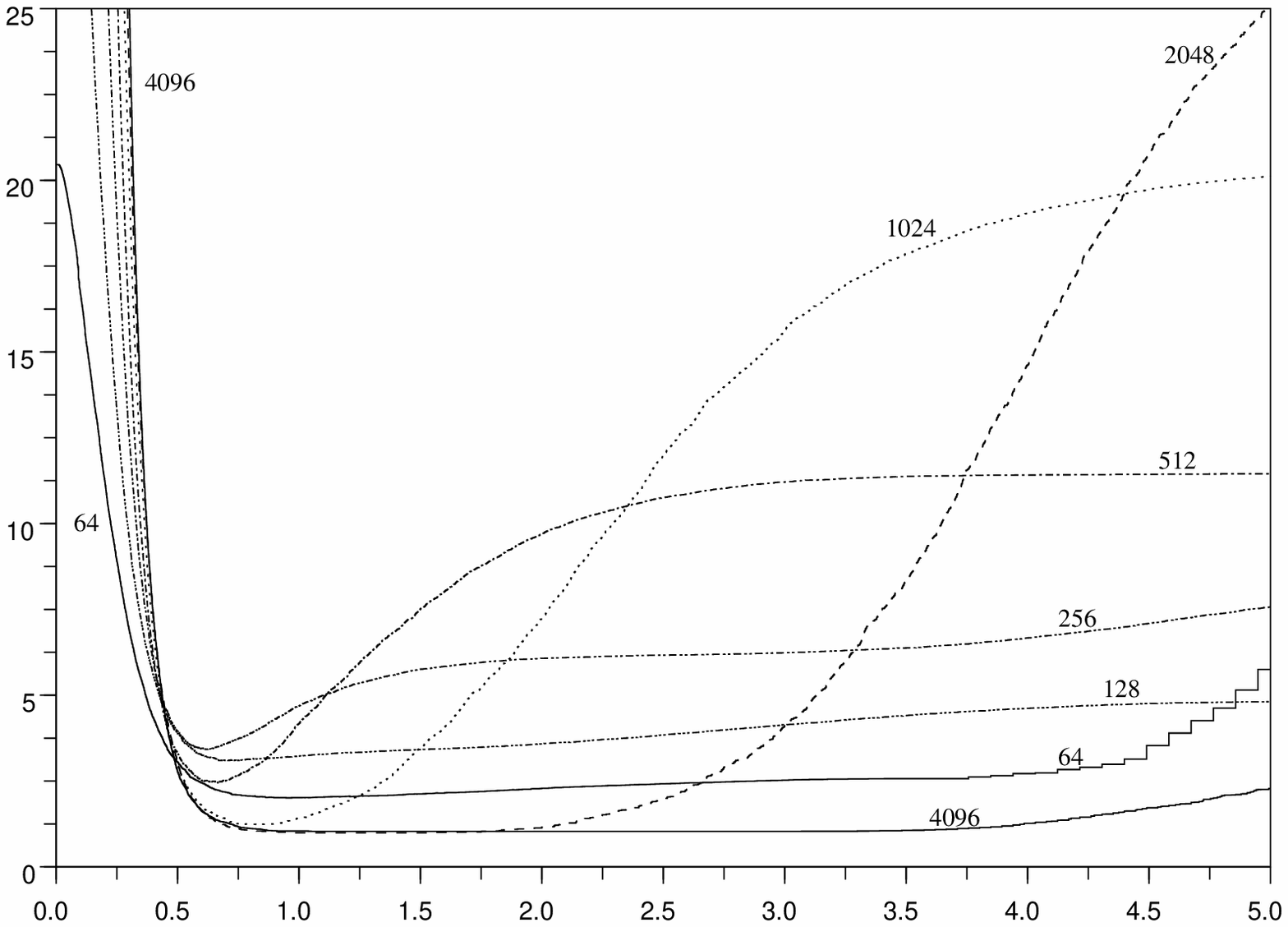}
\end{center}
\caption{The  function  $\gamma\to  \overline{R_n}(\gamma)$  for
'Gauss1'      decomposed      on      the      spline     basis      and      for
$n\in\{64,128,256,512,1024,2048,4096\}$ with $j_0=\log_2(n)$.}\label{fig-Gauss1}
\end{figure}
\begin{figure}[tbp]
\begin{minipage}[c]{.56\linewidth}
\hspace{-0.7cm}\includegraphics[width=\linewidth,angle=0]{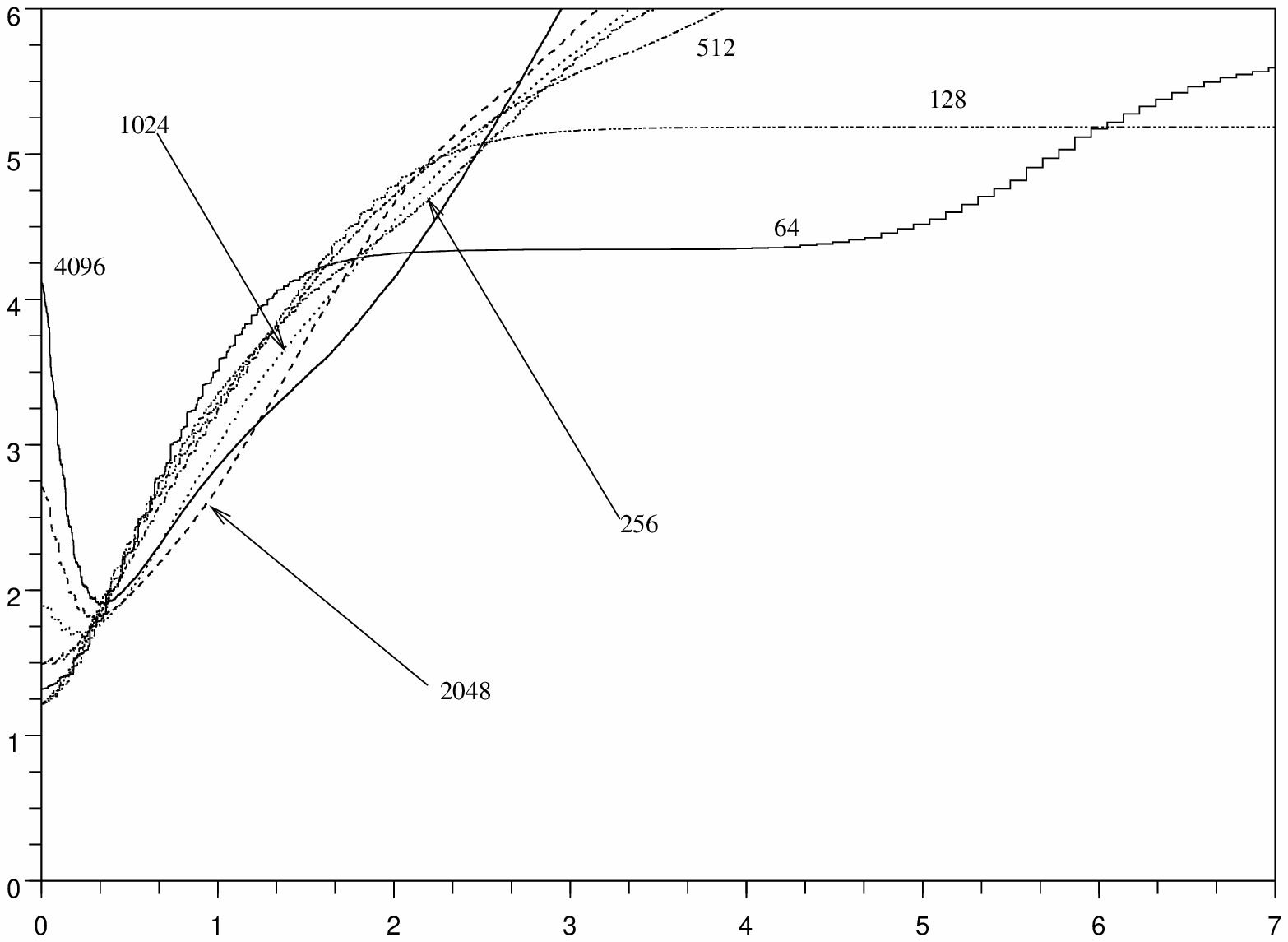}
\end{minipage} \hfill
\hspace{-1.5cm} \begin{minipage}[c]{.56\linewidth}
\includegraphics[width=\linewidth,angle=0]{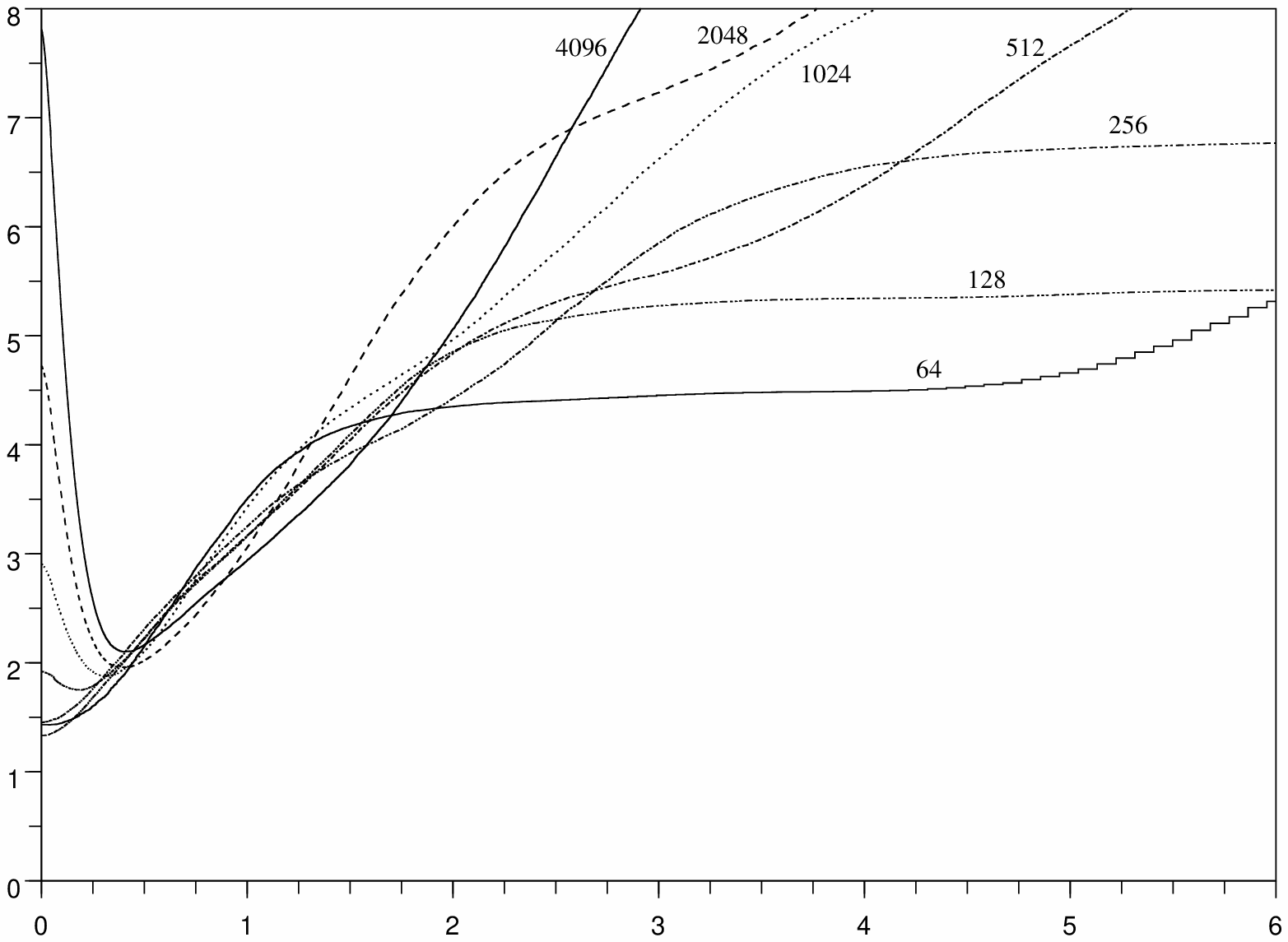}
\end{minipage}
\caption{The  function  $\gamma\to  \overline{R_n}(\gamma)$  for
'Bumps' decomposed on the Haar and the spline bases and for
$n\in\{64,128,256,512,1024,2048,4096\}$ with $j_0=\log_2(n)$.}\label{fig-Bumps}
\end{figure}

To discuss our results, we introduce
$$\gamma_{\min}(n)=\argmin_{\gamma>0}\overline{R_n}(\gamma).$$ 
 
For  'Haar1',
$\gamma_{\min}(n)\geq 1$ for any value of $n$ and taking
$\gamma<1$ deteriorates the performances of  the estimate.  Such a result was
established from the theoretical point of view in Theorem \ref{lower}. In fact,  Figure
\ref{fig-Haar1} allows to draw the following major
conclusion for 'Haar1':
\begin{equation}\label{super}
\overline{R_n}(\gamma)\approx\overline{R_n}(\gamma_{\min})\approx
1
\end{equation}
 for a wide range of $\gamma$  around $\gamma_{\min}>1$ that contains $\gamma=1$. For instance, when $n=4096$, the minimum of $\overline{R_n}$, close to 1,  is very flat and the minimizer is surrounded by the "plateau" $[1,177]$. 
 So, the values  of  $\gamma_{\min}(n)$ should  not  be considered  as
sacred. Our thresholding rule with
$\gamma=1$ performs very well since it achieves the same performance as the oracle estimator.

For 
'Gauss1', $\gamma_{\min}(n)\geq 0.5$ for any value of $n$. Moreover, as soon
as $n$ is large enough, the oracle ratio at $\gamma_{\min}$ is of order $1$. Besides, when $n\geq 2048$,
as for  'Haar1', $\gamma_{\min}(n)$  is larger than  $1$. We observe the
``plateau phenomenon'' as well and as for 'Haar1', the size of the plateau increases when $n$ increases.  This can be
explained by the following important  property of 'Gauss1'. 'Gauss1' can be well approximated by a
finite combination of the atoms of the spline basis. So, we have the strong
impression  that  the  asymptotic  result  of  Theorem  \ref{lower}  could  be
generalized for  the spline  basis as  soon as we  can build  positive signals
decomposed on the spline basis.  

Conclusions for 'Bumps' are very different. Remark that this irregular signal
has many significant wavelet coefficients at
high resolution levels whatever  the basis. We have $\gamma_{\min}(n)<0.5$ for
each  value of  $n$. Besides,  $\gamma_{\min}(n)\approx 0$  when  $n\leq 256$,
meaning that all the coefficients until  $j=j_0$ have to be kept to obtain the
best estimate. So, the parameter $j_0$ plays an essential role and has to be
well  calibrated   to  ensure  that   there  are  no   non-negligible  wavelet
coefficients for  $j>j_0$.  Other differences  between Figure \ref{fig-Haar1}
(or Figure \ref{fig-Gauss1}) and Figure \ref{fig-Bumps} have to be emphasized.  For
'Bumps', when $n\geq 512$, the minimum of $\overline{R_n}$ is well localized,
there is no plateau anymore and $\overline{R_n}(1)>2$ ($\overline{R_n}(\gamma_{\min}(n))$ is larger than 1). 

As  a  preliminary conclusion,  it  seems  that  the ideal  choice  of  $\gamma$  and  the
performance of  the thresholding rule  highly depend on the  decomposition of
the signal on the wavelet basis. Hence, in the sequel, we have decided to force $j_0=10$ so
that the decomposition on the basis is not too rough. To extend previous results  and for the sake
of exhaustiveness Figures
\ref{tous-1} and  \ref{tous-2} display the  average of the function  $R_n$ for
the signals 'Haar1',  'Haar2',  'Blocks',  'Comb',
'Gauss1', 'Gauss2', 'Beta0.5', 'Beta4' and 'Bumps' with $j_0=10$. For brevity, we only consider the values
$n\in\{64,256,1024,4096\}$ and the average of
$R_n$ is performed over 100 simulations. Note also that we fix $j_0=10$
and $100$ simulations (and not larger parameters) because  computational
difficulties arise when we deal with infinite support for heavy-tailed signals
('Beta4' and 'Comb') and for a wide range of $\gamma$. 
 Figure \ref{tous-1} gives the results
obtained for the  Haar basis and Figure \ref{tous-2} for  the spline basis. 
\begin{figure}[tbp]
\begin{center}
\includegraphics[width=0.75\linewidth,angle=-0]{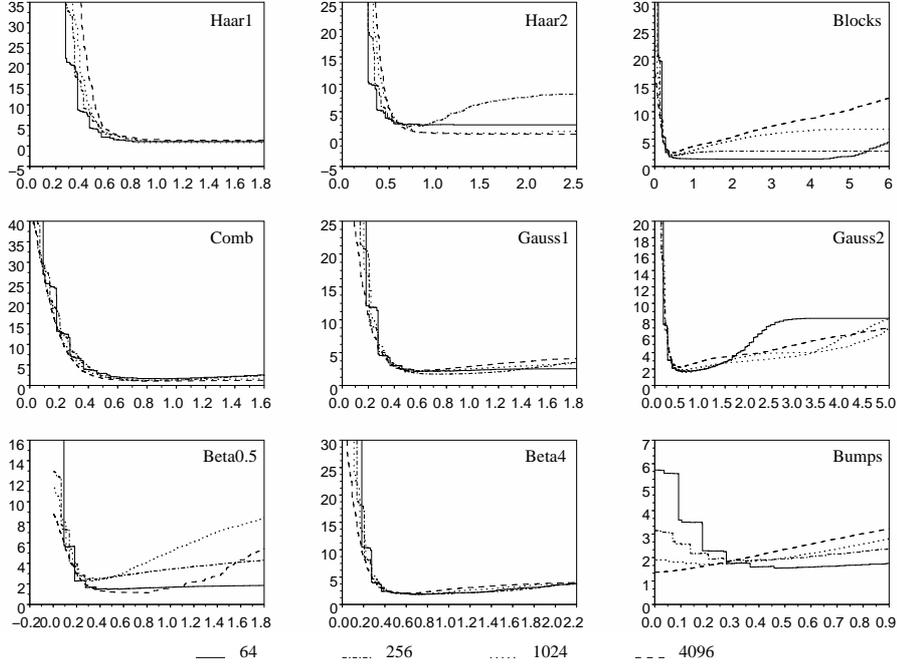}
\end{center}
\caption{Average over  100 iterations of the  function $R_n$
for signals decomposed on the Haar basis and for
$n\in\{64,256,1024,4096\}$ with $j_0=10$.}\label{tous-1}
\end{figure}
\begin{figure}[tbp]
\begin{center}
\includegraphics[width=0.75\linewidth,angle=-0]{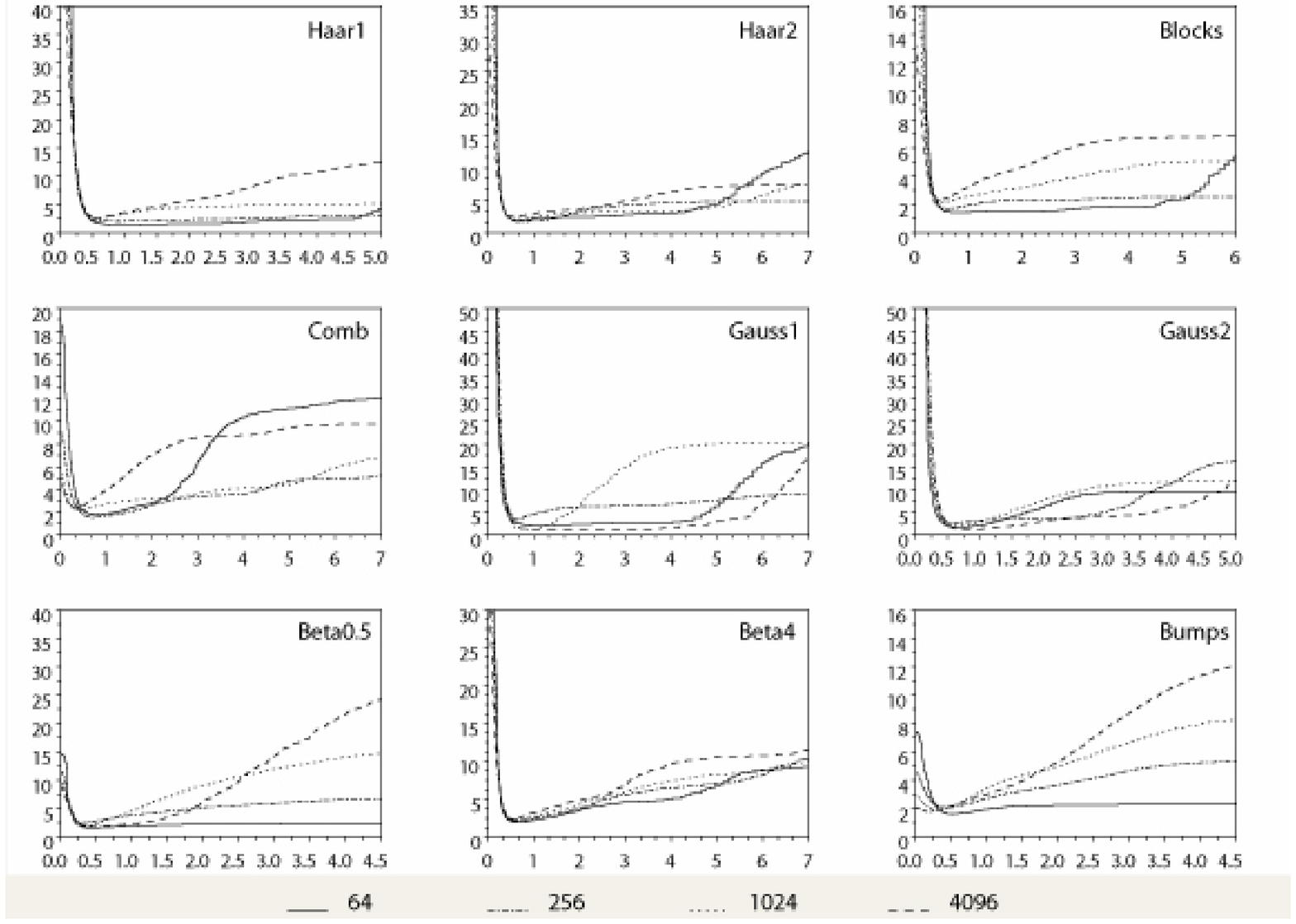}
\end{center}
\caption{Average over  100 iterations of the  function $R_n$
for signals decomposed on the spline basis and for
$n\in\{64,256,1024,4096\}$ with $j_0=10$.}\label{tous-2}
\end{figure}
To interpret the results, we introduce
$$R_n^{\log}(\gamma)=\frac{\sum_{\la\in\Gamma_n}(\tilde\be_\la-\be_\la)^2}{\sum_{\la\in\Gamma_n}\min(\be_\la^2,V_{\la,n}\log(n))}=\frac{\sum_{\la\in\Gamma_n}(\hb_\la \indic_{|\hb_\la|\geq\eta_{\la,\gamma}}-\be_\la)^2}{\sum_{\la\in\Gamma_n}\min(\be_\la^2,V_{\la,n}\log(n))},$$
where   the   denominator   appears    in   the   upper   bound   of   Theorem
\ref{inegoraclelavraie}.  We  also  measure  the $\ell_2$-performance  of  the
estimator by using
$$r_n(\gamma)=\sum_{\la\in\Gamma_n}(\tilde\be_\la-\be_\la)^2=\sum_{\la\in\Gamma_n}(\hb_\la \indic_{|\hb_\la|\geq\eta_{\la,\gamma}}-\be_\la)^2.$$
Table \ref{tableau} gives, for
each signal and for $n\in \{64,256,2048,4096\}$, the average of
$r_n (1)$, denoted $\overline{r_n}(1)$, the average of 
${R_n}(1)$ , denoted
$\overline{R_n}(1)$ and the average of $R_n^{\log}(1)$, denoted
$\overline{R_n^{\log}}(1)$ (100 simulations are performed).
In  view  of  Table  \ref{tableau},  let us introduce two classes of functions. The first class is the class of signals that are
well approximated  by a  finite combination  of the atoms  of the  basis (it
contains 'Haar1', 'Haar2' and 'Comb' for the Haar basis and 'Gauss1' and 'Gauss2' for the spline
basis).  For such  signals, the estimation problem  is close  to  a parametric
problem  and in  this  case the  performance  of the  oracle  estimate can  be
achieved at least for $n$ large enough and (\ref{super}) is true for a wide
range of $\gamma$ around $\gamma_{\min}$ that contains $\gamma=1$. The second class is the class of irregular signals with significant wavelet coefficients at high
resolution levels (it contains all the other cases except 'Beta0.5').  For  such  signals,  Table   \ref{tableau}  shows  that
$\overline{R_n(1)}$ seems to increase with $n$. But $\overline{R_n^{\log}(1)}$ remains constant, showing
that the  upper bound (with the logarithmic term) of Theorem
\ref{inegoraclelavraie} is probably achieved up to a constant. 'Beta0.5' has
only one significant coefficient at each level. This may explain why its behavior seems to be  between the first and
second class behavior. Finally let us note that the oracle ratio curve for
'Bumps', $j_0=10$ and $n=4096$ has a minimizer $\gamma_{\min}$ close to $0$
and has a
different behavior from the one with $j_0=12$ (see Figure \ref{fig-Bumps} ). It
 illustrates again the fact that 'Bumps' has still some important
coefficients at the level of resolution $j_0=12$ that can be taken into account if $\log_2(n)=12$.

\begin{table}[htbp]
\vspace{-0.1cm}\hspace{-2.5cm}
{\tiny 
\begin{center}
\begin{tabular}[]{|c|c|c|c|c|c|c|c|}
\cline{3-8}
\multicolumn{2}{r|}{}&\multicolumn{3}{c|}{}&\multicolumn{3}{c|}{}\\
\multicolumn{2}{r|}{}&\multicolumn{3}{c|}{Haar}&\multicolumn{3}{c|}{Spline}\\
\multicolumn{2}{r|}{}&\multicolumn{3}{c|}{}&\multicolumn{3}{c|}{}\\
\cline{2-8}
\multicolumn{1}{r|}{}&&&&&&&\\
\multicolumn{1}{r|}{}&$n$& $\overline{r_n}(1)$  & $\overline{R_n}(1)$ & $\overline{R_n^{\log}}(1)$ & $\overline{r_n}(1)$  & $\overline{R_n}(1)$ & $\overline{R_n^{\log}}(1)$ \\
\multicolumn{1}{r|}{}&&&&&&&\\
\hline 
&&&&&&&\\
 & 64 &0.016 &1.0 &0.2&0.10& 1.4&0.7\\
&&&&&&&\\
&256  & 0.0042& 1.1&0.2&0.068& 2.0&0.8\\
Haar1&& &&& &&\\
&1024  &0.0008 &0.8 &0.1&0.042&3.3 &0.9\\
&&&&&&&\\
&4096  &0.0002 &1.0 &0.2&0.016&3.5 &0.7\\
&&&&&&&\\
\hline
&&&&&&&\\
 & 64 &0.082 &2.6 &0.6&0.21&2.1 &1.0\\
&&&&&&&\\
&256  &0.026 &3.3 &0.6&0.085& 1.8&0.7\\
Haar2&& &&& &&\\
&1024  &0.0023 &1.2 &0.2&0.053& 2.4&0.9\\
&&&&&&&\\
&4096  &0.0004 &1.0 &0.1&0.026&2.9 &0.8\\
&&&&&&&\\
\hline
&&&&&&&\\
 & 64 &0.31 &1.4 &0.9&0.27&1.4 &0.9\\
&&&&&&&\\
&256  &0.26 & 2.5&1.0&0.21& 1.9&1.0\\
Blocks && &&& &&\\
&1024  &0.13 &2.9 &0.9&0.13& 2.6&0.9\\
&&&&&&&\\
&4096  &0.053 &3.7 &0.8&0.063&3.2 &0.8\\
&&&&&&&\\
\hline
&&&&&&&\\
 & 64 &0.61 &1.7 &0.4&1.71&1.8 &0.8\\
&&&&&&&\\
&256  &0.12 & 1.3&0.2&0.78&1.7 &0.7\\
Comb&& &&& &&\\
&1024  &0.032 & 1.4&0.2&0.52&2.7 &0.8\\
&&&&&&&\\
&4096  & 0.0063& 1.1&0.1&0.23& 4.0&0.7\\
&&&&&&&\\
\hline
&&&&&&&\\
 & 64 &0.21 &2.3 &0.9&0.10&2.1 &0.7\\
&&&&&&&\\
&256  &0.072 & 1.8&0.7&0.060&4.5 &0.9\\
Gauss1&& &&& &&\\
&1024  &0.039 &2.6 &0.7&0.0048&1.2 &0.2\\
&&&&&&&\\
&4096  & 0.018& 2.9&0.7&0.0017&1.2 &0.2\\
&&&&&&&\\
\hline 
&&&&&&&\\
 & 64 &0.17 &1.9 &0.7&0.12&2.1 &0.7\\
&&&&&&&\\
&256  &0.07 &2.0 &0.6&0.05&3.1 &0.6\\
Gauss2&& &&& &&\\
&1024  &0.031 &2.3 &0.6&0.012& 2.8&0.4\\
&&&&&&&\\
&4096  & 0.015&3.0 &0.7&0.0017& 1.2&0.2\\
&&&&&&&\\
\hline
&&&&&&&\\
 & 64 &1.6 &1.7 &1.0&2.2&1.9 &1.0\\
&&&&&&&\\
&256  & 1.1& 3.4&1.0&1.4&3.8 &1.0\\
Beta0.5&& &&& &&\\
&1024  &0.45 & 5.1&0.8&0.51&4.6 &0.8\\
&&&&&&&\\
&4096  & 0.045&1.6 &0.3&0.066&2.3 &0.3\\
&&&&&&&\\
\hline 
&&&&&&&\\
 & 64 &0.25 &2.1 &0.8&0.36&2.2 &0.9\\
&&&&&&&\\
&256  &0.093 &2.0 &0.6&0.16& 2.5&0.8\\
Beta4&& &&& &&\\
&1024  &0.041 &2.2 &0.6&0.061&2.7 &0.7\\
&&&&&&&\\
&4096  &0.020 &2.8 &0.7&0.024&3.3 &0.6\\
&&&&&&&\\
\hline 
&&&&&&&\\
 & 64 &4.9 & 1.8&1.0&4.3&2.0 &1.1\\
&&&&&&&\\
&256  &3.1 & 2.5&1.0&2.5&2.7 &1.0\\
Bumps&& &&& &&\\
&1024  &1.5 &3.0 &0.9&1.2&3.4 &0.9\\
&&&&&&&\\
&4096  &0.62 & 3.4&0.7&0.38&3.0 &0.6\\
&&&&&&&\\
\hline
\end{tabular}
\end{center}
}
\caption{Values of  $\overline{r_n}(1)$,  $\overline{R_n}(1)$ and $\overline{R_n^{\log}}(1)$ for each signal decomposed on the Haar basis or the spline basis and for $n\in\{64,256,1024,4096\}$.}\label{tableau}
\end{table}
Finally, we
would like to emphasize the following conclusions. Performances of our
thresholding rule are suitable since the ratio $\overline{R_n}(1)$ is controlled. Moreover a convenient choice of the basis improves this ratio
but also the performances of the estimator itself. Furthermore,
the size of the support does not play any role (compare estimation of 'Comb'
and  'Haar1' for instance)  and the  estimate $\tilde  f_{n,1}$ performs
well for recovering the size and location of peaks. 
\section{Proofs}\label{proofs}
In this  section, the notation  $\square$  represents an  absolute constant
whose value may  change at each  line. For any  $x>0$, the notation  $\lceil x\rceil$
denotes the smallest integer larger than $x$. Notations of Sections \ref{main}
and \ref{biorthogonal} are used. Recall also that we have set
$$\forall\;\la\in\Lambda,\quad F_\la=\int_{supp(\p_\la)}f(x)dx.$$
\subsection{Proof of Theorem \ref{inegoracle}}
Let $\ga, p, q, \e$ be as in Theorem \ref{inegoracle}. 
We start as usual for model selection with (\ref{etoile2}). One has for all subset $m$ of $\Ga_n$
$$\ga_n(\tilde{f}_{n,\gamma}) +\pen(\hat{m})\leq \ga_n(\hat{f}_m)
+\pen(m).$$
If $g=\sum_{\la\in\La} \al_\la\tp_\la$, setting $\nu_n(g)= \sum_{\la \in
  \La}\al_\la(\hb_\la-\be_\la)$, we obtain that
$$\ga_n(g)=\normp{g-f}^2-\normp{f}^2 - 2\nu_n(g).$$
Hence, 
\begin{equation*}
\label{depart}
\normp{\tilde{f}_{n,\gamma}-f}^2\leq \normp{\hat{f}_m-f}^2 +
2\nu_n(\tilde{f}_{n,\gamma}-\hat{f}_m)+\pen(m)-\pen(\hat{m}).
\end{equation*}
For any subset of indices $m'$, let 
$\chi(m')=\sqrt{\sum_{\la\in m'}(\hb_\la-\be_\la)^2} $ and let  $f_m=\sum_{\la \in m}\be_\la \tp_\la$ be the orthogonal projection of $f$ on $S_m$ for $\normp{.}$.
Then $\chi^2(m)=\nu_n(\hat{f}_m-f_m)=\normp{\hat{f}_m-f_m}^2=\normp{\hat{f}_m-f}^2-\normp{f_m-f}^2.$ Hence,
\begin{equation*}
\label{depart2}
\normp{\tilde{f}_{n,\gamma}-f}^2\leq \normp{f_m-f}^2-\chi^2(m)+
2\nu_n(\tilde{f}_{n,\gamma}-f_m)+\pen(m)-\pen(\hat{m}).
\end{equation*}
Furthermore,
$$\nu_n(\tilde{f}_{n,\gamma}-f_m) \leq \normp{\tilde{f}_{n,\gamma}-f_m}\chi(m\cup \hat{m}) \leq
\normp{\tilde{f}_{n,\gamma}-f}\chi(m\cup \hat{m})+\normp{f_m-f}\chi(m\cup \hat{m}) .$$
Using twice the fact that $2ab\leq\theta a^2+\theta^{-1} b^2$, for $\theta =
2/(2+\e)$ and $\theta =
2/\e$, we obtain that
$$2\nu_n(\tilde{f}_{n,\gamma}-f_m) \leq \frac{2}{2+\e}
\normp{\tilde{f}_{n,\gamma}-f}^2+\frac2\e\normp{f_m-f}^2+(1+\e)\chi^2(m\cup\hat{m}).$$
Hence we obtain that
$$\frac{\e}{2+\e}\normp{\tilde{f}_{n,\gamma}-f}^2\leq\left(1+\frac{2}{\e}\right)\sum_{\la \not\in m}\be_\la^2 +(1+\e)\chi^2(m\cup\hat{m})-\chi^2(m)+\pen(m)-\pen(\hat{m}).$$
But $\chi^2(m\cup\hat{m})\leq \chi^2(m)+\chi^2(\hat{m}).$ 
After integration it remains to control 
$$\mathcal{A}=\E((1+\e)\chi^2(\hat{m})-\pen(\hat{m})).$$
Since 
$$\hat m=\left\{\la\in \Gamma_n:\quad |\hat\be_\la|\geq\eta_{\la,\gamma}\right\},$$
we have
$$\mathcal{A}=\sum_{\la\in\Ga_n}\E\left(\left[(1+\e)(\hb_\la-\be_\la)^2-\eta_{\la,\gamma}^2\right]\indic_{|\hb_\la|\geq \eta_{\la,\gamma}}\right).$$
Hence,
$$\mathcal{A}\leq
\sum_{\la\in\Ga_n}\E\left((1+\e)(\hb_\la-\be_\la)^2\indic_{(1+\e)(\hb_\la-\be_\la)^2\geq\eta_{\la,\gamma}^2}\indic_{
      |\hb_\la|\geq \eta_{\la,\gamma}}\right).$$
Then, remark that if $|\hb_\la|\geq \eta_{\la,\gamma}$ then $|\hb_\la| \geq \frac{\mu\ln n
}n \norm{\p_\la}_\infty$, where $\mu=[\sqrt{6}+1/3]\gamma$ but also that $|\hb_\la| \leq \frac{\norm{\p_\la}_\infty
N_\la}{n}$, hence $N_\la\geq \mu \ln n,$
where $$N_\la=\int_{supp(\p_\la)} dN.$$
So, one can split $\mathcal{A}$ and bound this term by $LDLM + LDSM$, where
$$LDLM= \sum_{\la\in\Ga_n}\E\left((1+\e)(\hb_\la-\be_\la)^2\indic_{(1+\e)(\hb_\la-\be_\la)^2\geq\eta_{\la,\gamma}^2}\indic_{
      |\hb_\la|\geq \eta_{\la,\gamma}} \indic_{N_\la\geq \mu \ln n}\indic_{nF_\la\geq \theta\mu \ln n}\right),$$
and 
$$LDSM= \sum_{\la\in\Ga_n}\E\left((1+\e)(\hb_\la-\be_\la)^2\indic_{(1+\e)(\hb_\la-\be_\la)^2\geq\eta_{\la,\gamma}^2}\indic_{
      |\hb_\la|\geq \eta_{\la,\gamma}} \indic_{N_\la\geq \mu \ln n}\indic_{nF_\la\leq \theta\mu \ln n}\right),$$
 where $\theta<1
$ is a parameter that is  chosen later on. Here, $LDLM$ stands for ``large
deviation large mass'' and $LDSM$ stands for ``large deviation small mass''.
Let us begin with $LDLM$. By the H\"older Inequality 
$$LDLM \leq  \sum_{\la\in\Ga_n}
(1+\e)[\E|\hb_\la-\be_\la|^{2p})]^{1/p}\P(|\hb_\la-\be_\la|\geq
\eta_{\la,\gamma}/{\sqrt{1+\e}})^{1/q} \indic_{nF_\la\geq \theta\mu \ln n}.$$
 Before  going
further, let us state the following useful lemma:
\begin{Lemma} 
\label{toutesdev}
For any $u>0$
\begin{equation}
\label{surhb1}
\P\left(|\hb_\la-\be_\la|\geq
\sqrt{2uV_{\la,n}}+\frac{\norm{\p_\la}_\infty u}{3n}\right)\leq 2 e^{-u}.
\end{equation}
Moreover, for any $u>0$
\begin{equation}
\label{surv1}
\P\left(V_{\la,n}\geq \tilde{V}_{\la,n}(u)\right)\leq e^{-u},
\end{equation}
where
$$\tilde{V}_{\la,n}(u) = \hat{V}_{\la,n} +\sqrt{2\hat{V}_{\la,n} \frac{\norm{\p_\la}_\infty^2}{n^2}u}+3\frac{\norm{\p_\la}_\infty^2}{n^2}u.$$
\end{Lemma}
\begin{proof}
Equation (\ref{surhb1}) easily comes from the classical inequalities (see
Kingman's book \cite{kin} or  Equation (5.2) of \cite{ptrfpois}).  The
same classical inequalities applied to $-\p_\la^2/n^2$ instead of $\p_\la/n$ give that
$$\P\left(V_{\la,n}\geq \hat{V}_{\la,n} +\sqrt{2u \int_\X \frac{\p_\la^4(x)}{n^4}
n f(x) dx}+\frac{\norm{\p_\la}_\infty^2}{3n^2}u \right)\leq e^{-u}.$$
But one can remark that $$\int_\X \frac{\p_\la^4(x)}{n^4}
n f(x) dx\leq \frac{\norm{\p_\la}_\infty^2}{n^2} V_{\la,n}.$$
Set $a=u\frac{\norm{\p_\la}_\infty^2}{n^2}$, then
$$\P(V_{\la,n}-\sqrt{2V_{\la,n} a}-a/3\geq \hat{V}_{\la,n})\leq e^{-u}.$$
Let $\mathcal{P}(x)=x^2-\sqrt{2a}x-a/3$. The discriminant of this polynomial is $10a/3$ which is strictly larger than $2a$. Since $V_{\la,n}$ and $\hat{V}_{\la,n}$ are positive, this means that one can inverse the equation $\mathcal{P}(\sqrt{V_{\la,n}})=\hat V_{\la,n}$ and we obtain
$$\P(\sqrt{V_{\la,n}}\geq \mathcal{P}^{-1}(\hat{V}_{\la,n}))\leq e^{-u}.$$
But $\mathcal{P}^{-1}(\hat{V}_{\la,n})$ is the positive solution of 
$$(\mathcal{P}^{-1}(\hat{V}_{\la,n}))^2-\sqrt{2a}\mathcal{P}^{-1}(\hat{V}_{\la,n})-(a/3+\hat{V}_{\la,n})=0.$$
So, finally, $\mathcal{P}^{-1}(\hat{V}_{\la,n}) = \sqrt{\hat{V}_{\la,n}+5a/6}+\sqrt{a/2}$. To conclude it remains to remark that $\tilde{V}_{\la,n} \geq(\mathcal{P}^{-1}(\hat{V}_{\la,n}))^2.$
\end{proof}
Using Equations (\ref{surhb1}) and (\ref{surv1}) of Lemma \ref{toutesdev}, we have
\begin{eqnarray*}
&&\hspace{-2cm}\P(|\hb_\la-\be_\la|\geq
\eta_{\la,\gamma}/{\sqrt{1+\e}})\\  & \leq& \P\left(|\hb_\la-\be_\la|\geq \sqrt{\frac{2\ga\ln n }{1+\e} \tilde{V}_{\la,n}(\ga \ln n )}+\frac{\gamma\ln n \norm{\p_\la}_\infty }{3(1+\e)n}\right)\\
&\leq& \P\left(|\hb_\la-\be_\la|\geq \sqrt{\frac{2\ga\ln n }{1+\e}\tilde{V}_{\la,n}(\ga \ln n )}+\frac{\gamma\ln n \norm{\p_\la}_\infty }{3(1+\e)n}, V_{\la,n} \geq \tilde{V}_{\la,n}(\ga \ln n )\right)\\&+&\P\left(|\hb_\la-\be_\la|\geq\sqrt{\frac{2\ga\ln n }{1+\e}\tilde{V}_{\la,n}(\ga \ln n )}+\frac{\gamma\ln n \norm{\p_\la}_\infty }{3(1+\e)n}, V_{\la,n} < \tilde{V}_{\la,n}(\ga \ln n )\right)\\
&\leq &\P(V_{\la,n} \geq \tilde{V}_{\la,n}(\ga \ln n ))+ \P\left(|\hb_\la-\be_\la|\geq \sqrt{\frac{2\ga}{1+\e}\ln n V_{\la,n}}+\frac{\gamma\ln n \norm{\p_\la}_\infty }{3(1+\e)n}\right)\\
&\leq& n^{-\ga}+2n^{-\ga/(1+\e)}\\&\leq &3 n^{-\ga/(1+\e)}.
\end{eqnarray*}
We need another lemma which looks
like the Rosenthal inequality. 
\begin{Lemma}
\label{moment2}
For all $p\geq 2$, there exists some absolute constant $C$ such that
$$\E(|\hb_\la-\be_\la|^{2p}) \leq C^p p^{2p} \left(V_{\la,n}^p+\left[\frac{\norm{\p_\la}_\infty}{n}\right]^{2p-2}V_{\la,n}\right).$$
\end{Lemma}
\begin{proof}
We know  that a Poisson process  is infinitely divisible. This  means that for
all positive  integer $k$ one can  see $N$ as  the reunion of $k$  iid Poisson
processes,  $N^i$ with  intensity  (here) $nk^{-1}\times  f$  with respect  to
the Lebesgue measure. Hence, one can apply Rosenthal inequalities for all $k$, saying that
$$\hb_\la-\be_\la=\sum_{i=1}^k \int \frac{\p_\la(x)}{n} \left(dN_x^i-nk^{-1}f(x)dx\right)=\sum_{i=1}^k Y_i$$
where for any $i$, $$Y_i= \int \frac{\p_\la(x)}{n} \left(dN_x^i-nk^{-1}f(x)dx\right).$$ So
the $Y_i$'s are iid centered variables, all having a moment of order $2p$. 
We apply Rosenthal's inequality (see Theorem 2.5 of \cite{johnson}) on
the positive and negative parts of $Y_i$. This easily implies that
$$\E\left(\left|\sum_{i=1}^k   Y_i\right|^{2p}\right)    \leq   K(p)   \max\left(\left(\E\sum_{i=1}^kY_i^2\right)^p,
\left(\E\sum_{i=1}^k |Y_i|^{2p}\right)\right),$$ where $$K(p)\leq
\left(8\times\frac{2p}{\ln(2p)}\right)^{2p}.$$
It remains to bound the upper limit of $\E(\sum_{i=1}^k |Y_i|^{q})$ for all
$q\in\{2p,2\}\geq     2$    when     $k\to\infty$.     Let    us     introduce
$$\Omega_k=\{\forall\; i\in\{1,\dots,k\},  N^i_\X\leq 1\}.$$ Then,  it is easy
to see that $\P(\Omega_k^c)\leq k^{-1}(n\norm{f}_1)^2 $ (see e.g., (\ref{n2}) below). 
\\\\
On $\Omega_k$, $|Y_i|^q= O_k(k^{-q})$ if $\int \frac{\p_\la(x)}{n} dN_x^i =0$ and $|Y_i|^q= \left[\frac{|\p_\la(T)|}{n}\right]^{q} + O_k\left(k^{-1}\left[\frac{|\p_\la(T)|}{n}\right]^{q-1}\right)$ if $\int \frac{\p_\la(x)}{n} dN_x^i =\frac{\p_\la(T)}{n} $ where $T$ is the point of the process $N^i$.
Consequently,
\begin{multline}
\label{split}
\E\sum_{i=1}^k |Y_i|^{q} \leq \E\left(\indic_{\Omega_k} \left(\sum_{T\in N}\left[\left[\frac{|\p_\la(T)|}{n}\right]^{q} + O_k\left(k^{-1}\left[\frac{|\p_\la(T)|}{n}\right]^{q-1}\right)\right] + k O_k(k^{-q})\right)\right)\\ + \sqrt{\P(\Omega_k^c)} \sqrt{\E\left[\left(\sum_{i=1}^k |Y_i|^q\right)^2\right]}.
\end{multline}
But, \begin{multline*}\sum_{i=1}^k |Y_i|^q \leq 2^{q-1} \left(\sum_{i=1}^k \left[\left[\frac{\norm{\p_\la}_\infty}{n}\right]^q (N^i_\X)^q + \left(k^{-1}\int |\p_\la(x)| f(x) dx \right)^q\right]\right)\\
\leq 2^{q-1} \left(\left[\frac{\norm{\p_\la}_\infty}{n}\right]^q N_\X^q + k\left(k^{-1}\int |\p_\la(x)| f(x) dx \right)^q\right).\end{multline*}
So, when $k\to +\infty$, the last term in (\ref{split}) converges to 0 since a Poisson variable has moments of every order and
$$
\lim\sup_{k\to\infty} \E\sum_{i=1}^k |Y_i|^{q} \leq  \E\left( \int \left[\frac{|\p_\la(x)|}{n}\right]^{q} dN_x\right)\leq \left[\frac{\norm{\p_\la}_\infty}{n}\right]^{q-2} V_{\la,n},$$ 
which concludes the proof.
\end{proof}
Since
$$\left[\frac{\norm{\p_\la}_\infty}{n}\right]^{2p-2}V_{\la,n}\leq
\max\left(V_{\la,n}^p, \left[\frac{\norm{\p_\la}_\infty}{n}\right]^{2p}\right),$$
there exists some constant $\tilde{C}$  such that
$$\E(|\hb_\la-\be_\la|^{2p}) \leq{\tilde{C}}^p p^{2p}\left(V_{\la,n}^p+\left[\frac{\norm{\p_\la}_\infty}{n}\right]^{2p}\right).$$
Finally, 
\begin{eqnarray*}
LDLM &\leq& \square(1+\e) p^2 n^{-\gamma/(q(1+\e))} \sum_{\la \in \Ga_n}
\left(V_{\la,n}+\left(\frac{\norm{\p_\la}_\infty}{n}\right)^{2}\right)
\indic_{nF_\la\geq \theta \mu \ln n }.
\end{eqnarray*}
Since $\norm{\p_\la}_\infty \leq c_{\p,n}\sqrt{n}$ for all $\la \in \Ga_n$, one has
\begin{eqnarray*}
 LDLM &\leq& \square(1+\e) p^2 c_{\p,n}^2 n^{-\gamma/(q(1+\e))} \sum_{\la \in \Ga_n} \left(F_\la + \frac{1}{n}\right)
  \indic_{nF_\la\geq \theta \mu \ln n }\\ &\leq& \square(1+\e) p^2 c_{\p,n}^2 n^{-\gamma/(q(1+\e))} \left(\sum_{\la \in
  \Ga_n} F_\la + \frac1n \sum_{\la \in
  \Ga_n} \frac{nF_\la}{\theta \mu\ln n}\right).
\end{eqnarray*}
But,
\begin{equation}\label{majoF}
\sum_{\la\in    \Ga_n}F_{\la}=\sum_{\la\in    \Ga_n}\int    f(x)\indic_{x\in
supp(\p_\la)}dx=\int f(x)dx\sum_{\la\in \Ga_n}\indic_{x\in supp(\p_\la)}.
\end{equation}
Using (\ref{mphi}), we then have
\begin{eqnarray*}
\sum_{\la\in    \Ga_n}F_{\la}\leq\norm{f}_1m_{\p,n} \ln n.
\end{eqnarray*}
This is exactly what we need for the first part provided that $\theta$ is an absolute constant and $\mu>1$. 
Now we go back to $LDSM$.
Applying the H\"older inequality again one obtains,
$$LDSM \leq (1+\e)\sum_{\la \in \Ga_n} \E(|\hb_\la-\be_\la|^{2p})^{1/p}
\P(N_\la-nF_\la \geq (1-\theta)\mu \ln n )^{1/q}.$$ 
To deal with  this term, we state
the following result.
\begin{Lemma}
\label{nombredepoints} There exists an absolute constant $0<\theta<1$ such that if $nF_\la \leq \theta \mu \ln n $, then,  for all $n$ such that $(1-\theta)\mu \ln n \geq 2$,
$$\P(N_\la-nF_\la \geq (1-\theta)\mu \ln n) \leq F_\la n^{-\gamma}.$$
\end{Lemma}
\begin{proof}
We use the same classical inequalities (see
Kingman's book \cite{kin} or  equation (5.2) of \cite{ptrfpois}). 
$$\P(N_\la-nF_\la \geq (1-\theta)\mu \ln n)
\leq\exp\left(-\frac{((1-\theta)\mu\ln n)^2}{2(nF_\la+(1-\theta)\mu\ln n/3)}\right)\leq
n^{-\frac{3(1-\theta)^2}{2(2\theta+1)} \mu}.$$
If $nF_\la \geq n^{-\gamma-1}$, then provided that  $\frac{3(1-\theta)^2}{2(2\theta+1)} \mu \geq
2 \gamma+2,$ one has the result. This imposes the value of $\theta$. Indeed since $$\frac{3(1-\theta)^2}{2(2\theta+1)} \mu = \frac{3(1-\theta)^2}{2(2\theta+1)}
(\sqrt{6}+1/3)\ga$$ one takes $\theta$ such that  $$\frac{3(1-\theta)^2}{2(2\theta+1)}(\sqrt{6}+1/3)=4.$$
If  $nF_\la \leq n^{-\gamma-1}$, 
\begin{multline}\label{n2}
\P(N_\la-nF_\la \geq (1-\theta)\mu \ln n)\leq \P(N_\la >(1-\theta) \mu
\ln n)\leq \P(N_\la \geq 2)\\ 
\leq \sum_{k\geq 2} \frac{(nF_\la)^k}{k!} e^{-n F_\la}\leq (nF_\la)^2
\leq F_\la n^{-\gamma}.
\end{multline}
\end{proof} 
We  apply  Lemma  \ref{nombredepoints}   to  bound  the  deviation  and  Lemma
\ref{moment2} to bound $\E(|\hb_\la-\be_\la|^{2p})$. Hence,
$$LDSM \leq \square (1+\e) p^2  n^{-\gamma/q}\sum_{\la \in \Ga_n}
\left(V_{\la,n}+\left[\frac{\norm{\p_\la}_\infty}{n}\right]^{2-2/p}V_{\la,n}^{1/p}\right)
F_\la^{1/q}.$$
Since $\norm{\p_\la}_\infty\leq c_{\p,n}\sqrt{n}$,
$$LDSM \leq \square (1+\e) p^2 c_{\p,n}^2 n^{-\gamma/q}\sum_{\la \in \Ga_n} (F_\la^{1+1/q}+F_\la).$$
Finally, as previously, by using (\ref{majoF})
$$LDSM
\leq
 \square (1+\e) p^2 c_{\p,n}^{2} m_{\p,n} n^{-\gamma/q} \ln (n)(\norm{f}_1)\max(\norm{f}_1, 1)^{1/q}.$$
\subsection{Proof of Theorem \ref{inegoraclelavraie}}
At first, we apply Theorem \ref{inegoracle} with $c_{\p,n}=\|\p\|_{\infty}2^{j_0/2}n^{-1/2}$.
For the last term, we want to prove that one can always find $q$ and $\e$ such
that $2^{j_0}n^{-\ga/(q(1+\e))-1}\ln(n)=o(n^{-1})$.  But if $\ga  >c$ then one
can always  find $q>1$  and $\e>0$ such  that $\ga>cq(1+\e)$ and  this implies
also that  $\ga>1+\e$. So,  by exchanging the  infimum and the  expectation we
obtain that
\begin{multline*}
\E(\normp{\tilde{f}_{n,\gamma}-f}^2)\leq (1+2\e^{-1})\inf_{m\subset \Ga_n} \left\{(1+2\e^{-1})\sum_{\la\not\in m} \be_\la^2+\sum_{\la\in m} [\e V_{\la,n} + \E(\eta_{\la,\gamma}^2)]\right\}\\ 
+ \frac{C_2(\gamma,\|f\|_1, c,c', \p)}{n}.
\end{multline*}
But for all $\delta>0$, $$\E(\eta_{\la,\gamma}^2)\leq (1+\delta) 2 \ga \ln n \E(\tilde{V}_{\la,n})+ (1+\delta^{-1})\left(\frac{\ga \ln n}{3n}\right)^2\norm{\p_\la}_\infty^2.$$
Moreover $$\E(\tilde{V}_{\la,n})\leq (1+\delta) V_{\la,n} +(1+\delta^{-1}) 3 \gamma \ln n \frac{\norm{\p_\la}_\infty^2}{n^2}.$$
So, finally for all $\delta>0$,
\begin{multline}\label{init}
\E(\normp{\tilde{f}_{n,\gamma}-f}^2)\leq (1+2\e^{-1})\\\inf_{m\subset \Ga_n} \left\{(1+2\e^{-1})\sum_{\la\not\in m} \be_\la^2+\sum_{\la\in m} [\e + (1+\delta)^2 2 \ga \ln n] V_{\la,n} + c(\delta,\ga)\sum_{\la\in m}\left(\frac{\ln n \norm{\p_\la}_\infty}{n}\right)^2\right\}\\ 
+ \frac{C_2(\gamma,\|f\|_1, c,c', \p)}{n},
\end{multline}
where $c(\delta,\ga)$ is a positive constant. One needs the following lemma. 
\begin{Lemma}\label{ecra}
We set $$S_\p=\max\{\sup_{x\in supp(\phi)}|\phi(x)|,\sup_{x\in
 supp(\psi)}|\psi(x)|\}$$ and $$I_\p=\min\{\inf_{x\in
 supp(\phi)}|\phi(x)|,\inf_{x\in
 supp(\psi)}|\psi(x)|\}.$$
Using (\ref{minophi}), we define $\Theta_\p=\frac{S_\p^2}{I_\p^2}.$
We have, for all $\la \in \La$, 
\begin{itemize}
\item[-] if 
$F_\la\leq \Theta_\p\frac{\ln(n)}{n},$
then
$\be_\la^2\leq\Theta_\p^2\si_\la^2\frac{\ln(n)}{n},$\\
\item[-] if $F_\la>\Theta_\p\frac{\ln(n)}{n},$
then                                $\norm{\p_\la}_{\infty}\frac{\ln(n)}{n}\leq
\si_\la\sqrt{\frac{\ln(n)}{n}}.$
\end{itemize}
\end{Lemma}
\begin{proof}
We note $\la=(j,k)$ and assume that $j\geq 0$ (arguments are similar for
$j=-1$).\\ If 
$F_\la\leq\Theta_\p\frac{\ln(n)}{n}$, we have
\begin{eqnarray*}
|\be_\la|&\leq &S_\psi 2^{j/2}F_\la\\
&\leq&S_\p2^{j/2}\sqrt{F_\la}\sqrt{\Theta_\p}\sqrt{\frac{\ln(n)}{n}}\\
&\leq&S_\p I^{-1}_\p\sqrt{\Theta_\p}\si_\la\sqrt{\frac{\ln(n)}{n}}\\
&\leq&\Theta_\p\si_\la\sqrt{\frac{\ln(n)}{n}},
\end{eqnarray*}
since
$$\si_\la^2\geq I^2_\p 2^jF_\la.$$
For the second point, observe that
\begin{eqnarray*}
\si_\la\sqrt{\frac{\ln(n)}{n}}&\geq&2^{j/2}I_\p\sqrt{\Theta_\p} \frac{\ln(n)}{n}
\end{eqnarray*}
and 
\begin{eqnarray*}
\norm{\psi_\la}_{\infty}\frac{\ln(n)}{n}&\leq&2^{j/2}S_\p\frac{\ln(n)}{n}.
\end{eqnarray*}
\end{proof}
Now let us apply (\ref{init}) for some fixed $\delta,\e$ to 
$$m=\left\{\la\in\Ga_n:\quad\be_\la^2>\Theta_\p^2\frac{\si_\la^2}{n}\ln n\ \right\}.$$
This implies that for all $\la \in m$, $
F_\la>\Theta_\p\frac{\ln(n)}{n}.$
So, since $\Theta_\p\geq 1$,
\begin{multline*}
\E(\normp{\tilde{f}_{n,\gamma}-f}^2)\leq C(\ga) \times\\
\left[\sum_{\la\in\Ga_n}\be_\la^2\indic_{\be_\la^2\leq\Theta_\p^2\frac{\si_\la^2}{n}\ln
n}+\sum_{\la\notin\Gamma_n}\be_\la^2+\sum_{\la\in\Ga_n}\left[
\frac{\ln n}{n} \si_\la^2 +\left(\frac{\ln n
}{n}\right)^2\norm{\p_\la}_\infty^2\right]\indic_{\be_\la^2>\Theta_\p^2\frac{\si_\la^2}{n}\ln
n, \
F_\la>\Theta_\p\frac{\ln(n)}{n}}\right]\\
+\frac{C_2(\gamma,\|f\|_1, c,c', \p)}{n}\\
\leq
C(\gamma)\left[\sum_{\la\in\Ga_n}\left(\be_\la^2\indic_{\be_\la^2\leq
\Theta_\p^2 V_{\la,n}\ln n }+2\ln n
V_{\la,n}\indic_{\be_\la^2>\Theta_\p^2 V_{\la,n}\ln
n}\right)+\sum_{\la\notin\Gamma_n}\be_\la^2\right]+\\
+\frac{C_2(\gamma,\|f\|_1, c,c', \p)}{n}\\
\leq C_1(\gamma)\left[\sum_{\la\in\Ga_n}\min(\be_\la^2,\Theta_\p^2
V_{\la,n}\ln n)+\sum_{\la\notin\Gamma_n}\be_\la^2\right]+\frac{C_2(\gamma,\|f\|_1, c,c', \p)}{n},
\end{multline*}
where $C(\ga)$ and $C_1(\ga)$ are  positive quantities depending only
on $\gamma$.
\subsection{Proof of Theorem \ref{maxisets}}
Let us assume that $f$ belongs to
$B_{2,\Gamma}^{\frac{\al}{1+2\al}}(R^{\frac{1}{1+2\al}})\cap       W_\al(R)\cap
\L_1(R)\cap \L_2(R)$. Inequality (\ref{inegoraclelavraie1}) of Theorem~\ref{inegoraclelavraie} implies for all $n$,
\begin{multline*}
\E(\normp{\tilde{f}_{n,\gamma}-f}^2)\leq C_1(\gamma,\p)\left[\sum_{\la\in\Gamma_n}\left(\be_\la^21_{|\be_\la|\leq\si_\la\sqrt{\frac{\ln n}{n}}}+
V_{\la,n}\ln n1_{|\be_\la|>\si_\la\sqrt{\frac{\ln
n}{n}}}\right)+\sum_{\la\not\in\Gamma_n}\be_\la^2\right]+\\+\frac{C_2(\gamma,R,
c,c', \p)}{n}.
\end{multline*}
But
\begin{eqnarray*}
\sum_{\la\in\Gamma_n}V_{\la,n}\ln n1_{|\be_\la|>\si_\la\sqrt{\frac{\ln n}{n}}}&=&\sum_{\la\in\Gamma_n}\si_\la^2\frac{\ln n}{n}\sum_{k=0}^{+\infty}1_{2^{-k-1}\be_\la^2\leq\si_\la^2\frac{\ln n}{n}<2^{-k}\be_\la^2}\\
&\leq&\sum_{k=0}^{+\infty}2^{-k}\sum_{\la\in\Lambda}\be_\la^21_{|\be_\la|\leq 2^{(k+1)/2}\si_\la\sqrt{\frac{\ln n}{n}}}\\
&\leq&\sum_{k=0}^{+\infty}2^{-k}R^{\frac{2}{1+2\al}}\left(2^{(k+1)/2}\sqrt{\frac{\ln n}{n}}\right)^{\frac{4\al}{1+2\al}}\\
&\leq&R^{\frac{2}{1+2\al}}\rho_{n,\al}^2\sum_{k=0}^{+\infty}2^{-k+\frac{2\al(k+1)}{1+2\al}}
\end{eqnarray*}
and
$$\sum_{\la\not\in\Gamma_n}\be_\la^2\leq R^{\frac{2}{1+2\al}}\rho_{n,\al}^2.$$
So,
$$
\E(\normp{\tilde{f}_{n,\gamma}-f}^2)\leq C(\gamma,\p,\al)R^{\frac{2}{1+2\al}}\rho_{n,\al}^2+\frac{C_2(\gamma,R, c,c', \p)}{n},$$
where
$C(\gamma,\p,\al)$ depends on $\gamma$, the basis and $\al$. Hence,
$$\E(\normp{\tilde{f}_{n,\gamma}-f}^2)\leq C(\gamma,\p,\al)R^{\frac{2}{1+2\al}}\rho_{n,\al}^2 (1+o_n(1))$$
and $f$ belongs to $MS(\tilde
f_\gamma,\rho_\al)(R')$ for $R'$ large enough.
\\\\
Conversely, let us suppose that $f$ belongs to $MS(\tilde
f_\gamma,\rho_\al)(R')\cap \L_1(R')\cap \L_2(R').$
Then, for
any $n$,
$$\E(\normp{\tilde{f}_{n,\gamma}-f}^2)\leq {R'}^2\left(\frac{\ln
n}{n}\right)^{\frac{2\al}{1+2\al}}.$$
 Consequently, for
any $n$, $$\sum_{\la\not\in\Gamma_n}\be_{\la}^2 \leq {R'}^2\left(\frac{\ln
n}{n}\right)^{\frac{2\al}{1+2\al}}.$$
This implies that $f$ belongs to  
$B_{2,\Gamma}^{\frac{\al}{1+2\al}}(R')$.\\
Now, we want to prove
that $f\in W_\al(R)$ for $R>0$. We have
$$
\sum_{\la\in \La} \be_\la^2 \indic_{|\be_\la|\leq\si_\la\sqrt{\frac{\ga\ln n}{2n}}} \leq 
\sum_{\la\not \in \Ga_n} \be_\la^2 + \sum_{\la\in \Ga_n} \be_\la^2
\indic_{|\be_\la|\leq\si_\la\sqrt{\frac{\ga\ln n}{2n}}}.$$
But $\tb_\la=\hb_\la \indic_{|\hb_\la|\geq\eta_{\la,\gamma}}$, so, 
$$|\be_\la|\indic_{|\be_\la|\leq \frac{\eta_{\la,\gamma}}{2}}\leq |\be_\la-\tb_\la|.$$
So, for any $n$,
\begin{eqnarray*}
\sum_{\la\in \La} \be_\la^2 \indic_{|\be_\la|\leq\si_\la\sqrt{\frac{\ga\ln n}{2n}}} &\leq &\sum_{\la\not \in \Ga_n} \be_\la^2+ \E\left\{\sum_{\la\in \Ga_n} \be_\la^2
\indic_{|\be_\la|\leq\si_\la\sqrt{\frac{\ga\ln n}{2n}}}[\indic_{|\be_\la|\leq
 \frac{\eta_{\la,\gamma}}{2}}+\indic_{|\be_\la|> \frac{\eta_{\la,\gamma}}{2}}]\right\}\\
&\leq &\sum_{\la\not \in \Ga_n} \be_\la^2+  \sum_{\la\in
\Ga_n}\E[(\tb_\la-\be_\la)^2] + \sum_{\la\in \Ga_n} \be_\la^2
\indic_{|\be_\la|\leq\si_\la\sqrt{\frac{\ga\ln n}{2n}}} \E(\indic_{|\be_\la|> \frac{\eta_{\la,\gamma}}{2}})\\
&\leq &\sum_{\la\not \in \Ga_n} \be_\la^2+  \sum_{\la\in
\Ga_n}\E[(\tb_\la-\be_\la)^2] + \sum_{\la\in \Ga_n} \be_\la^2
\P\left(\si_\la\sqrt{\frac{\ga\ln n}{2n}}> \frac{\eta_{\la,\gamma}}{2}\right)\\
&\leq & \E(\normp{\tilde{f}_{n,\gamma}-f}^2) + \sum_{\la\in \Ga_n} \be_\la^2
\P\left(\si_\la\sqrt{\frac{\ga\ln n}{2n}}> \frac{\eta_{\la,\gamma}}{2}\right).
\end{eqnarray*}
Using Lemma \ref{toutesdev}, 
$$
\P\left(\si_\la\sqrt{\frac{2\ga\ln n}{n}}>\eta_{\la,\gamma}\right)\leq \P(\tilde{V}_{\la,n} \leq V_{\la,n})\leq n^{-\gamma}
$$ 
and 
$$\sum_{\la\in \La} \be_\la^2 \indic_{|\be_\la|\leq\si_\la\sqrt{\frac{\ga\ln n}{2n}}}\leq (R')^2\left(\sqrt{\frac{\ln n}{n}}\right)^{\frac{4\al}{1+2\al}}+\|f\|_{\tilde \p}^2n^{-\ga}.$$
Since this is true for every $n$, we have for any $t\leq 1$,
\begin{equation}\label{gammamin}
\sum_{\la\in \La} \be_\la^2 \indic_{|\be_\la|\leq \si_\la t}\leq R^{\frac{2}{1+2\al}}\left(\sqrt{\frac{2}{\ga}}\ t\right)^{\frac{4\al}{1+2\al}},
\end{equation}
where $R$ is a constant large enough depending on $R'$. Note that
$$\sup_{t\geq 1}t^{\frac{-4\al}{1+2\al}}\sum_{\la\in\Lambda}\be_\la^2
\indic_{|\be_\la|\leq \si_\la t}\leq \|f\|_{\tilde \p}^2.$$
We conclude that $$f\in B^{\frac{\al}{1+2\al}}_{2,\Gamma}(R)\cap W_\al(R)$$
for $R$ large enough.
\subsection{Proof of Proposition \ref{contreexlinfini}}
Since $\beta<1/2$, $f_\beta\in \L_1\cap \L_2$. If the Haar basis is considered, the wavelet coefficients $\beta_{j,k}$ of $f_\beta$ can be calculated and we obtain for any $j\geq 0$, for any $k\not\in\left\{0,\dots,2^j-1\right\}$, $\beta_{j,k}=0$
and for any $j\geq 0$, for any $k\in\left\{0,\dots,2^j-1\right\}$,
$$\beta_{j,k}=(1-\beta)^{-1}2^{-j\left(\frac{1}{2}-\beta\right)}\left(2\left(k+\frac{1}{2}\right)^{1-\beta}-k^{1-\beta}-\left(k+1\right)^{1-\beta}\right)$$
and there exists a constant $0<c_{1,\beta}<\infty$ only depending on $\beta$ such that 
$$\lim_{k\to\infty}2^{j\left(\frac{1}{2}-\beta\right)}k^{1+\beta}
\beta_{j,k}=c_{1,\beta}.$$
Moreover the $\beta_{j,k}$'s are strictly positive. Consequently they
can         be         bounded         up         and         below,  up to a constant,        by
$2^{-j\left(\frac{1}{2}-\beta\right)}k^{-(1+\beta)}$. 
Similarly, for any $j\geq 0$, for any $k\in\left\{0,\dots,2^j-1\right\}$,
$$\si^2_{j,k}=(1-\beta)^{-1}2^{j\beta}\left((k+1)^{1-\beta}-k^{1-\beta}\right).$$
and there exists a constant $0<c_{2,\beta}<\infty$ only depending on $\beta$ such that 
$$\lim_{k\to\infty}2^{-j\beta}k^{\beta}\si^2_{j,k}=c_{2,\beta}.$$
There exist two constants $\kappa(\beta)$ and $\kappa'(\beta)$ only depending on $\beta$ such that for any $0<t<1$,
$$|\beta_{j,k}|\leq t\si_{j,k} \Rightarrow k\geq \kappa(\beta)t^{-\frac{2}{\beta+2}}2^{j\left(\frac{\beta-1}{\beta+2}\right)}$$
and
$$
\kappa(\beta)t^{-\frac{2}{\beta+2}}2^{j\left(\frac{\beta-1}{\beta+2}\right)}\geq2^j\iff
2^j\leq\kappa'(\beta)t^{-\frac{2}{3}}.$$
So,  if  $2^j\leq  \kappa'(\beta)t^{-\frac{2}{3}},$ since  $\beta_{jk}=0$  for
$k\geq 2^j$, $$\sum_{k\in\Z}\beta_{j,k}^21_{\beta_{j,k}\leq t\si_{j,k}}=0.$$
We obtain
$$\sum_{\la\in\Lambda}\beta_\la^21_{|\be_\la|\leq
t\si_\la}
\leq C(\beta)\sum_{j=-1}^{+\infty}2^{-j(1-2\beta)}1_{2^j> \kappa'(\beta)t^{-\frac{2}{3}}}\sum_{k=1}^{2^j-1}k^{-2-2\beta}\leq C'(\beta) t^{\frac{2-4\beta}{3}},$$
where  $C(\beta)$  and $C'(\beta)$  denote  two  constants  only depending  on
$\beta$.   So,   for   any   $0<\al<\frac{1}{4}$,  if   we   take   $\beta\leq
\frac{1-4\al}{2+4\al}$, then, for any $0<t<1$, $t^{\frac{2-4\beta}{3}}\leq t^{\frac{4\al}{1+2\al}}$. Finally, there exists $c\geq 1$, such that for any $n$,
$$\sum_{\la\not\in\Gamma_n}\be_\la^2\leq R^2\rho_{n,\al}^2$$ where $R>0$. And in this case,
$$f_\beta\not\in \L_{\infty},\quad f_\beta\in {\cal B}^{\frac{\al}{c(1+2\al)}}_{2,\infty}\cap W_\al:=MS(\tilde f_\gamma,\rho_\al).$$
\subsection{Proof of Theorem \ref{minimaxnoncompact}}
Since 
\begin{equation}\label{majosi}
\forall\;\la=(j,k),\quad                                          \si_\la^2\leq
 \min\left[\max(2^j;1)\norm{\p}_{\infty}^2F_{j,k} \ ; \ \norm{f}_{\infty}\norm{\p}_{2}^2\right],
\end{equation}
where $\p\in\{\phi,\psi\}$ according to the value of $j$, we have for any $t>0$
and any $\tilde J\geq 0$
\begin{eqnarray*}
\sum_{\la} \be_\la^2 \indic_{|\be_\la|\leq \si_\la t }
&\leq&\sum_{j<       \tilde       J}\sum_k\si_{j,k}^2t^2+\sum_{j\geq\tilde
J}\sum_k\be_{j,k}^2\left(\frac{\si_{j,k} t}{|\be_{j,k}|}\right)^{2-p}
\\
&\leq&\max(\norm{\phi}_{\infty}^2,\norm{\psi}_{\infty}^2)t^2\sum_{j<       \tilde       J}2^j\sum_kF_{j,k}+\sum_{j\geq\tilde
J}\sum_k\be_{j,k}^2\left(\frac{t\sqrt{\norm{f}_{\infty}\norm{\psi}_{2}^2}}{|\be_{j,k}|}\right)^{2-p}\\
&\leq& c(\p,R')\left(2^{\tilde J}t^2+t^{2-p}\sum_{j\geq\tilde
J}\sum_k|\be_{j,k}|^p\right),
\end{eqnarray*}
where $c(\p,R')$ is  a constant only depending on the basis  and on $R'$. Now,
let us assume that $f$  belongs to ${\cal B}^\al_{p,\infty}(R)$ (that contains
${\cal B}^\al_{p,q}(R)$, see Section \ref{biorthogonal}), with
$\al+\frac{1}{2}-\frac{1}{p}>0$. Then, 
$$\sum_{\la}          \be_\la^2          \indic_{|\be_\la|\leq
\si_\la t}\leq c_1(\p,\al,p,R')\left(2^{\tilde
J}t^2+t^{2-p}R^p2^{-\tilde Jp(\al+\frac{1}{2}-\frac{1}{p})}\right).$$
where $c_1(\p,\al,p,R')$ depends on the basis, $\al$, $p$ and $R'$. With $\tilde
J$ such that
$$2^{\tilde J}\leq R^{\frac{2}{1+2\al}} t^{\frac{-2}{1+2\al}}<2^{\tilde J+1},$$
$$\sum_{\la}          \be_\la^2          \indic_{|\be_\la|\leq
\si_\la t}\leq c_2(\p,\al,p,R')
 R^{\frac{2}{1+2\al}} t^{\frac{4\al}{1+2\al}}$$
where $c_2(\p,\al,p,R')$  depends on the basis,  $\al$, $p$ and  $R'$. So, $f$
belongs to $W_{\al}(R'')$ for $R''$ large enough.\\
Furthermore, using (\ref{inclusion}), if $p\leq 2$ and $$\al\left(1-\frac1{c(1+2\alpha)}\right)\geq \frac1p -\frac12$$
$${\cal B}^\al_{p,\infty}(R)\subset{\cal
B}^{\frac{\al}{c(1+2\al)}}_{2,\infty}(R).$$
Finally, for $R''$ large enough,
$${\cal B}^\al_{p,q}(R)\subset{\cal B}^\al_{p,\infty}(R)\subset{\cal
B}^{\frac{\al}{c(1+2\al)}}_{2,\infty}(R'') \cap W_{\al}(R'').$$
We recall
$$MS(\tilde f_\gamma,\rho_\al):={\cal
B}^{\frac{\al}{c(1+2\al)}}_{2,\infty} \cap W_{\al},$$
which  proves (\ref{risqueBesov}).\\  
Moreover
$$\inf_{\hat{f}} \sup_{f\in \mathcal{B}_{p.q}^\al(R)\cap \mathcal{L}_{1,2,\infty}(R')}
\E(\norm{\hat{f}-f}^2) \geq C(\al,R,R') n^{-\frac{2\al}{2\al+1}},$$
where $C(\al,R,R')$ is a constant. Indeed, using computations similar to those of
Theorem 2 of \cite{djkp}, it is easy to prove that if $K$ is a compact interval and
${\cal B}^\al_{p,q,K}(R)$ is the set of functions supported by $K$ and belonging to ${\cal
B}^\al_{p,q}(R)$ the minimax risk associated
with ${\cal  B}^\al_{p,q,K}(R)$ is larger than $n^{-2\al/(1+2\al)}$  up to a
constant. \\
But (\ref{adapt*}) implies that $\al>\al^*$ and $p>p^*$ satisfy (\ref{sural}).
This proves the adaptive minimax properties of $\tilde f_\gamma$ stated in
the theorem. 
\subsection{Proof of Theorem \ref{contrexem}}
The proof  is established for $p<\infty$.  Similar arguments lead  to the same
results for $p=\infty$. Let us fix real numbers $n_*>1$ and $f_*>1$ and let us define the following increasing sequence 
$$a_0=0,\quad a_1=4\quad \mbox{ and } \forall \ l \geq 1, \ a_{l+1}=2a_l+2^{\lceil n_* l\rceil+1}.$$
Let $b_l=\frac{a_{l+1}}{2}-1.$
Let $I^+_\jk = [k2^{-j}, (k+1/2)2^{-j}]$ and $I^-_\jk =[(k+1/2)2^{-j},(k+1)2^{-j}]$.
Set for all $x\in \R$,  $$f_l(x) = \sum_{m=a_l}^{b_l} 2^{(1-f_*)l+1} \indic_{I^+_{l,m}}$$ and $$f(x) =\sum_{l=0}^{+\infty} f_l(x).$$ 
The  $f_l$'s have  support  in $S_l=[a_l2^{-l},  a_{l+1}2^{-(l+1)}[$. All  the
$S_l$'s are disjoint and  we can prove by an easy induction  that all the $a_l
2^{-l}$'s are even positive integer numbers (indeed, $a_{l+1}2^{-(l+1)}= 2^{\lceil n_* l\rceil-l}+a_l2^{-l}$ and $\lceil n_* l\rceil-l>0$ if $l\not=0$). 
\\\\
Now, let us  compute the wavelet coefficients associated  with $f$ denoted
$\be_\jk$ for $j\geq 0$ and for any $k\in\Z$ and $\al_k=\be_{-1,k}$ for any $k\in\Z$. We are working with the Haar basis. Recall that the spaces considered are viewed as sequence spaces. 

For the  $\be_\jk$'s, let  us remark that  $supp(\p_\jk)$ is  always included
between two  successive integers, consequently  there exists a  unique $l_\jk$
such  that  $supp(\p_\jk)\subset  S_{l_\jk}$.  So,  $$\be_\jk=\int  f_{l_\jk}
\p_{\jk}.$$  Moreover,  if  $j\not=l_\jk$,  the coefficient  is  zero:  either
$j>l_\jk$  and $\p_\jk$ sees  only one  flat line,  or $j<l_\jk$  and $\p_\jk$
integrates the  same number of flat  pieces in $I^+_\jk$  and
$I^-_\jk$ ; since
the pieces have all the same level, this is also 0. 
Finally, for $j=l_\jk$, the computation is easy and we find
$$\be_\jk= 2^{(1-f_*)j+1}\int_{I^+_\jk} 2^{j/2} [\indic_{I^+_\jk}-\indic_{I^-_\jk}]  \times\indic_{a_j\leq k\leq b_j} = 2^{-j(f_*-1/2)} \indic_{a_j\leq k\leq b_j}.$$
For the coefficients $\al_k$'s, there exists also a unique $l_k$ such that $supp (\p_{-1,k})\subset S_{l_k}$ and
$$\al_k=2^{(1-f_*)l_k+1}\frac12=\sum_l 2^{(1-f_*)l}\indic_{a_l2^{-l}\leq k< a_{l+1} 2^{-(l+1)}}.$$
Now, we want to compute $\si_\jk$ when $\be_\jk\not = 0$. If $j\geq 0$
$$F_{j,k}=\int_{supp(\psi_{j,k})}f(x)dx=2^{j(1-f_*)}2^{-j}=2^{-jf_*},$$
$$\si_{j,k}^2=\int\psi_{j,k}^2(x)f(x)dx=2^j\int_{supp(\psi_{j,k})}f(x)dx=2^jF_{j,k}=2^{j(1-f_*)}.$$
If $j=-1$ $$\si_{\jk} =F_\jk=\al_k.$$
\\
Now, we fix the parameter $n_*$ and $f_*$ such that
\begin{enumerate}
\item $\norm{f}_1<\infty$, $\norm{f}_2<\infty$, $\norm{f}_{\infty}<\infty$,
\item $f\in {\cal B}^{\al}_{p,\infty}$,
\item $f\notin W_\al$.
\end{enumerate}
Since $f_*>1$, then $\norm{f}_{\infty}<\infty$.
We have
\begin{equation}\label{ce1}
\norm{f}_1=\sum_{l=0}^{+\infty}\sum_{m=a_l}^{b_l}2^{(1-f_*)l+1}2^{-l-1}=\sum_{l=0}^{+\infty}2^{\lceil
n_*l\rceil}2^{-f_*l}<\infty\iff f_*>n_*. 
\end{equation}
We have for all $j\geq 0 $
\begin{eqnarray*}
\sum_k|\be_{j,k}|^p&=&\sum_{k=a_j}^{b_j}|2^{-j(f_*-1/2)}|^p\\
&=&2^{\lceil n_*j\rceil}2^{jp/2}2^{-jf_*p}.
\end{eqnarray*}
Then, 
\begin{eqnarray}\label{ce2}
f\in {\cal B}^{\al}_{p,\infty}&\iff & \exists R>0, \forall j\geq 0 , 2^{j(n_*+p/2-f_*p)}\leq R^p 2^{-jp(\al+1/2-1/p)}\nonumber\\
&\iff&n_*+p/2-f_*p\leq -p\al-p/2+1\nonumber\\
&\iff&n_*\leq pf_*-p+1-p\al.
\end{eqnarray}
Indeed, note that we have
\begin{eqnarray*}
\sum_{k\in\Z}|\al_k|^p&=&\sum_{l\geq 0}2^{\lceil n_*l\rceil-l}\left(2^{(1-f_*)l}\right)^p<\infty
\end{eqnarray*}
if and only if $n_*-1+p-f_*p<0$, which is true as soon as $f_*>n_*$. Note also that
$$
\norm{f}_2<\infty\iff 2f_*>1+n_*,
$$
which is also true as soon as $f_*>n_*$.\\
Now, we would like to build $f$ such that $f$ does not belong to $W_\al$.
We have for any $t<1$,
\begin{eqnarray*}
\sum_{k=a_j}^{b_j}\be_{j,k}^2\indic_{|\be_{j,k}|\leq t\si_{j,k}}&=&\sum_{k=a_j}^{b_j}2^{-2j(f_*-1/2)}\indic_{2^{-j(f_*-1/2)}\leq t2^{j(1-f_*)/2}}\\
&=&2^{j(1-2f_*)}2^{\lceil n_*j\rceil}\indic_{2^{-jf_*}\leq t^2}.
\end{eqnarray*}
So, with $j=\lceil \log_2 (t^{-2/f_*})\rceil$,
\begin{eqnarray}\label{ce4}
\sup_{t<1}t^{-4\al/(1+2\al)}\sum_j\sum_{k=a_j}^{b_j}\be_{j,k}^2\indic_{|\be_{j,k}|\leq t\si_{j,k}}=+\infty&\Leftarrow&\sup_{t<1}t^{-4\al/(1+2\al)}t^{-2(1+n_*-2f_*)/f_*}=+\infty\nonumber\\
&\iff&-2(1+n_*-2f_*)/f_*<4\al/(1+2\al)\nonumber\\
&\iff&2f_*-n_*-1<\frac{2\al f_*}{1+2\al}\nonumber\\
&\iff&n_*>-1+\frac{2f_*(1+\al)}{1+2\al},
\end{eqnarray}
and in  this case, $f\notin W_\al.$  Now, we choose $n_*>1$  and $f_*>1$ such
that (\ref{ce1}), (\ref{ce2}) and (\ref{ce4}) are satisfied. 
For this purpose, we
take
$$f_*=1+2\al-\delta\in \left](1+2\al)\frac{(p\al+p-2)}{2p\al+p-2\al-2},1+2\al\right[$$
for $\delta\in ]0,\al[$ and $\delta$ small enough. Note that $p>2$ implies
$$(1+2\al)\frac{(p\al+p-2)}{2p\al+p-2\al-2}<1+2\al.$$
We also take
$$n_*=\min(f_*-\delta', pf_*-p+1-p\al)\in ]1,pf_*-p+1-p\al]$$
for $\delta'$ small enough. Note that $$pf_*-p+1-p\al=p(1+2\al-\delta)-p+1-p\al=p\al+1-p\delta>1.$$
With such a choice, we have 
$n_*<f_*$  and   $n_*\leq  pf_*-p+1-p\al$. So   (\ref{ce1}) and
(\ref{ce2})  are satisfied. It remains to check (\ref{ce4}). 
We have
\begin{eqnarray*}
pf_*-p+1-p\al>-1+\frac{2f_*(1+\al)}{1+2\al}&\iff&f_*\left[\frac{2(1+\al)}{1+2\al}-p\right]<2-p-p\al\\
&\iff& f_*(2+2\al-p-2p\al)<(1+2\al)(2-p-p\al)\\
&\iff&f_*>(1+2\al)\frac{(p\al+p-2)}{2p\al+p-2\al-2},
\end{eqnarray*}
and
\begin{eqnarray*}
f_*-\delta'>-1+\frac{2f_*(1+\al)}{1+2\al}&\iff&f_*\left[\frac{2(1+\al)}{1+2\al}-1\right]<1-\delta'\\
&\iff&2(1+\al)f_*-f_*(1+2\al)<(1+2\al)(1-\delta')\\
&\iff &f_*<(1+2\al)(1-\delta'),
\end{eqnarray*}
which is true for $\delta'$ small enough. So (\ref{ce4}) is satisfied, which concludes the proof of the theorem. 
\subsection{Proof of Theorem \ref{vitesse}}
The proof is established for $q=\infty$ and $p<\infty$. Similar arguments lead
to the  same results  for $p=\infty$. In  the sequel,  $C$  designates  a constant
depending on $R'$, $\ga$, $c$, $c'$, on the parameters of the Besov ball, on the
basis and
that may change at each line. We have for any $0<t<1$ and any $j\geq 0$,
\begin{eqnarray}\label{CS}
\sum_k\be_{j,k}^21_{|\be_{j,k}|\leq t\si_{j,k}}\leq \left(\sum_k|\be_{j,k}|^p\right)^{\frac{1}{p}}\left(\sum_k|\be_{j,k}|^r1_{|\be_{j,k}|\leq t\si_{j,k}}\right)^{\frac{1}{r}}
\end{eqnarray}
with $\frac{1}{p}+\frac{1}{r}=1.$ So, using (\ref{majosi}),
we have if $f\in\L_{\infty}(R')\cap\L_1(R')\cap{\cal B}^\al_{p,\infty}(R),$ 
\begin{eqnarray*}
\sum_k\be_{j,k}^21_{|\be_{j,k}|\leq t\si_{j,k}}&\leq &C 2^{-j\left(\al+\frac{1}{2}-\frac{1}{p}\right)}\left(\sum_k|\be_{j,k}|(t\si_{j,k})^{r-1}\right)^{\frac{1}{r}}\\&\leq &C 2^{-j\left(\al+\frac{1}{2}-\frac{1}{p}\right)}\left(\sum_k|\be_{j,k}|t^{r-1}\right)^{\frac{1}{r}}\\
&\leq &C 2^{-j\left(\al+\frac{1}{2}-\frac{1}{p}-\frac{1}{2r}\right)}t^{1-\frac{1}{r}}.\\
\end{eqnarray*}
Indeed,
\begin{equation}\label{L1}
f\in\L_{1}(R')\Rightarrow \sum_k|\be_{j,k}|\leq C 2^{\frac{j}{2}} 
\end{equation}
 (see \cite{jll}, p. 197). So, for $\al>1/(2p)$, we have for any $t>0$ and any $\tilde J\geq 0$
\begin{eqnarray*}
\sum_{\la}  \be_\la^2   \indic_{|\be_\la|\leq \si_\la t}&=&\sum_{j}\sum_{k}\be_{j,k}^2\indic_{|\be_{j,k}|\leq \si_{j,k}t}\\
&\leq&C\left[t^2\sum_{j<       \tilde       J}2^j\sum_kF_{j,k}+\sum_{j\geq\tilde
J}2^{-j\left(\al+\frac{1}{2}-\frac{1}{p}-\frac{1}{2r}\right)}t^{1-\frac{1}{r}}\right]\mbox{
using (\ref{majosi}) again}\\
&\leq&C\left[t^22^{\tilde J}+2^{-\tilde J\left(\al-\frac{1}{2p}\right)}t^{1-\frac{1}{r}}\right].
\end{eqnarray*}
With 
$$2^{\tilde J}\leq t^{-\frac{1+\frac{1}{r}}{\al+\frac{1}{2}+\frac{1}{2r}}}<2^{\tilde J+1}$$
we have
$$ \sum_{\la} \be_\la^2 \indic_{|\be_\la|\leq \si_\la t}\leq Ct^{\frac{2\al}{\al+\frac{1}{2}+\frac{1}{2r}}}.$$
We obtain 
$$ \sum_{\la}  \be_\la^2   \indic_{|\be_\la|\leq \si_\la t}\leq C t^{\frac{2\al}{\al+1-\frac{1}{2p}}}.$$
So, with $t=\sqrt{\frac{\ln n}{n}}$,
$$\sum_{\la\in\Gamma_n}      \be_\la^2      \indic_{|\be_\la|\leq      \si_\la
\sqrt{\frac{\ln                  n}{n}}}\leq                  C\left(\frac{\ln
n}{n}\right)^{\frac{\al}{\al+1-\frac{1}{2p}}}.$$
Furthermore,
\begin{eqnarray*}
\sum_{\la\in\Gamma_n}V_{\la,n}\ln n1_{|\be_\la|>\si_\la\sqrt{\frac{\ln n}{n}}}&=&\sum_{\la\in\Gamma_n}\si_\la^2\frac{\ln n}{n}\sum_{k=0}^{+\infty}1_{2^{-k-1}\be_\la^2\leq\si_\la^2\frac{\ln n}{n}<2^{-k}\be_\la^2}\\
&\leq&\sum_{k=0}^{+\infty}2^{-k}\sum_{\la\in\Lambda}\be_\la^21_{|\be_\la|\leq 2^{(k+1)/2}\si_\la\sqrt{\frac{\ln n}{n}}}\\
&\leq&C \sum_{k=0}^{+\infty}2^{-k}\left(2^{(k+1)/2}\sqrt{\frac{\ln n}{n}}\right)^{\frac{2\al}{\al+1-\frac{1}{2p}}}\\
&\leq&C \left(\frac{\ln n}{n}\right)^{\frac{\al}{\al+1-\frac{1}{2p}}}\sum_{k=0}^{+\infty}2^{-k+\frac{\al(k+1)}{1+\al-\frac{1}{2p}}}\\&\leq& C\left(\frac{\ln n}{n}\right)^{\frac{\al}{\al+1-\frac{1}{2p}}}.
\end{eqnarray*}
Now, using (\ref{CS}), (\ref{L1}) and (\ref{inclusion})
we have when $\la=(j,k)\not\in\Gamma_n$,
\begin{eqnarray*}
\sum_k\be_{j,k}^2&\leq &C 2^{-j\left(\al+\frac{1}{2}-\frac{1}{p}\right)}\left(\sum_k|\be_{j,k}|(\sup_k|\be_{j,k}|)^{r-1}\right)^{\frac{1}{r}}\\&\leq &C 2^{-j\left(\al+\frac{1}{2}-\frac{1}{p}\right)}\left(\sum_k|\be_{j,k}|2^{-\frac{j(r-1)}{2}}\right)^{\frac{1}{r}}\\
&\leq &C 2^{-j\al}.\\
\end{eqnarray*}
and applying Theorem~\ref{inegoraclelavraie}, we obtain for $c\geq 1$,
$$\sum_{\la\not\in\Gamma_n}  \be_\la^2\leq  C   \left(\frac{\ln
n}{n}\right)^{\frac{\al}{\al+1-\frac{1}{2p}}}$$
and
$$\E(\|\tilde    f_{n,\gamma}-f\|_{\tilde\p}^2)\leq  C   \left(\frac{\ln
n}{n}\right)^{\frac{\al}{\al+1-\frac{1}{2p}}}.$$
\subsection{Proof of Theorem \ref{minimaxsurmaxiset}}
Let us consider the Haar basis. For $j\geq 0$ and $D\in\{0,1,\dots,2^j\}$, we set
$$\mathcal{C}_{j,D}=\{f_m=\rho\indic_{[0,1]}+a_{j,D} \sum_{k\in
m}\tp_{\jk}:\quad |m|=D, m \subset \mathcal{N}_j\},$$
where $$\mathcal{N}_j=\{k:\quad \tp_{\jk} \mbox{ has support in } [0,1]\}.$$
The parameters $j, D, \rho ,a_{j,D}$ is chosen later to fulfill some requirements. Note that $$N_j=\mbox{card}(\mathcal{N}_j)=2^j.$$
We know that there exists a
subset of $\mathcal{C}_{j,D}$, denoted $\mathcal{M}_{j,D}$, and some universal
constants, denoted $\theta' $ and $\sigma$, such that for all
$m, m'\in \mathcal{M}_{j,D},$ $$\mbox{card}(m\Delta m')\geq \theta' D,\quad
\ln(\mbox{card}(\mathcal{M}_{j,D}))\geq \sigma D
\ln\left(\frac{2^j}{D}\right)$$ (see Lemma 8 of \cite{ptrfpois}).
 Now, let us describe all the requirements necessary to obtain the lower bound
of the risk. 
\begin{itemize}
\item To ensure $f_m\geq 0$ and the equivalence between the Kullback distance and the
$\L_{2}$-norm (see below), the $f_m$'s have to be larger than $\rho/2$. 
Since the $\tp_{\jk}$'s have disjoint support, this means that
\begin{equation}\label{cont1}
\rho \geq 2^{1+j/2} |a_{j,D}|.
\end{equation} 
\item  We need  the  $f_m$'s  to be  in  $\L_1(R'')\cap \L_\infty(R'')$.  Since
$\norm{f}_1=\rho$ and $\norm{f}_{\infty}=\rho+2^{j/2}|a_{j,D}|$, we need
\begin{equation}\label{cont0}
\rho+2^{j/2}|a_{j,D}|\leq R''.
\end{equation}
\item The $f_m$'s have to belong to $ {\cal
B}^{\frac{\al}{1+2\al}}_{2,\infty}(R')$ i.e. 
\begin{equation}\label{cont2}
\rho+2^{j\alpha/(1+2\alpha)}\sqrt{D}|a_{j,D}| \leq R'.
\end{equation}
\item  The  $f_m$'s   have  to  belong  to  $  W_\al(R)$.   We  have   $\si_\la^2=\rho$. Hence for any $t>0$
$$\rho^2\indic_{\rho\leq \sqrt{\rho} t}+D
a_{j,D}^2\indic_{|a_{j,D}|\leq \sqrt{ \rho} t} \leq  R^{2 /(1+2\alpha)} t^{4\al/(1+2\al)}.$$
If $|a_{j,D}|\leq \rho$, then it is enough to have
\begin{equation}\label{cont3}
\rho^2+Da_{j,D}^2 \leq R^{2 /(1+2\alpha)}\rho^{2\al/(1+2\al)} 
\end{equation}
and
\begin{equation}\label{cont4}
Da_{j,D}^2 \leq R^{2 /(1+2\alpha)} \left(\frac{a_{j,D}^2}{\rho}\right)^{2\al/(1+2\al)}. 
\end{equation}
\end{itemize} 
If the parameters satisfy these equations, then 
$${\cal R}(W_\al(R)\cap{\cal B}^{\frac{\al}{1+2\al}}_{2,\infty}(R')\cap {\cal L}_{1,2,\infty}(R''))\geq {\cal R}(\mathcal{M}_{j,D}).$$
Moreover  if  for  any  estimator $\hat{f}$,  we  define  $\hat{f}'=\arg\inf_{g\in
\mathcal{M}_{j,D}} \normp{g-\hat{f}}$, then for $f\in\mathcal{M}_{j,D}$, $$\normp{f-\hat{f}'}\leq \normp{f-\hat{f}}+\normp{\hat{f}-\hat{f}'} \leq 2 \normp{f-\hat{f}}.$$ Hence,
$${\cal R}(\mathcal{M}_{j,D})\geq \frac14 \inf_{\hat{f}\in \mathcal{M}_{j,D}}\sup_{f\in \mathcal{M}_{j,D}}\E(\normp{f-\hat{f}}^2).$$
But for every $m\not = m'$, $\normp{f_m-f_{m'}}^2=\sum_{k\in m \Delta m'} a_{j,D}^2 \geq \theta'D a_{j,D}^2$. Hence,
$${\cal R}(\mathcal{M}_{j,D})\geq\frac14 \theta' D a_{j,D}^2 \inf_{\hat{f}\in \mathcal{M}_{j,D}}(1-\inf_{f\in \mathcal{M}_{j,D}}\P(\hat{f}=f)).$$
We now use Fano's Lemma of \cite{birFano},  and to do so we need to provide an
upper bound of the Kullback-Leibler distance between two points of $ \mathcal{M}_{j,D}$. 
But for every $m\not = m'$,
\begin{eqnarray*}
K(\P_{f_m'}, \P_{f_m}) & =& n\int_\R f_{m'} \left(\exp\left(\ln\frac{f_m}{f_{m'}}\right)-\ln\frac{f_m}{f_{m'}}-1\right) \\
&=& 
n\int_\R \left(f_m-f_{m'} - f_{m'}\ln \left(1+\frac{f_m-f_{m'}}{f_{m'}}\right)\right) \\
&\leq & n\int_\R \frac{(f_m-f_{m'})^2}{f_{m}} \\
&\leq& \frac2\rho n \norm{ f_m-f_{m'}}_2^2 \\
&\leq & 
\frac2\rho n  D a_{j,D}^2,
\end{eqnarray*}
since $\ln(1+x)\geq x/(1+x).$
So finally, following similar arguments  to those used by \cite{ptrfpois} (pages
148 and 149), Fano's lemma implies that there exists an absolute constant $c<1$ such that
$${\cal R}(\mathcal{M}_{j,D})\geq \frac{(1-c)\theta'}{4} D a_{j,D}^2$$
as soon as the mean Kullback Leibler distance is small enough, which is implied by 
\begin{equation}
\label{cont5}
\frac2\rho n  D a_{j,D}^2 \leq c \sigma D \ln(2^j/D).
\end{equation}
Let us take $j$ such that $2^j\leq n/\ln n \leq 2^{j+1}$ and with $D\leq 2^j$,
$$a_{j,D}^2=\frac{ \rho^2}{4n} \ln (2^j/D).$$
First note that (\ref{cont5}) is  automatically fulfilled as soon as $\rho\leq
2c\sigma$, that is true if $\rho$ an absolute constant small enough. Then
$$\rho+2^{j/2}|a_{j,D}| \leq \rho +2^{j/2}\sqrt{\frac{\rho^2\ln n}{4n}} \leq 1.5 \rho.$$
So,  if  $\rho$  is  an  absolute  constant  small  enough,  (\ref{cont0})  is
satisfied. Moreover 
$$2^{1+j/2}|a_{j,D}|\leq 2^{1+j/2} \sqrt{\frac{\rho^2\ln n}{4n}}\leq \rho.$$
This gives (\ref{cont1}).
Now, take an integer $D=D_n$ such that
$$D_n\sim_{n\to\infty} R^{2/(1+2\alpha)} \left(\frac{n}{\ln n}\right)^{1/(1+2\alpha)}.$$
For $n$ large enough, $D_n\leq 2^j$ and $D_n$ is feasible. We have for $R$ fixed,
$$a_{j,D_n}^2 \sim_{n\to\infty} c_\alpha \rho^2 \frac{\ln n}{n},$$
where $c_\al$ is a constant only depending on $\al$. Therefore,
$$
\rho+2^{j\alpha/(1+2\alpha)} \sqrt{D_n}|a_{j,D_n}|=\rho +\sqrt{c_\alpha} \rho  R^{1/(1+2\alpha)} +o_n(1).$$
Since $R^{1/(1+2\alpha)}\leq R'$ it is sufficient to take $\rho$ small enough but constant depending only on $\alpha$ to obtain (\ref{cont2}).
Moreover,
$$D_na_{j,D_n}^2 \sim_{n\to\infty}c_\alpha \rho^2 R^{2/(1+2\alpha)} \left(\frac{\ln n}{n}\right)^{2\alpha/(1+2\alpha)}.$$
Hence (\ref{cont3}) is equivalent to $\rho^2 < R^{2/(1+2\alpha)} \rho^{2\alpha/(1+2\alpha)}$. Since $R\geq 1$, this is true as soon as $\rho <1$.
Finally (\ref{cont4}) is equivalent, when $n$ tends to $+\infty$, to
$$c_\alpha \rho^2 \leq (c_\alpha \rho)^{2\alpha/(1+2\alpha)}.$$
Once again this is true for $\rho$ small enough depending on $\alpha$.
As we can choose $\rho$ not depending on $R,R',R''$, this concludes
the proof.\\ Corollary \ref{adaptivemini} is completely straight
forward once we notice that if $R'\geq R$ then for every $\al$, $R'\geq R^{\frac1{1+2\al}}$.
 
\subsection{Proof of Theorem \ref{lower}}
Let $\al>1$ and $n$ be fixed. We set $j$ a positive integer such that
$$\frac{n}{(\ln n)^\al}\leq 2^j <\frac{2n}{(\ln n)^\al}.$$ 
For all $k\in\{0,...,2^j-1\}$, we define $$N^+_\jk = \int_{k2^{-j}}^{(k+1/2)2^{-j}}dN,\quad N^-_\jk =
\int_{(k+1/2)2^{-j}}^{(k+1)2^{-j}}dN.$$ All these variables are iid random Poisson
variables of parameter $\mu_{n,j} = n2^{-j-1}.$
Moreover,
$$\hb_\jk = \frac{2^{j/2}}{n}(N^+_\jk-N^-_\jk) \mbox{ and }
\hat{V}_{(\jk),n} = \frac{2^{j}}{n^2}(N^+_\jk+N^-_\jk).$$
Hence,
$$\E(\normp{\tilde{f}_{n,\gamma}-f}^2)\geq \sum_{k=0}^{2^j-1}
\frac{2^j}{n^2}\E\left((N^+_\jk-N^-_\jk)^2 \indic_{|N^+_\jk-N^-_\jk|\geq
\sqrt{2\gamma\ln (n) (N^+_\jk+N^-_\jk)}+\ln(n) u_n}\right).$$
Denote by $v_{n,j}= \left(\sqrt{4\gamma\ln(n) \mu_{n,j}}+\ln(n) u_n\right)^2.$
Remark  that if $N^+_\jk=\mu_{n,j} + \frac{\sqrt{v_{n,j}}}{2} $ and $N^-_\jk= \mu_{n,j} -\frac{\sqrt{v_{n,j}}}{2}$, then $$|N^+_\jk-N^-_\jk|= \sqrt{2\gamma\ln(n) (N^+_\jk+N^-_\jk)}+\ln(n) u_n.$$ 
Let $N^+$ and $N^-$ be two independent Poisson variables of parameter
$\mu_{n,j}$. Then,
$$\E(\normp{\tilde{f}_{n,\gamma}-f}^2)\geq 
\frac{2^{2j}}{n^2}v_{n,j}
\P\left(N^+=\mu_{n,j} +\frac{\sqrt{v_{n,j}}}{2}\mbox{ and }N^-= \mu_{n,j} - \frac{\sqrt{v_{n,j}}}{2}\right).
$$
Note that $$\frac{1}{4}(\ln n)^\al<\mu_{n,j} \leq \frac{1}{2}(\ln n)^\al,$$ and $$\lim_{n\to +\infty}\frac{\sqrt{v_{n,j}}}{\mu_{n,j}}=0.$$  So, $l_{n,j}=\mu_{n,j}+\frac{\sqrt{v_{n,j}}}{2}$ and $m_{n,j}=\mu_{n,j}-\frac{\sqrt{v_{n,j}}}{2}$ go to
$+\infty$ with $n$.
Hence by Stirling formula, 
\begin{eqnarray*}
\E(\normp{\tilde{f}_{n,\gamma}-f}^2)&\geq& \frac{v_{n,j}}{(\ln n)^{2\al}}
\P\left(N^+=\mu_{n,j}+\frac{\sqrt{v_{n,j}}}{2}\right)\P\left(N^-=\mu_{n,j}-\frac{\sqrt{v_{n,j}}}{2}\right)\\
&\geq & \frac{v_{n,j}}{(\ln
n)^{2\al}}\frac{\mu_{n,j}^{l_{n,j}}}{l_{n,j}!}e^{-\mu_{n,j}}\frac{\mu_{n,j}^{m_{n,j}}}{m_{n,j}!}e^{-\mu_{n,j}}\\
&\geq &\frac{4\gamma\mu_{n,j}}{(\ln
n)^{2\al-1}}\left(\frac{\mu_{n,j}}{l_{n,j}}\right)^{l_{n,j}}
e^{-(\mu_{n,j}-l_{n,j})}\left(\frac{\mu_{n,j}}{m_{n,j}}\right)^{m_{n,j}}
e^{-(\mu_{n,j}-m_{n,j})}\frac{(1+o_n(1))}{2\pi\sqrt{l_{n,j}m_{n,j}}}\\
&\geq & \frac{2\gamma}{\pi(\ln
n)^{2\al-1}}e^{-\mu_{n,j} \left[h\left(\frac{\sqrt{v_{n,j}} }{2\mu_{n,j}}\right)+h\left(-\frac{\sqrt{v_{n,j}} }{2\mu_{n,j}}\right)\right]}(1+o_n(1))
\end{eqnarray*}
where $h(x)=(1+x)\ln(1+x)-x =x^2/2+O(x^3)$.
So, 
$$\E(\normp{\tilde{f}_{n,\gamma}-f}^2)\geq \frac{2\gamma}{\pi(\ln
n)^{2\al-1}}e^{- \frac{v_{n,j}}{4\mu_{n,j}}+O_n\left(\frac{v_{n,j}^{3/2}}{\mu_{n,j}^2}\right)}(1+o_n(1)).$$
Since
$$v_{n,j}=4\gamma\ln(n)\mu_{n,j}(1+o_n(1)),$$
we obtain
$$\E(\normp{\tilde{f}_{n,\gamma}-f}^2)\geq \frac{2\gamma}{\pi(\ln
n)^{2\al-1}}e^{- \gamma\ln(n)+o_n(\ln(n))}(1+o_n(1)).$$
Finally, for every $\e>0$,
$$\E(\normp{\tilde{f}_{n,\gamma}-f}^2)\geq \frac{1}{n^{\gamma+\e}}(1+o_n(1)).$$
\subsection{Proof of Theorem \ref{classFn}}
We use notations of Lemma \ref{ecra}. Let $f\in {\cal F}_n$. We apply (\ref{init}) with $\e=1.4$. Then, with $\ga=1+\sqrt{2}$, and $\delta>0$ such that $(1+\delta)^2=11.8/(2\ga\times(1+2/\e))\simeq 1.006$, (\ref{init}) becomes
\begin{multline*}
\E(\normp{\tilde{f}_{n,\gamma}-f}^2)\leq \\\inf_{m\subset \Ga_n} \left\{6\sum_{\la\not\in m} \be_\la^2+\sum_{\la\in m} [3.4 + 11.8\ln n] V_{\la,n} + c(\delta,\ga)(1+2\e^{-1})\sum_{\la\in m}\left(\frac{\ln n \norm{\p_\la}_\infty}{n}\right)^2\right\}\\ 
+ \frac{C_2(\gamma,\|f\|_1, c,c', \p)}{n}.
\end{multline*}   Now,   take   $$m=\{\la\in   \Ga_n:\quad
\be_\la^2>V_{\la,n}\}.$$
If $m$ is empty, then $\be_\la^2=\min(\be_\la^2,V_{\la,n})$ for every
$\la$ of $\Ga_n$. Hence
$$ \E(\normp{\tilde{f}_{n,\gamma}-f}^2)\leq 6\sum_{\la\in \Ga_n}
\be_\la^2 + \frac{C_2(\gamma,\|f\|_1, c,c', \p)}{n}.$$
The result is true for $n$ large enough even if the $\be_\la$'s are
all zero and this explains the presence
of $1/n$ in the oracle ratio.\\
If $m$ is not empty,
note $\la=(j,k)$. Since $F_\la  \leq 2^{-j} \norm{f}_\infty$,  if $F_\la\not=0$, then
$2^j=O(n/\ln   n)$   and   $\la    \in\Ga_n$. Since
$$|\be_\la|\leq S_\p2^{j/2}F_\la,$$
this implies that $F_\la$ is non zero for all $\la \in m$, and that if $\be_\la\not=0$ then  $\la\in\Gamma_n$. Now,
$$V_{\la,n} =\frac{1}{n}\sigma_\la^2 \geq \frac{1}{n} 2^{j} I^2_\p F_\la \geq	 \frac{1}{n\Theta_\p} \norm{\p_\la}_\infty^2 F_\la.$$
Hence,  for  all  $n$,  if  $\la\in  m$,  $$V_{\la,n}  \ln  n  \geq  \frac{(\ln  n
)^{2}(\ln\ln n)}{\Theta_\p n^2} \norm{\p_\la}_\infty^2 $$
and if $n$ is large enough, $$0.2\ln n \sum_{\la \in m} V_{\la,n} \geq  c(\delta,\ga)(1+2\e^{-1})\sum_{\la \in m} \left(\frac{\ln n }{n}\right)^2 \norm{\p_\la}_\infty^2+3.4\sum_{\la \in m} V_{\la,n}.$$
\subsection{Proof of Theorem \ref{uppth}}
Before proving Theorem \ref{uppth}, let us state the following result. 
\begin{Prop}\label{gagrand}
Let  $\ga_{\min}\in(1,\gamma)$  be  fixed  and  let $\eta_{\la, \ga_{\min}}$  be  the  threshold
associated with $\ga_{\min}$:
 $$\eta_{\la, \ga_{\min}}=\sqrt{2\gamma_{\min}\ln n \tilde{V}_{\la,n} }+\frac{\gamma_{\min}\ln n}{3n}\norm{\p_\la}_\infty,$$
where    $$\tilde{V}_{\la,n}=\hat{V}_{\la,n}+\sqrt{2\gamma_{\min}     \ln    n    \hat{V}_{\la,n}
\frac{\norm{\p_\la}_\infty^2}{n^2}}+3                \gamma_{\min}                \ln
n\frac{\norm{\p_\la}_\infty^2}{n^2}$$
(see Theorem  \ref{inegoracle}). Let $u=(u_n)_n$ be some  sequence of positive
numbers and$$\La_u=\{\la \mbox{ such that }\P(\eta_{\la,\gamma} >|\be_\la|+\eta_{\la,\gamma_{\min}})\geq 1 -
u_n\}.$$
Then
$$\E(\normp{\tilde{f}_{n,\gamma}-f}^2)\geq \left(\sum_{\la \in
\La_u}\be_\la^2\right) (1-(3n^{-\ga_{\min}}+u_n)).$$
\end{Prop}
\begin{proof}
\begin{eqnarray*}
\E(\normp{\tilde{f}_{n,\gamma}-f}^2) &\geq & \sum_{\la \in
\La_u} \E\left((\hb_\la-\be_\la)^2\indic_{|\hb_\la|\geq \eta_{\la,\gamma}}+\be_\la^2 \indic_{|\hb_\la|<\eta_{\la,\gamma}}\right).\\
&\geq &\sum_{\la \in
\La_u}\be_\la^2 \P( |\hb_\la|<\eta_{\la,\gamma})\\
&\geq & \sum_{\la \in
\La_u}\be_\la^2 \P(|\hb_\la-\be_\la|+|\be_\la|<\eta_{\la,\gamma})\\
&\geq & \sum_{\la \in
\La_u}\be_\la^2 \P(|\hb_\la-\be_\la|<\eta_{\la,\gamma_{\min}} \mbox{ and }
\eta_{\la,\gamma_{\min}}+|\be_\la|<\eta_{\la,\gamma})\\
&\geq & \sum_{\la \in
\La_u}\be_\la^2
\left(1-\left(\P(|\hb_\la-\be_\la|\geq\eta_{\la,\gamma_{\min}})+\P(\eta_{\la,\gamma_{\min}}+|\be_\la|\geq\eta_{\la,\gamma})\right)\right)\\
&\geq &\left(\sum_{\la \in
\La_u}\be_\la^2\right) (1-(3n^{-\ga_{\min}}+u_n)),
\end{eqnarray*}
by applying Lemma \ref{toutesdev}. 
\end{proof}
Using this proposition, we give the proof of Theorem \ref{uppth}.
Let us consider 
$$f=\indic_{[0,1]} +\sum_{k\in\mathcal{N}_j}\sqrt{\frac{2(\sqrt{\ga}-\sqrt{\ga_{\min}})^2\ln n}{n}} \tp_{\jk},$$
with $$ \mathcal{N}_j=\{0,1,\dots,2^j-1\}$$
and   $$\frac{n}{(\ln  n)^{1+\al}}<2^j\leq   \frac{2n}{(\ln  n)^{1+\al}},\quad
\al>0.$$
Note that for any $(j,k)$, if $F_{j,k}\not=0$, then $F_{j,k}=2^{-j}\geq \frac{(\ln n)(\ln\ln
n)}{n}$ for $n$ large enough and $f$ belongs to ${\cal F}_n$.  Furthermore,
$V_{(-1,0),n}=\frac{1}{n}$ and for any
$k\in \mathcal{N}_j$, $V_{(j,k),n}=\frac{1}{n}$. So, for $n$ large enough,
$$\sum_{\la \in \Ga_n}\min( \beta_\la^2, V_{\la,n})=V_{(-1,0),n}+\sum_{k\in\mathcal{N}_j}V_{(j,k),n}=\frac{1}{n}+\sum_{k\in\mathcal{N}_j}\frac{1}{n}.$$
Now,  to   apply  Proposition   \ref{gagrand},  let  us   set  for   any  $n$,
$u_n=n^{-\ga}$ and observe that for any $\e>0$,
$$
\P(\eta_{\la,\gamma_{\min}}+|\be_\la|\geq \eta_{\la,\gamma})\leq \P((1+\e)2\ga_{\min}\ln
n \tilde{V}_{\la,n}(\ga_{\min}) + (1+\e^{-1}) \be_\la^2 > 2\ga\ln n
\tilde{V}_{\la,n}(\ga)),
$$
since $\ga_{\min}<\ga.$
With $\e = \sqrt{\ga/\ga_{\min}}-1$ and $\theta=\sqrt{\ga_{\min}/\ga}$, 
\begin{multline*}\P((1+\e)2\ga_{\min}\ln
n \tilde{V}_{\la,n}(\ga_{\min}) + (1+\e^{-1}) \be_\la^2 > 2\ga\ln n
\tilde{V}_{\la,n}(\ga))= \\\P(\theta \tilde{V}_{\la,n}(\ga_{\min})+(1-\theta) V_{\la,n} > \tilde{V}_{\la,n}(\ga)).\end{multline*}
Since $ \tilde{V}_{\la,n}(\ga_{\min})<\tilde{V}_{\la,n}(\ga)$, 
$$\P(\eta_{\la,\gamma_{\min}}+|\be_\la|\geq \eta_{\la,\gamma})\leq \P(V_{\la,n} > \tilde{V}_{\la,n}(\ga)) \leq u_n.$$
So, 
$$\{(j,k):\quad k\in\mathcal{N}_j\}\subset \La_u,$$
and 
\begin{eqnarray*}
\E(\normp{\tilde{f}_{n,\gamma}-f}^2)&\geq&\sum_{
k\in\mathcal{N}_j}\be_{j,k}^2(1-(3n^{-\ga_{\min}}+n^{-\ga}))\\
&\geq&(\sqrt{\ga}-\sqrt{\ga_{\min}})^2                                      2\ln
n\sum_{k\in\mathcal{N}_j}\frac{1}{n}(1-(3n^{-\ga_{\min}}+n^{-\ga}))\\
&\geq&(\sqrt{\ga}-\sqrt{\ga_{\min}})^2 2\ln n\left(\sum_{\la \in \Ga_n}\min( \beta_\la^2, V_{\la,n})-\frac{1}{n}\right)(1-(3n^{-\ga_{\min}}+n^{-\ga})).
\end{eqnarray*}
Finally, since $\mbox{card}(\mathcal{N}_j)\to +\infty$ when $n\to +\infty$,
$$\frac{\E(\normp{\tilde{f}_{n,\gamma}-f}^2)}{\sum_{\la \in \Ga_n}
\min( \beta_\la^2, V_{\la,n})+\frac{1}{n}}\geq (\sqrt{\ga}-\sqrt{\ga_{\min}})^2 2\ln n(1+o_n(1)).$$
\hfill
  $\blacksquare$

\noindent{\bf\Large Appendix}

\noindent The following table gives the definition of the signals used in Section \ref{simu}.

\vspace{1cm}
{\scriptsize
\hspace{-0.8cm}
\begin{tabular}{|c|c|c|}
\hline 
&&\\
Haar1 & Haar2 &  Blocks \\ 
&&\\
 $\displaystyle{\bf 1}_{[0,1]}$ & $\displaystyle  1.5~{\bf 1}_{[0,0.125]}+0.5~{\bf 1}_{[0.125,0.25]}+{\bf 1}_{[0.25,1]}$ & $\displaystyle\left(2+\sum_j \frac{h_j}2\left(1+\mbox{sgn}(x-p_j)\right)\right)\frac{{\bf 1}_{[0,1]}}{3.551}$ \\ 
&&\\
\hline
&&\\
Comb  & Gauss1 & Gauss2  \\ 
&&\\
$\displaystyle 32\sum_{k=1}^{+\infty}\frac1{k2^k}{\bf
  1}_{[k^2/32,(k^2+k)/32]}$& $\displaystyle\frac{1}{0.25\sqrt{2\pi}}
\exp\left(\frac{(x-0.5)^2}{2\times 0.25^2}\right)$ &$\displaystyle
\frac{1}{\sqrt{2\pi}}
\exp\left({\frac{(x-0.5)^2}{2\times 0.25^2}}\right)+\frac{3}{\sqrt{2\pi}}
\exp\left({\frac{(x-5)^2}{2\times 0.25^2}}\right)$\\
&&\\
\hline 
&&\\
Beta0.5 & Beta4 & Bumps\\
&&\\
$\displaystyle0.5x^{-0.5}{\bf 1}_{]0,1]}$ & $\displaystyle 3x^4{\bf 1}_{[1,+\infty[}$ & $\displaystyle\left(\sum_jg_j\left(1+\frac{|x-p_j|}{w_j}\right)^{-4}\right)\frac{{\bf 1}_{[0,1]}}{0.284}$\\
&&\\
\hline
\end{tabular}}
\vspace{0.3cm}

where

$ 
\begin{tabular}{ccccccccccccccc}
  p & =  & [ & 0.1 & 0.13 & 0.15 &0.23 &0.25& 0.4  & 0.44 & 0.65 & 0.76 &0.78  & 0.81 &  ]  \\ 
  h & =  & [  & 4 & -5 &3  & -4 & 5 &-4.2  &2.1  & 4.3 & -3.1 & 2.1 & -4.2 & ]  \\ 
 g  & = & [ & 4  & 5 & 3 & 4 & 5 & 4.2 & 2.1 & 4.3  &3.1  & 5.1  & 4.2  & ] \\ 
  w&=  &[  &0.005  & 0.005 &0.006  &0.01  &0.01  &0.03  &0.01  &0.01  &0.005  &0.008  &0.005  & ] 
\end{tabular} 
$
\\\\

\noindent{\bf Acknowledgment.} The authors  acknowledge the support of the  French Agence Nationale
de la Recherche  (ANR), under grant ATLAS (JCJC06\_137446) ''From Applications
to Theory in Learning and Adaptive Statistics''. We would like to warmly thank Lucien Birg\'e
for his advises and his encouragements. 
\bibliographystyle{plain}

\begin{thebibliography}{10}
\bibitem{aut}
Autin, F. {\it Maxiset for density estimation on $\mathbb R$},
Math. Methods Statist. {\bf 15}(2), 123--145, (2006).
\bibitem{abs}  Antoniadis,  A.,  Besbeas,   P.,  Sapatinas,  E.  {\it  Wavelet
shrinkage  for natural  exponential families  with cubic  variance functions},
Sankhya {\bf 63}, 309--327, (2001). 
\bibitem{as}  Antoniadis,  A.,  Sapatinas,  T.  {\it  Wavelet
shrinkage for natural exponential families with quadratic variance functions},
Biometrika {\bf 88}(3), 805--820, (2001).
\bibitem{bb} Baraud, Y., Birg\'e L. {\it Estimating the intensity of a
random measure by histogram type estimators}, 2006, manuscript.
\bibitem{birFano} Birg\'e, L. {\it A new look at an old result: Fano's Lemma},
2001, manuscript. 
\bibitem{bir} Birg\'e, L. {\it Model selection for Poisson processes},
2006, manuscript. 
\bibitem{bm} Birg\'e, L., Massart P. {\it Minimal penalties for
Gaussian model selection},  Probab. Theory Related Fields,{\bf138}(1-2),33--73  (2007).
\bibitem{bh} Bretagnolle,  J., Huber, C. {\it Estimation  des densit\'es: risque
minimax}, Z. Wahrsch. Verw. Gebiete {\bf 47}(2), 119--137, (1979). 
\bibitem{btw}  Bunea F.,  Tsybakov, A.B.,  Wegkamp, M.H.  {\it  Sparse density
estimation with $l_1$ penalties}, 2007, manuscript. 
\bibitem{ck} Cavalier, L., Koo, J.Y. 
{\it Poisson intensity estimation for tomographic data using a wavelet shrinkage approach}, 
IEEE Trans. Inform. Theory {\bf 48}(10), 2794--2802, (2002). 
 \bibitem{cdf} Cohen,  A., Daubechies,  I., Feauveau, J.C.  {\it Biorthogonal
bases of  compactly supported wavelets},  Comm. Pure Appl. Math.   {\bf 45}(5),
485--560, (1992). 
\bibitem{dj}  Delyon, B.,  Juditsky, A.  {\it  On the  computation of  wavelet
coefficients}, J. Approx. Theory {\bf 88}(1), 47--79, (1997). 
\bibitem{dl} DeVore, R.A., Lorentz, G.G. {\it Constructive approximation}, Springer-Verlag, Berlin, 1993. 
\bibitem{don2}  Donoho, D.L. {\it  Nonlinear wavelet  methods for  recovery of
signals, densities, and spectra from indirect and noisy data}, 
 Different perspectives on wavelets (San Antonio, TX, 1993),  173--205,
 Proc. Sympos. Appl. Math., {\bf 47}, Amer. Math. Soc., Providence, RI, (1993).
\bibitem{don}  Donoho, D.L.  {\it  Smooth wavelet  decompositions with  blocky
coefficient   kernels},   Recent  advances   in   wavelet  analysis,   Wavelet
Anal. Appl., {\bf 3}, Academic Press, Boston, MA, 259--308, (1994). 
\bibitem{dojo}  Donoho, D.L., Johnstone,  I.M. {\it  Ideal spatial  adaptation by
wavelet shrinkage}, Biometrika, {\bf 81}(3), 425--455, (1994). 
\bibitem{djkp}  Donoho,   D.L.,  Johnstone,  I.M.,   Kerkyacharian  G.,  Picard
D.  {\it Density estimation by wavelet thresholding}, Annals of Statistics,
{\bf 24}(2), 508--539, (1996). 
\bibitem{gol} Golubev, G.K. {\it  Nonparametric estimation of smooth densities
of a  distribution in  $L\sb 2$}, Problems  Inform. Transmission  {\bf 28}(1),
44--54, (1992). 
\bibitem{gs} Gusto, G., Schbath, S.  {\it FADO: a statistical method to detect
favored or avoided distances between motif occurrences using the hawkes model}, Statistical Applications in Genetics and Molecular Biology, {\bf 4}(1), (2005). 
\bibitem{hardle} H\"ardle, W., Kerkyacharian, G., Picard, D.,
Tsybakov, A. {\it Wavelets, approximation and statistical
applications}, Lecture Notes in Statistics, {\bf 129},
Springer-Verlag, New York, 1998.
\bibitem{ik}  Ibragimov, I.A.,  Kas'minskij, R.Z.  {\it On  the
estimation of  a signal,  its derivatives and  the maximum point  for Gaussian
observations}, Teor. Veroyatnost. i Primenen.  {\bf 25}(4), 718--733, (1980). 
\bibitem{johnson} Johnson, W.B. {\it Best Constants in Moment
Inequalities for Linear Combinations of Independent and Exchangeable
Random Variables}, Annals of probability {\bf 13}(1), 234--253, (1985).
\bibitem{joh} Johnstone, I.M. {\it Minimax Bayes, asymptotic minimax and sparse wavelet priors}.  Statistical decision theory and related topics, V (West Lafayette, IN, 1992), 303--326, Springer, New York, 1994. 
\bibitem{jll}  Juditsky,  A.,  Lambert-Lacroix  S. {\it  On  minimax  density
estimation on $\R$}, Bernoulli {\bf 10}(2), 187--220, (2004). 
\bibitem{kk}  Kim,  W.C.,  Koo,  J.Y.  {\it  Inhomogeneous  Poisson  intensity
estimation  via information  projections  onto wavelet  subspaces}, J.  Korean
Statist. Soc. {\bf 31}(3), 343--357, (2002). 
\bibitem{kin}  Kingman,  J.F.C. {\it  Poisson  processes.}  Oxford studies  in
Probability, 1993. 
\bibitem{kol}  Kolaczyk, E.D.  {\it  Wavelet shrinkage  estimation of  certain
Poisson intensity  signals using  corrected thresholds}, Statist.  Sinica {\bf
9}(1), 119--135, (1999). 
\bibitem{kn} Kolaczyk, E.D., Nowak, R.D. {\it Multiscale likelihood analysis and complexity penalized estimation},
Ann. Statist. {\bf 32}(2), 500--527, (2004).
\bibitem{kut} Kutoyants,  Y.A. {\it Statistical inference  for spatial Poisson
processes.}  Lecture Notes in Statistics, {\bf 134}. Spinger Edition. 1998. 
\bibitem{leb} Lebarbier, E.  {\it Detecting multiple change-points in the mean of Gaussian process by model selection}, Signal Processing {\bf 85}(4), 717--736, (2005).
\bibitem{ptrfpois}  Reynaud-Bouret, P.  {\it  Adaptive  estimation  of  the
intensity    of   inhomogeneous    Poisson    processes   via    concentration
inequalities}, Probability Theory and Related Fields {\bf 126}(1),
103--153, (2003). 
\bibitem{rbr} Reynaud-Bouret, P., Roy, E.{\it Some non asymptotic tail estimates for Hawkes processes},  Bulletin of the Belgian Mathematical Society-Simon Stevin, { \bf13}(5), 883--896 (2007), Proceedings of the 2005 joint BeNeLuxFra conference in Mathematics.
\bibitem{rivbernoulli} Rivoirard, V. {\it Nonlinear estimation over weak Besov
spaces and minimax Bayes method}, Bernoulli {\bf 12}(4), 609--632, (2006). 
\bibitem{rud}  Rudemo,  M. {\it  Empirical  choice  of  histograms and  kernel
density estimators}, Scand. J. Statist. {\bf 9}(2), 65--78, (1982).  
\end{thebibliography}

\end{document}